\documentclass{amsart}
\usepackage{geometry}
\usepackage{multicol}
\usepackage{fullpage,graphicx,subfigure,mathpazo,color}
\usepackage{amsthm,amsmath,amscd,tikz,mathrsfs,amssymb,enumitem}
\usepackage[normalem]{ulem}
\usepackage{setspace}
\usepackage{graphicx}
\usepackage{float}
\usepackage{indentfirst}
\usepackage{hyperref}
\usetikzlibrary{arrows.meta,decorations.pathmorphing,decorations.markings,positioning}  
\numberwithin{equation}{section}
\newcommand{\ii}{\mathrm{i}}
\newcommand{\ee}{\mathrm{e}}
\newcommand{\dd}{\mathrm{d}}
\newcommand{\oo}{\mathcal{O}}

\newcommand{\M}{\mathbf{M}}
\newcommand{\V}{\mathbf{V}}
\newcommand{\E}{\mathbf{E}}

\newtheorem{rhp}{Riemann-Hilbert Problem}[section]
\newtheorem{theorem}{Theorem}
\newtheorem{lemma}{Lemma}
\newtheorem{prop}{Proposition}[section]

\newtheorem{definition}{Definition}

\newtheorem{remark}{Remark}
\newtheorem{assumption}{Assumption}[section]

\usepackage{graphicx}      
\usepackage{titlesec}

\titleformat{\section}{\centering\LARGE\bfseries}{\thesection}{1em}{}
\titleformat{\subsection}{\Large\bfseries}{\thesubsection}{1em}{}
\numberwithin{figure}{section}

\begin{document}
	\title{Asymptotic Stability of Multi-Solitons for Coupled Nonlinear Schr\"odinger Equations via the $\bar{\partial}$-Method}
	\author{Yubin Huang}
	\address{School of Mathematics, South China University of Technology, Guangzhou, China, 510641}
	\email{202410188295@mail.scut.edu.cn}
	
	\author{Liming Ling}
	\address{School of Mathematics, South China University of Technology, Guangzhou, China, 510641}
	\email{linglm@scut.edu.cn}
	
  \author{Huajie Su}
	\address{School of Mathematics, South China University of Technology, Guangzhou, China, 510641}
	\email{mashj@mail.scut.edu.cn}
	
	\begin{abstract}
    The Riemann-Hilbert problem for the focusing coupled nonlinear Schr\"odinger (CNLS) equation is 
    formulated on the basis of the corresponding $3\times3$ matrix spectral problem. 
    We remove the discrete spectrum of initial RHP with
    the aid of Darboux transformations. Based on the $\bar{\partial}$-steepest descent method, 
    we establish the long-time asymptotic behavior of solutions to the CNLS equation
    for initial condition in the weighted Sobolev space. 
    Compared to the improved nonlinear steepest descent method, 
    we improve the error estimate up to order $\mathcal{O}\left(t^{-3/4}\right)$. 
    Furthermore, we obtain the asymptotic stability of multi-soliton 
    solutions for CNLS equation and 
    analyze the law of multi-soliton collision in the view-point of 
    Yang-Baxter map.  
	\end{abstract}
	
	\maketitle
	\section{Introduction}
  In this paper, we study the asymptotic stability of multi-solitons for the focusing coupled 
  nonlinear Schr\"odinger 
  (CNLS) equation on $\mathbb{R}\times\mathbb{R}$:
  \begin{equation}\label{CNLS}
    \begin{cases}
      \ii\mathbf{q}_{t}+\frac{1}{2}\mathbf{q}_{xx}+
      \mathbf{q}\mathbf{q}^{\dagger}\mathbf{q}=0,
    \quad\mathbf{q}(x,t)=\begin{bmatrix}
      q_{1}(x,t),q_{2}(x,t)
    \end{bmatrix},\\
    \mathbf{q}_0(x)=\begin{bmatrix}
      q_{1}(x,0),q_{2}(x,0)
    \end{bmatrix},
    \end{cases}
  \end{equation}
  \noindent
  where the superscript $\dagger$ denotes 
  Hermitian conjugate and the initial data $\left\{q_{1}(x,0),q_{2}(x,0)\right\}$ belong to the weighted 
  Sobolev space 
  \begin{equation}
    H^{1,1}(\mathbb{R})=\{f(x)\in L^2(\mathbb{R}):f^{\prime}(x),xf(x)\in L^2(\mathbb{R})\}.
  \end{equation}
  The CNLS equation is a multi-component extension of the classical nonlinear Schr\"odinger (NLS) equation.
  It plays a pivotal role in nonlinear mathematical physics owing to a 
  rich range of dynamics, 
  including bright-bright solitons, bright-dark solitons, vector rogue waves, higher-order localized waves, 
  breathers, and multi-hump solitons \cite{Guo2011, Mu2015, lingzhang2024, DING2020109580, SONG2020105046}. 
  The stability analysis of these solutions to the CNLS equation is an important problem in 
  mathematical physics. Numerous studies have focused on their spectral stability 
  \cite{mesentsev1992stability, li_structural_1998, pelinovsky_inertia_2005}, orbital 
  stability \cite{LingSu2026orbitalstability,ling2025nonlinear}, and instability \cite{li_mechanism_2000, pelinovsky_instabilities_2005}. 
 
  In the study of long-time asymptotics for integrable systems, Deift and Zhou pioneered 
  the nonlinear steepest descent (NSD) method for the corresponding Riemann-Hilbert problems (RHPs) derived 
  from the inverse scattering transform \cite{deift1993steepest}. The key idea of the nonlinear 
  steepest descent method is to deform the jump contour from the real axis to a new contour along which 
  the oscillatory factors associated with the phase function are transformed into exponentially decaying factors. 
  However, it is not possible to directly perform analytic extension off the real axis for the initial 
  jump due to the presence of the non-analytic reflection coefficient in general. To address this 
  issue, they utilized the Taylor formula to split the reflection coefficient into a rational part and 
  a small but non-analytic part which contributes mainly to the error term $\oo(t^{-1}\log t)$ 
  in the asymptotic expansion of the recovered solution. Thus, within the framework of this classical nonlinear 
  steepest descent, most studies for integrable systems are based on the assumption that 
  the initial data belongs to the Schwartz space, such as the NLS equation \cite{deift1994long,deift1994casestudy}, 
  the CNLS equation \cite{Geng_Liu_2017,Geng2022CNLS}, the derivative nonlinear Schr\"odinger (DNLS) equation 
  \cite{GengLiu2024DNLS}, the short pulse equation 
  \cite{Geng2024short-pulse, Geng2024short-pulse2} and so on.

  Lately, Zhou has made significant efforts to relax the regularity conditions of the initial data from the 
  Schwartz space to the weighted Sobolev space \cite{Zhou1998}. In \cite{deift2002long,deift2011long}, 
  the improved nonlinear steepest descent (INSD) method was introduced to study the long-time asymptotics for the 
  defocusing and focusing NLS equation in the weighted Sobolev space $H^{1,1}(\mathbb{R})$. In this improved 
  method, the reflection coefficient $r(z)$, which is non-analytic and lacks enough smoothness, was replaced by 
  a rational function $[r](z)$ for contour deformation. However, this rational approximation of the reflection coefficient would 
  lead to an error term $\oo(t^{-(1/2+\kappa)})$ for any $0<\kappa<1/4$. Dieng and McLaughlin developed the 
  $\bar{\partial}$-steepest descent method as a variant of the classical NSD method \cite{McLaughlin_2018}. Rather than 
  rational approximation, they introduced the non-analytic extensions to deform the jump contour off the real axis. 
  This leads to a hybrid $\bar{\partial}$-RHP, which can be decomposed into a solvable model RHP and a pure 
  $\bar{\partial}$-problem. Moreover, the new pure $\bar{\partial}$-problem is equivalent to an integral equation and 
  can be solved through Neumann series. Compared to the INSD method, it avoids delicate
  estimates involving complicated $L^p$ estimates of Cauchy projection operators and achieves a sharp result that 
  the error term can be controlled at the order $\oo(t^{-3/4})$ in 
  the framework of $\bar{\partial}$-steepest descent method. 
  This method has shown its advantage in the wide application to study the long-time asymptotic behavior of 
  integrable nonlinear evolution equations such as the modified Korteweg-de Vries (mKdV) equation \cite{Fan2023mKdV,Fan2023nonload-mKdV,chen2021soliton}, 
  the Novikov equation \cite{Fan2023Novikov}, the Hunter-Saxton equation \cite{Fan2024Hunter-Saxton}, 
  DNLS equation \cite{Fan2022DNLS}, the defocusing Ablowitz-Ladik system \cite{Fanengui2024Ablowitz-Ladik}, 
  the Sasa-Satsuma equation \cite{Fan2022Sasa-Satsuma}, 
  the Camassa-Holm equation \cite{Fan2024Camassa-Holm}, the $n$-component 
  focusing NLS equation \cite{yanzhengya2025nNLS} 
  and so on. 
  
  The $\bar{\partial}$-steepest descent method has been also applied to explore the asymptotic 
  stability of the soliton solution for NLS equation. Scipio and Dmitry obtained 
  the asymptotic stability of the single-soliton solution for the cubic NLS equation 
  \cite{stability-one-soliton-nls}. Later in 2017, 
  Aaron established the asymptotic stability of the $n$-soliton solution for the cubic NLS equation 
  \cite{stability-n-soliton-nls}. Moreover, 
  the work for asymptotic stability extend to NLS equation underlying nonzero background. 
  Scipio and Robert studied the asymptotic stability of the $n$-soliton solution
  for the defocusing NLS equation with the finite density type initial data \cite{stability-n-soliton--nls-nonzero}. 
  
  For the $n$-soliton solution, it can be described as a sum of $n$ single-soliton solution as $t\to\pm\infty$. 
  And it is well known that the scalar soliton interactions influence each other only 
  by a phase shift determined by the conserved amplitudes and velocities. 
  This implies that sequential collisions have additive phase shifts and 
  the first collision does not influence the second collision. In contrast, the collision between 
  the vector solitons associated with the CNLS equation is more complex because of  
  internal degrees of freedom. Although total energy per soliton is conserved (elastic collision), 
  there can be a significant redistribution of energy among the components. 
  This redistribution of energy for vector soliton collision follows directly from Manakov formulae 
  \cite{manakov1974theory} and has been confirmed in experiments \cite{Energy-exchange-interactions}.
  Although the effect of an $n$-soliton collision for CNLS equation can never 
  be written as the algebraic sum of the effects of pair collisions, which is quite different
  from that for NLS equation, it admits a factorization as a nonlinear superposition of 
  $\binom{n}{2}$ pair collisions, independently of the order in which the pair collisions 
  take place. This factorization was first established by Tsuchida via an a posteriori 
  derivation from the explicit $n$-soliton solution \cite{Nsoliton_collision}. In 2014, 
  Caudrelier and Zhang provided a more 
  concise derivation of the factorization property of $n$-soliton collisions 
  by exploiting properties of the dressing factor \cite{Yang-Baxter-reflection}.
  Moreover, the map that describes the interaction of two solitons 
  satisfies the Yang-Baxter relation, thereby providing a new “set-theoretical” solution to the 
  quantum Yang-Baxter equation. The Yang-Baxter equation (YBE), along with its associated R-matrix, 
  constitutes a cornerstone in the theory of quantum integrable systems. It was first introduced by 
  Yang in the context of multi-particle systems with delta-function interactions \cite{yang1967some}. 
  Subsequent research on the YBE gave rise to powerful analytical methods such as Baxter Q-operator and 
  motivated the algebraic framework of the quantum inverse scattering method \cite{korepin1997quantum}, 
  which ultimately led to the invention of quantum groups \cite{drinfel1985hopf}.

  A natural question arises: is it possible to study the asymptotic behavior of the focusing CNLS equation with 
  the initial data $\mathbf{q}_0(x)$ lying in the weighted Sobolev space 
  using the $\bar{\partial}$-steepest descent method? Furthermore, 
  within this framework, a key concern is whether it is feasible to 
  establish the asymptotic stability of the multi-soliton solution 
  for the CNLS equation and, from the perspective of Yang-Baxter map, 
  investigate the collision law for the pure $n$-soliton, even for the generalized $n$-soliton 
  (i.e. the related reflection coefficient does not vanish)? 
  Among these questions, the most immediate one concerns the 
  accuracy of the long-time asymptotic expansion. 
  In our previous work \cite{HYBLLMZXE2025}, we applied the INSD method to 
  obtain the long-time asymptotics for the focusing CNLS equation in 
  the weighted Sobolev space $H^{1,2}(\mathbb{R})$. The error term in 
  the long-time asymptotic expansion can be maintained at the order 
  $\oo(t^{-(1/2+\kappa)})$ for any $0<\kappa<1/4$, which is 
  the same as that for the scalar case \cite{deift2011long}. 
  Hence, it is natural to ask whether the $\bar{\partial}$-steepest descent 
  method can improve the error estimate 
  to order $\mathcal{O}\left(t^{-3/4}\right)$ 
  in the weighted Sobolev space $H^{1,1}(\mathbb{R})$, 
  as has been achieved for the focusing NLS equation \cite{McLaughlin_2018}.

  In this paper, we apply the $\bar{\partial}$-steepest descent method to the 
  inverse scattering transform for the focusing CNLS equation and then analyze the long-time 
  asymptotics of the solution to equation \eqref{CNLS}. 
  We impose an assumption that initial 
  data $\mathbf{q}_0(x) \in H^{1,1}(\mathbb{R})$ and then the corresponding scattering 
  coefficient would also lie in $H^{1,1}(\mathbb{R})$, 
  which is the same as that for the scalar case \cite{Liu_2019}. 
  Moreover, we also assume that 
  $\bar{a}(\lambda)$ only has finite simple zeros which implies that all discrete eigenvalues are simple. 
  In $\bar{\partial}$-steepest descent method, the crucial step 
  is introducing a non-analytic extension to   
  deform the jump matrices away from the real axis, yielding 
  to a $\bar{\partial}$-RHP which could 
  be decomposed into a solvable model RHP and a pure $\bar{\partial}$ problem. 
  If we construct the non-analytic extension in a manner similar to that for the focusing NLS equation \cite{McLaughlin_2018}, 
  it becomes challenging to preserve the desirable properties \eqref{condition-R} 
  for the $\bar{\partial}$ derivative 
  of the non-analytic extension. 
  This challenge arises due to the matrix function $\pmb{\delta}(\lambda)$ 
  which lacks a closed form in the context of the CNLS 
  equation \cite{Geng_Liu_2017}. 
  The desirable properties are essential to ensure that 
  the error term arising from the subsequent pure 
  $\bar{\partial}$ problem can be controlled at $\oo(t^{-3/4})$. 
  To address this issue, we introduce a new form of non-analytic 
  extension with adequate properties on $\bar{\partial}$ derivative, 
  albeit at the cost of not being able to obtain a solvable model RHP directly 
  after dropping the $\bar{\partial}$ component for the $\bar{\partial}$-RHP. 
  Hence, we need to 
  utilize the estimates for $\pmb{\delta}(\lambda)$ in 
  Lemma \ref{estimate-R-delta} to demonstrate that the model RHP \ref{solvable-model} is a sufficiently accurate 
  approximation to RHP \ref{rhp-without-dbar} which results from 
  the $\bar{\partial}$-RHP \ref{dbar-rhp} by dropping the $\bar{\partial}$ component. 
  Subsequently, we obtain the long time asymptotic expansion of 
  the solution for the CNLS \eqref{CNLS} and establish the asymptotic stability of the multi-soliton solution
  as $t\to\pm\infty$. We show that, for the pure $n$-soliton solution(i.e.reflection coefficient $\mathbf{R}(\lambda)= 0$), 
  an $n$-soliton collision can indeed be factorized into a nonlinear superposition of 
  $\binom{n}{2}$ pair collisions with arbitrary order. When reflection coefficient satisfies 
  $\mathbf{R}(\lambda)\not\equiv 0$, the long time asymptotics of $\mathbf{q}(x,t)$ can 
  be still described as a sum of $n$ single-soliton solution as $t\to\pm\infty$. However, 
  in this case, the $n$-soliton collision 
  does not admit the same factorization property as that for the pure $n$-soliton solution.

  \begin{remark}
    Our previous work required the scattering 
    coefficients to be in $H^{2}(\mathbb{R})$
    for the estimation of the matrix 
    $\pmb{\delta}(\lambda)$ (see \cite[Lemma 6]{HYBLLMZXE2025}). 
    In this paper, we refine the proof of the estimate 
    for $\pmb{\delta}(\lambda)$ in Lemma 
    \ref{estimate-R-delta}, thereby reducing the 
    required regularity of the scattering coefficients to 
    $H^{1}(\mathbb{R})$. This improvement allows 
    us to relax the initial data to 
    the weighted Sobolev space $H^{1,1}(\mathbb{R})$. 
  \end{remark}

The main results of this paper are stated in the following theorem:

\begin{theorem}\label{main-result-dbar-CNLS}
  Let $\mathbf{q}(x,t)$ be the solution to the initial-value problem \eqref{CNLS} corresponding to 
  initial data $\mathbf{q}_0(x)=\mathbf{q}(x,0)$.   
  Under the conditions outlined in Assumption \ref{assumption-q0} for $\mathbf{q}_0(x)$, 
  the long-time asymptotic expansion of $\mathbf{q}(x,t)$ is given by:
  \begin{equation}\label{recover-formula-dbar}
  \mathbf{q}(x,t)
  =
  -2\frac{\begin{bmatrix}
  \det \mathbf{Y}_1, \det \mathbf{Y}_2
  \end{bmatrix}}{\det \mathbf{G}} - \mathrm{i}\eta^2 \pmb{\beta}_{12} t^{-1/2}
  + \mathcal{O}(t^{-3/4}),
  \end{equation}
  where the scalar function $\eta$ is defined in \eqref{scalar-eta}, 
  the vector-valued function $\pmb{\beta}_{12}$ is given by \eqref{residue-M6}, and
  \begin{equation*}
  \begin{aligned}
    \mathbf{X} &= [\pmb{\varphi}_1, \dots, \pmb{\varphi}_n],\quad    
    \mathbf{Y}_1 = \begin{bmatrix}
      \mathbf{G} & -\mathbf{X}_2^\dagger \\
      \mathbf{X}_1 & 0
    \end{bmatrix},\quad  
    \mathbf{Y}_2 = \begin{bmatrix}
      \mathbf{G} & -\mathbf{X}_3^\dagger \\
      \mathbf{X}_1 & 0
    \end{bmatrix},\quad  
    \mathbf{G} = \left(\frac{\pmb{\varphi}_j^\dagger \pmb{\varphi}_i}{\lambda_i-\lambda_j^*}\right)_{1\le j,i\le n}, \\
    \mathbf{v}_{i} &= \begin{bmatrix}
      \prod_{j=i+1}^{n} \frac{\lambda_i-\lambda_j^*}{\lambda_i-\lambda_j} \\
      \frac{\mathbf{c}_i^{\top}}{\lambda_i-\lambda_i^*} \prod_{j=1}^{i-1} \frac{\lambda_i-\lambda_j}{\lambda_i-\lambda_j^*}
    \end{bmatrix},\quad
    \pmb{\varphi}_{i} = \left(\mathbb{I}_3 + t^{-1/2}\tilde{\mathbf{P}}_i\right) \pmb{\Delta}(\lambda_i) \mathrm{e}^{-\mathrm{i} \lambda_{i}(x + \lambda_{i} t) \pmb{\Lambda}_{3}} \mathbf{v}_{i}, \\
    \tilde{\mathbf{P}}_i &= \frac{1}{2(\lambda_i-\xi)}
    \begin{pmatrix}
      0 & -\mathrm{i}\eta^2\pmb{\beta}_{12} \\
      \mathrm{i}\eta^{-2}\pmb{\beta}_{21} & \mathbf{0}
    \end{pmatrix},\quad \xi = -\frac{x}{2t},\quad 
    \mathbf{\Lambda}_3 = \operatorname{diag}(1, -1, -1),
  \end{aligned}
  \end{equation*}
  where $\mathbf{X}_i$ represents the $i$-th row of matrix $\mathbf{X}$ for $i=1,2,3$, and 
  $\pmb{\Delta}(\cdot)$ is defined in \eqref{def-Delta}.
\end{theorem}

\begin{remark}
  A fundamental class of solutions to the CNLS equation \eqref{CNLS} is the single-soliton solution, 
  characterized as follows: 
  \begin{equation}\label{soliton-CNLS}
    \mathbf{q}^{\mathrm{sol}}_{\omega,\gamma,v}(x-x_0,t)
    = \omega \operatorname{sech}(\omega(x-x_0-vt)) \mathrm{e}^{\mathrm{i} xv + \frac{1}{2}\mathrm{i} t(\omega^2-v^2) + \mathrm{i}\gamma} \mathbf{c},
  \end{equation}
  where $(\omega, \gamma, v)$ are parameters independent of $(x,t)$ and $\mathbf{c}=(c_1,c_2)$ 
  is the polarization vector satisfying $|\mathbf{c}|=1$. The amplitude profile $|\mathbf{q}^{\mathrm{sol}}_{\omega,\gamma,v}|$ 
  follows a hyperbolic secant shape with peak amplitude $\omega$ and velocity $v$. 
  The phase of the solution evolves linearly with respect to both space $x$ and time $t$, 
  while the polarization vector $\mathbf{c}$ governs the relative power distribution between two components of the CNLS system. 
\end{remark}
  Theorem \ref{main-result-dbar-CNLS} provides the long-time asymptotics of $\mathbf{q}(x,t)$ associated with 
$n$ simple discrete eigenvalues. This result enables us to deduce the asymptotic stability of the 
soliton solutions for the CNLS equation in the following two theorems.

\begin{theorem}\label{main-result-asymptotic-stability}
  Let $\mathbf{q}(x,t)$ be the solution to the CNLS equation \eqref{CNLS} 
  with initial data $\mathbf{q}_0(x)\in H^{1,1}(\mathbb{R})$. If there exists a sufficiently small $\varepsilon>0$ 
  such that 
  \begin{equation}
    \left\|\mathbf{q}_0(\cdot)-\mathbf{q}^{sol}_{\omega,\gamma,v}
    (\cdot-x_0,0)\right\|_{H^{1,1}(\mathbb{R})}\le \varepsilon,
  \end{equation}
  where $\mathbf{q}^{sol}_{\omega,\gamma,v}(x-x_0,t)$ is a single-soliton solution 
  of the CNLS equation \eqref{CNLS} given by equation \eqref{soliton-CNLS}, 
  then there exist two ground states
  $\mathbf{q}^{sol\pm}_{\omega_1,\gamma_1^\pm,v_1}(x-x_1^\pm,t)$ such that 
  \begin{equation}
    \left\|\mathbf{q}(\cdot,t)-
    \mathbf{q}^{sol\pm}_{\omega_1,\gamma_1^\pm,v_1}(x-x_1^\pm,t)\right\|
    _{L^{\infty}(\mathbb{R})}\lesssim t^{-1/2},\,\ \text{as} \,\ t\to\pm\infty.
  \end{equation}
\end{theorem}

The proof of Theorem \ref{main-result-asymptotic-stability} regarding 
the asymptotic stability of the single-soliton solution will be 
provided in Subsection \ref{sec-5.1}.

\begin{theorem}\label{main-result-asymptotic-stability-n-soliton}
  Let $\mathbf{q}(x,t)$ be the solution to the CNLS equation \eqref{CNLS} with 
  initial data $\mathbf{q}_0(x)\in H^{1,1}(\mathbb{R})$. 
  If there exist a sufficiently small $\varepsilon>0$ and 
  a pure $n$-soliton $\mathbf{q}_{sol}^{n}(x,t)$ such that 
  \begin{equation}
    \left\|\mathbf{q}_0(\cdot)-\mathbf{q}_{sol}^{n}(\cdot,0)
    \right\|_{H^{1,1}(\mathbb{R})}< \varepsilon,
  \end{equation} 
  then the associated scattering data generated by $\mathbf{q}_0(x)$ can be expressed as 
  $\{\mathbf{R}(\lambda),\{\lambda_j, \mathbf{c}_j\}_{j=1}^n\}$. 
  Furthermore, there exist two pure $n$-soliton solutions $\mathbf{q}_{sol}^{n\pm}(x,t)$ 
  such that 
  \begin{equation}
    \left\|\mathbf{q}(\cdot,t)-\mathbf{q}_{sol}^{n\pm}(\cdot,t)\right\|_{
    L^\infty(\mathbb{R})}\lesssim t^{-1/2},\ \ \text{as}\,\ t\to\pm\infty.
  \end{equation}
  Here, we order $\lambda_j=\xi_j+\ii\eta_j$ such that 
  \begin{equation}
    \xi_1<\xi_2<\cdots<\xi_n,
  \end{equation}
  and $\mathbf{q}_{sol}^{n\pm}(x,t)$ are associated with the scattering 
  data $\mathcal{D}_n(\mathbf{q}_{sol}^{n\pm}(x,0))=
  \{0,\{\lambda_i\}_{i=1}^n, \{\tilde{\mathbf{c}}_i^\pm\}_{i=1}^n\}$, where 
  $\{\tilde{\mathbf{c}}_i^\pm\}_{i=1}^n$ are provided by equations 
  \eqref{def-scattering-data-positive} and \eqref{def-scattering-data-neg}. 

  Moreover, the long-time asymptotics of $\mathbf{q}(x,t)$ can 
  be expressed as a sum of $n$ individual single-soliton solutions:
  \begin{equation}
    \mathbf{q}(x,t)=\sum_{k=1}^{n}\mathbf{q}_{sol}^{k\pm}(x,t)+\oo(t^{-1/2}),
    \ \ \text{as}\,\ t\to\pm\infty.
  \end{equation}
  Here, the single-soliton solutions $\mathbf{q}_{sol}^{k\pm}(x,t)$ can 
  be represented in the explicit determinant form
  \begin{equation}\label{q-n-sol-pm}
    \mathbf{q}_{sol}^{k\pm}(x,t)=
      2\eta_k\mathrm{sech}\left(2\eta_k\left(x-x_k^\pm+2t\xi_k\right)\right)
    \ee^{-2\ii\left(\xi_kx-\left(\eta_k^2-\xi_k^2\right)t\right)+\ii\phi_k^\pm}\tilde{\pmb{c}}_k^\pm,
  \end{equation}
  whose detailed expressions are given by equations \eqref{asy-q-l} and \eqref{q-sol-k-neg}.
  Alternatively, the single-soliton solutions $\mathbf{q}_{sol}^{k\pm}(x,t)$ 
  can also be described in the product form
  \begin{equation}
    \mathbf{q}_{sol}^{k\pm}(x,t)=2\eta_k\mathrm{sech}\left(2\eta_k 
  \left(x-x_k^\pm+2t\xi_k\right)\right)
    \ee^{-2\ii\left(\xi_kx-\left(\eta_k^2-
    \xi_k^2\right)t\right)-\frac{\ii\pi}{2}}\pmb{v}_k^\pm,
  \end{equation}
  as detailed in Proposition \ref{nls-n-soliton-iteration-form}.
\end{theorem}
  \begin{remark}
  The parameter $x_k^\pm$ and the polarization $\tilde{\pmb{c}}_k^\pm$ in equation \eqref{q-n-sol-pm} 
  characterize the asymptotic effect of the soliton-soliton interactions during collisions. 
  The polarization of each soliton is rotated, resulting from the altered relative power 
  distribution between its two components after the collision. 
  The observed polarization rotation constitutes a collision scenario fundamentally different 
  from that of the standard NLS soliton.
  However, the peak amplitude of $\left|\mathbf{q}_{sol}^{k\pm}(x,t)\right|$ remains $2\eta_k$ both 
  before and after the collision. This indicates that 
  the total power of each CNLS soliton is conserved (i.e., fully transmitted) throughout the interaction.
\end{remark}

The structure of this article is organized as follows: In Section 2, we construct the initial RHP 
\ref{initial-rhp-with-pole} associated with the CNLS equation \eqref{CNLS} and 
derive the recovery formula \eqref{recover-CNLS} using the inverse scattering transform. 
In Section 3, we remove the residue conditions of the 
initial RHP \ref{initial-rhp-with-pole} with the aid of Darboux transformations and 
perform the conjugation necessary for contour deformation. 
In Section 4, we apply the $\bar{\partial}$-steepest descent method to obtain
the long-time asymptotic behavior of RHP \ref{initial-rhp-without-role} without poles,
and then analyze the long-time asymptotics of the solution $\mathbf{q}(x,t)$ 
to the CNLS equation \eqref{CNLS}. In Section 5, we establish the asymptotic stability of 
$\mathbf{q}(x,t)$ in the presence of a single discrete eigenvalue and, more generally, 
$n$ discrete eigenvalues. Moreover, we analyze the laws of interaction from the perspective of the Yang-Baxter map. 
Additionally, there are three Appendices provided at the end of this paper.

\begin{remark}[Remark on Notation]
  The superscripts $^*$ and $^\dagger$ on a matrix denote the element-wise complex conjugate and the Hermitian conjugate, respectively.
  The superscript $^\top$ denotes the matrix transpose.
  $\mathbb{I}_{n}$ denotes the $n \times n$ identity matrix. Let a $3\times 3$ matrix $A$ 
  be written in block form as
  \begin{equation*}
    A=\begin{bmatrix}A_{11} & A_{12}\\ A_{21} & A_{22}\end{bmatrix},
  \end{equation*}
  where $A_{11}$ is a scalar.
\end{remark}
    
\section{Direct and inverse scattering transform}
This section is devoted to the inverse 
scattering transform for the CNLS equation 
and the construction of the associated RHP. 
This formulation serves as a prerequisite 
for the $\bar{\partial}$-steepest descent method, 
which will be applied in the following sections to 
analyze the long-time asymptotics of 
the solution of the corresponding RHP.

The CNLS equation \eqref{CNLS} is equivalent to the compatibility condition 
$\mathbf{\Phi}_{xt}(\lambda;x,t)=\mathbf{\Phi}_{tx}(\lambda;x,t)$ of 
the following Lax pair \cite{HYBLLMZXE2025}: 
\begin{equation}\label{lax-pair-Phi}
  \begin{aligned}
    \mathbf{\Phi}_x &= \left(-\mathrm{i}\lambda\mathbf{\Lambda}_3+\mathrm{i}\mathbf{Q}\right)\mathbf{\Phi}, \\
    \mathbf{\Phi}_t &= \left(-\mathrm{i}\lambda^2\mathbf{\Lambda}_3+\mathrm{i}\lambda\mathbf{Q}
    +\frac{1}{2}\left(\mathbf{Q}_{x}+\mathrm{i}\mathbf{Q}^{2}\right)\mathbf{\Lambda}_3\right)\mathbf{\Phi},
  \end{aligned}
\end{equation}
where $\lambda\in\mathbb{C}\cup \{\infty\}$ is the spectral parameter, and 
\begin{equation}
    \mathbf{\Lambda}_3=\mathrm{diag}\left(1,-1,-1\right),\quad 
    \mathbf{Q}=\begin{pmatrix}0 & \mathbf{q}\\ \mathbf{q}^\dagger & \mathbf{0}\end{pmatrix}.
\end{equation}
It follows from the given initial data $\mathbf{q}_0(x)\in H^{1,1}(\mathbb{R})$ that 
the Jost solutions $\mathbf{\Phi}^\pm$ satisfy the following asymptotic conditions: 
\begin{equation}
  \mathbf{\Phi}^{\pm}(\lambda;x,t)\to\mathrm{e}^{-\mathrm{i}\lambda(x+\lambda t)\mathbf{\Lambda}_{3}},
  \quad x\to\pm\infty.
\end{equation}
We introduce the following ansatz
\begin{equation}
  \pmb{\mu}^{\pm}(\lambda;x,t)=\mathbf{\Phi}^{\pm}(\lambda;x,t)\mathrm{e}^{\mathrm{i}\lambda(x+\lambda t)\mathbf{\Lambda}_{3}}.
\end{equation}
It can be derived from equation \eqref{lax-pair-Phi} that the matrix functions $\pmb{\mu}^{\pm}$ satisfy 
\begin{equation}\label{lax-pair-mu}
  \begin{aligned}
    \pmb{\mu}^{\pm}_{x} &= -\mathrm{i}\lambda\left[\mathbf{\Lambda}_{3},\pmb{\mu}^{\pm}\right]
    +\mathrm{i}\mathbf{Q}\pmb{\mu}^{\pm}, \quad  [\mathbf{A}, \mathbf{B}] = \mathbf{AB} - \mathbf{BA},\\
    \pmb{\mu}^{\pm}_{t} &= -\mathrm{i}\lambda^{2}\left[\mathbf{\Lambda}_{3},\pmb{\mu}^{\pm}\right]+
    \left[\mathrm{i}\lambda \mathbf{Q}+\frac{1}{2}(\mathbf{Q}_{x}+\mathrm{i}\mathbf{Q}^2)\mathbf{\Lambda}_{3}\right]\pmb{\mu}^{\pm},
  \end{aligned}
\end{equation}
and the normalization conditions 
\begin{equation}
  \pmb{\mu}^{\pm}(\lambda;x,t)\to\mathbb{I}_{3},\quad x\to\pm\infty.
\end{equation}
The Jost solutions $\pmb{\mu}^{\pm}$ of equation \eqref{lax-pair-mu} can be 
rewritten as Volterra-type integral equations
\begin{equation}\label{vol-type-mu}
  \pmb{\mu}^{\pm}(\lambda;x,t)
  =\mathbb{I}_{3}+\mathrm{i}\int_{\pm\infty}^{x}\mathrm{e}^{-\mathrm{i}\lambda(x-y)\mathbf{\Lambda}_{3}}
  \mathbf{Q}(y,t){\pmb{\mu}}^{\pm}(\lambda;y,t)\mathrm{e}^{\mathrm{i}\lambda(x-y)\mathbf{\Lambda}_{3}}\,\mathrm{d}y.
\end{equation}
We rewrite the matrix functions $\pmb{\mu}^{\pm}$ in block form 
\begin{equation}
  \pmb{\mu}^{\pm}(\lambda;x,t)=(\pmb{\mu}^{\pm}_{1}(\lambda;x,t),
  \pmb{\mu}^{\pm}_{2}(\lambda;x,t)),
\end{equation}
where $\pmb{\mu}_{1}$ and $\pmb{\mu}_{2}$ denote the first column and 
the last two columns of $\pmb{\mu}$, respectively. Based on the 
exponential factor in equation \eqref{vol-type-mu}, it is easy to conclude that 
the matrix functions $\pmb{\mu}^{-}_{1}$ and $\pmb{\mu}^{+}_{2}$ are analytic in the upper half $\lambda$-plane, 
while the matrix functions $\pmb{\mu}^{+}_{1}$ and $\pmb{\mu}^{-}_{2}$ are analytic in the lower half 
$\lambda$-plane. Thus, we define the following matrices 
\begin{equation*}
  \pmb{\mu}_{+}(\lambda;x,t)=(\pmb{\mu}^{-}_{1},\pmb{\mu}^{+}_{2}),
  \quad \pmb{\mu}_{-}(\lambda;x,t)=(\pmb{\mu}^{+}_{1},\pmb{\mu}^{-}_{2}),
\end{equation*}
where $\pmb{\mu}_{+}$ ($\pmb{\mu}_{-}$) is analytic in the upper (lower) half of the $\lambda$-plane. 
The traceless property of the matrix $\mathbf{Q}$, along with Abel's formula, ensures that 
the determinant of the matrix $\pmb{\mu}^{\pm}$ is independent of the variables $x$ and $t$. This implies that 
$\det \pmb{\mu}^{\pm}=1$, using the boundary conditions at $x\to\pm\infty$.
Moreover, since the matrix functions 
$\pmb{\mu}^{\pm}(\lambda;x,t)\mathrm{e}^{-\mathrm{i}\lambda(x+\lambda t)\mathbf{\Lambda}_{3}}$ satisfy the linear 
differential equation system \eqref{lax-pair-Phi}, they are linearly related. Hence, there exists 
a scattering matrix $\mathbf{S}(\lambda)$ such that 
\begin{equation}\label{def-scattering-matrix}
  \pmb{\mu}^-(\lambda;x,t)=\pmb{\mu}^+(\lambda;x,t)
  \mathrm{e}^{-\mathrm{i}\lambda(x+\lambda t)\mathrm{ad}_{{\mathbf{\Lambda}}_{3}}}\mathbf{S}(\lambda),\quad \lambda\in\mathbb{R},
\end{equation}
where 
\begin{equation}
  \mathbf{S}(\lambda)=
  \begin{pmatrix}\bar{a}(\lambda) & \mathbf{b}(\lambda)\\
    \overline{\mathbf{b}}(\lambda) & \mathbf{a}(\lambda)\end{pmatrix},\quad
    \det(\mathbf{S}(\lambda))=1, \quad\text{and} \quad
    \mathrm{e}^{\mathrm{ad}_{{\mathbf{\Lambda}}_{3}}}\mathbf{A}=
    \mathrm{e}^{{\mathbf{\Lambda}}_{3}}\mathbf{A}\mathrm{e}^{-{\mathbf{\Lambda}}_{3}}.
\end{equation}
We obtain the following results, the proofs and details of which can 
be found in \cite{HYBLLMZXE2025}:
\begin{itemize}
  \item The scattering matrix $\mathbf{S}(\lambda)$ satisfies the following symmetry condition:
  \begin{equation}\label{symmetry-S}
    \mathbf{S}(\lambda)\mathbf{S}^\dagger(\lambda^*)=\mathbb{I}_{3},
  \end{equation}
  which implies 
  \begin{equation}\label{relation-S-element}
    \bar{a}(\lambda)=\det(\mathbf{a}^\dagger(\lambda^*)),\quad 
    \overline{\mathbf{b}}(\lambda)=-\mathrm{adj}(\mathbf{a}^\dagger(\lambda^*))\mathbf{b}^\dagger(\lambda^*),
  \end{equation}
  where $\mathrm{adj}(\mathbf{A})$ represents the adjoint matrix of $\mathbf{A}$. 
  Hence, the scattering matrix $\mathbf{S}(\lambda)$ can be rewritten in the following 
  block form: 
  \begin{equation}
    \mathbf{S}(\lambda)=\begin{pmatrix}
      \det(\mathbf{a}^\dagger(\lambda^*)) & \mathbf{b}(\lambda)\\
      -\mathrm{adj}(\mathbf{a}^\dagger(\lambda^*))\mathbf{b}^\dagger(\lambda^*) & \mathbf{a}(\lambda)
    \end{pmatrix}.
  \end{equation}
  
  \item The functions $\bar{a}(\lambda)$ and $\mathbf{a}(\lambda)$ are 
  analytic in the upper and lower half $\lambda$-planes, respectively. Assuming that $\bar{a}(\lambda)$ has 
  a simple zero at $\lambda=\lambda_i\in\mathbb{C}_+$, there exists a norming 
  constant matrix $\mathbf{h}_{i}$ such that 
  \begin{equation}
    \begin{aligned}
      &\mathbf{\Phi}^-_1(\lambda_i;x,t)=\mathbf{\Phi}^+_2(\lambda_i;x,t)\mathbf{h}_i,\\
      &\mathbf{\Phi}^-_2(\lambda_i^*;x,t){\rm adj}(\mathbf{a}(\lambda_i^*))
      =-\mathbf{\Phi}^+_1(\lambda_i^*;x,t)\mathbf{h}_i^\dagger.
    \end{aligned}
  \end{equation}
  
  \item The reflection coefficient $\hat{\mathbf{R}}$ is defined as follows:
  \begin{equation}\label{def-reflection-coefficient}
    \hat{\mathbf{R}}(\lambda)=\mathbf{b}(\lambda)\mathbf{a}(\lambda)^{-1}.
  \end{equation}
  We denote the simple zeros of $\bar{a}(\lambda)$ as the discrete spectrum set 
  \begin{equation}
      \mathcal{Z}=\left\{ \lambda_i \mid \bar{a}(\lambda_{i})=
      \det(\mathbf{a}^\dagger(\lambda_i^*))=0 \right\}_{i=1}^{N}.
  \end{equation}
  Thus, the reflection data $\mathcal{S}(N)$ is the collection 
  \begin{equation}
    \mathcal{S}(N)=
    \left\{ \hat{\mathbf{R}}(\lambda), \quad \mathcal{Z}, \quad 
    \left\{\mathbf{h}_{i}\right\}_{i=1}^{N} \right\}.
  \end{equation}
  Furthermore, we can conclude that if the initial data $\mathbf{q}_0(x)$ for the Cauchy 
  problem of the CNLS equation \eqref{CNLS} belongs to the weighted Sobolev space 
  $H^{1,1}(\mathbb{R})$, then the reflection coefficient $\hat{\mathbf{R}}(\lambda)$ also lies in 
  $H^{1,1}(\mathbb{R})$ \cite{Liu_2019}. 
\end{itemize}
To avoid the pathologies that can arise when dealing with general initial data $\mathbf{q}_0(x)$, 
we establish the following assumption, which holds throughout the subsequent analysis.

\begin{assumption}\label{assumption-q0}
  In the context of the Cauchy problem for the CNLS equation \eqref{CNLS}, 
  the initial data $\mathbf{q}_0$ generates generic scattering data, characterized as follows:
  \begin{enumerate}
    \item There are no spectral singularities, i.e., there exists a constant $c>0$ such that 
    \begin{equation*}
      |\bar{a}(\lambda)|>c, \quad \forall \lambda\in\mathbb{R}.
    \end{equation*}
    \item The function $\bar{a}(\lambda)$ has only finitely many simple zeros, which indicates that 
    the discrete spectrum is simple and the set $\mathcal{Z}$ is finite.
    \item If the initial data $\mathbf{q}_0(x)\in H^{1,1}(\mathbb{R})$, then the corresponding reflection 
    coefficient $\hat{\mathbf{R}}(\lambda)\in H^{1,1}(\mathbb{R})$.
  \end{enumerate}
\end{assumption}

In inverse scattering theory, the goal is to recover the solution $\mathbf{q}(x,t)$ of the CNLS equation 
\eqref{CNLS} through the corresponding RHP. To construct a proper RHP with the appropriate normalization condition, 
we define the piecewise analytic matrix function $\hat{\M}(\lambda)$ in the following manner:
\begin{equation}\label{def-M}
  \hat{\M}(\lambda)=\hat{\M}(\lambda;x,t):=
  \begin{cases}
    \pmb{\mu}_+\begin{pmatrix}
      \frac{1}{\bar{a}(\lambda)} & \mathbf{0}\\ \mathbf{0} & \mathbb{I}_2
    \end{pmatrix}, & \lambda\in\mathbb{C}_+,\\
    \pmb{\mu}_-\begin{pmatrix}
      1 & \mathbf{0}\\ \mathbf{0} & \mathbf{a}^{-1}(\lambda)
    \end{pmatrix}, & \lambda\in\mathbb{C}_-.
  \end{cases}
\end{equation}
It is straightforward to show through direct calculation that the matrix $\hat{\M}(\lambda)$ satisfies the following RHP.
\begin{rhp}\label{initial-rhp-with-pole}
  Find a matrix function $\hat{\M}(\lambda)=\hat{\M}(\lambda;x,t)$ with the following properties: 
  \begin{enumerate}
    \item Analyticity: $\hat{\M}(\lambda)$ is meromorphic in $\mathbb{C}\setminus\mathbb{R}$.
    
    \item Jump condition: $\hat{\M}(\lambda)$ has continuous boundary 
    values $\hat{\M}_\pm(\lambda)$ on $\mathbb{R}$ given by
    \begin{equation}
      \hat{\M}_\pm(\lambda)=\lim_{\varepsilon \downarrow 0}\hat{\M}(\lambda\pm\mathrm{i}\varepsilon), 
    \end{equation}
    satisfying $\hat{\M}_+(\lambda)=\hat{\M}_-(\lambda)\hat{\V}(\lambda)$, where
    \begin{equation}\label{initial-Jump}
     \hat{\V}(\lambda)=
     \begin{pmatrix}
       1+\hat{\mathbf{R}}(\lambda)\hat{\mathbf{R}}^\dagger(\lambda) & -\hat{\mathbf{R}}(\lambda)\mathrm{e}^{-2\mathrm{i} t\theta(\lambda)}\\
       -\hat{\mathbf{R}}^\dagger(\lambda)\mathrm{e}^{2\mathrm{i} t\theta(\lambda)} & \mathbb{I}_2
     \end{pmatrix},
    \end{equation}
    and $\hat{\mathbf{R}}(\lambda)$ is the reflection coefficient defined in equation \eqref{def-reflection-coefficient}. 
    The phase function $\theta(\lambda)$ and the stationary point $\xi$ are given by 
    \begin{equation}
     \theta(\lambda)=\theta(\lambda;x,t)=\lambda^{2}-2\xi 
     \lambda=(\lambda-\xi)^{2}-\xi^{2},\quad \xi=-x/(2t).
    \end{equation}
    
    \item Residue conditions: $\hat{\M}(\lambda)$ has simple poles at $\lambda_i\in\mathcal{Z}$ 
    and $\lambda_i^*\in\mathcal{Z}^*$, with 
    \begin{equation}\label{residue-condition}
      \begin{aligned}
        \operatorname*{Res}_{\lambda=\lambda_i}\hat{\M} &=
        \lim_{\lambda\to\lambda_i}\hat{\M}
        \begin{pmatrix}0 & \mathbf{0}\\
          -\mathbf{c}_i\mathrm{e}^{2\mathrm{i} t\theta} & \mathbf{0}\end{pmatrix},\\
        \operatorname*{Res}_{\lambda=\lambda_i^*}\hat{\M} &=
        \lim_{\lambda\to\lambda_i^*}\hat{\M}
        \begin{pmatrix}0 & \mathbf{c}^\dagger_i\mathrm{e}^{-2\mathrm{i} t\theta}\\
          \mathbf{0} & \mathbf{0}\end{pmatrix},
      \end{aligned}
    \end{equation}
    where 
    \begin{equation}
      \mathbf{c}_i=-\frac{\mathbf{h}_i}{k_i},\quad
      k_i=\left.\frac{\mathrm{d}\bar{a}(\lambda)}{\mathrm{d}\lambda}\right|_{\lambda=\lambda_i},\quad i=1,\ldots,N.
    \end{equation}
    
    \item Asymptotic condition: 
    $\hat{\M}(\lambda)=\mathbb{I}_{3}+\mathcal{O}\left(\lambda^{-1}\right)$, as $\lambda\rightarrow\infty$.
  \end{enumerate} 
\end{rhp}

Inserting the asymptotic expansion of $\hat{\M}(\lambda)$, given by 
\begin{equation}
  \hat{\M}=\mathbb{I}_3+\lambda^{-1}\hat{\M}_1+\mathcal{O}(\lambda^{-2}),\quad 
  \lambda\to \infty,
\end{equation}
into equation \eqref{lax-pair-mu}, we arrive at the recovery formula:
\begin{equation}\label{recover-CNLS}
  \mathbf{q}(x,t)=\lim_{\lambda\rightarrow\infty} (2\lambda \hat{\M}(\lambda;x,t))_{12},
\end{equation}
where $\mathbf{q}(x,t)$ is the solution of the CNLS equation \eqref{CNLS}.

We now outline some basic notation that will be relevant for the subsequent sections:
\begin{enumerate}
  \item The norm of any matrix $\M$ is defined by $|\M|=(\operatorname{tr}(\M^\dagger \M))^{1/2}$. The $L^p$ norm 
  and $H^{i,j}$ norm of a matrix function $\mathbf{A}(\cdot)$ are expressed as 
  \begin{equation*}
    \|\mathbf{A}(\cdot)\|_{L^p}=\| |\mathbf{A}(\cdot)| \|_{L^p},\quad
    \|\mathbf{A}(\cdot)\|_{H^{i,j}}
    =\left(\sum_{a_{km}\in\mathbf{A}}\|a_{km}(\cdot)\|_{H^{i,j}}^2\right)^{1/2}.
  \end{equation*} 
  
  \item We write $A\lesssim B$ for two quantities $A$ and $B$ if there exists a constant $C>0$ such that 
  $|A|\leq CB$.
  
  \item For any oriented contour $\Sigma$, we denote the left side by $+$ and the right side by $-$.
\end{enumerate}

\section{Conjugation}
The $\bar{\partial}$-steepest descent method is 
  more conveniently applied to Riemann-Hilbert problems free of 
  residue conditions. In this section, we remove 
  the residue conditions of RHP \ref{initial-rhp-with-pole} 
  with the aid of Darboux transformations, and then perform 
  the conjugation of the jump matrix $\V(\lambda)$ 
  in preparation for deforming the oscillatory jump 
  into the appropriate decay regions.
\subsection{Decomposition of the continuous and discrete spectra}
To analyze the long-time asymptotic behavior of RHP \ref{initial-rhp-with-pole} 
with the residue condition \eqref{residue-condition} by using the $\bar{\partial}$-steepest 
descent method, we need to introduce an outer model RHP corresponding to the $n$-soliton solution, 
similar to the approach in \cite{McLaughlin_2018}. 
It is more convenient to analyze the asymptotic behavior of an RHP without a residue condition by 
the $\bar{\partial}$-steepest descent method. 
Actually, the jump condition \eqref{initial-Jump} represents the continuous spectrum, 
and the residue condition \eqref{residue-condition} represents the discrete spectrum. 
By \cite[Theorem 1]{HYBLLMZXE2025}, we can decompose the continuous spectrum and the discrete spectrum, 
establishing the following relation between $\hat{\M}(\lambda)$ and $\M(\lambda)$, the latter of which is the solution 
of an RHP without a residue condition:
\begin{equation}\label{relation-hatM-M}
  \begin{cases}
    \hat{\M}(\lambda ; x, t)=\mathbf{T}^{(+)}(\lambda ; x, t) \M(\lambda ; x, t) 
    \mathrm{diag}\left(\prod_{i=1}^{n}\frac{\lambda-\lambda_{i}^{*}}{\lambda-\lambda_{i}}, 1,1\right),
    & \operatorname{Im} \lambda>0,\\[4pt]
    \hat{\M}(\lambda ; x, t)=\mathbf{T}^{(-)}(\lambda ; x, t) \M(\lambda ; x, t) 
    \mathrm{diag}\left(1, \prod_{i=1}^{n}\frac{\lambda-\lambda_{i}}{\lambda-\lambda_{i}^{*}},
    \prod_{i=1}^{n}\frac{\lambda-\lambda_{i}}{\lambda-\lambda_{i}^{*}}\right),
    & \operatorname{Im} \lambda<0,
  \end{cases}
\end{equation}
where $\mathbf{T}^{(\pm)}(\lambda ; x, t)$ represent the Darboux matrices defined in equation \eqref{construction-T}, and 
$\M(\lambda)$ satisfies the following RHP. 
More details about the construction of $\mathbf{T}^{(\pm)}(\lambda ; x, t)$ 
can be found in Appendix A.

\begin{rhp}\label{initial-rhp-without-role}
  Find a matrix function $\M(\lambda)=\M(\lambda;x,t)$ with the following properties: 
  \begin{enumerate}
    \item Analyticity: $\M(\lambda)$ is analytic in $\mathbb{C}\setminus\mathbb{R}$.
    
    \item Jump condition: $\M(\lambda)$ has continuous boundary 
    values $\M_\pm(\lambda)$ on $\mathbb{R}$ given by
    \begin{equation}
      \M_\pm(\lambda)=\lim_{\varepsilon \downarrow 0}\M(\lambda\pm\mathrm{i}\varepsilon), 
    \end{equation}
    satisfying $\M_+(\lambda)=\M_-(\lambda)\V(\lambda)$, where
    \begin{equation}\label{def-R-no-soliton}
      \V(\lambda)=
      \begin{pmatrix}
       1+\mathbf{R}(\lambda)\mathbf{R}^\dagger(\lambda) & -\mathbf{R}(\lambda)\mathrm{e}^{-2\mathrm{i} t\theta(\lambda)}\\
       -\mathbf{R}^\dagger(\lambda)\mathrm{e}^{2\mathrm{i} t\theta(\lambda)} & \mathbb{I}_2
     \end{pmatrix},
     \end{equation}
     and $\mathbf{R}(\lambda)=\prod_{i=1}^{n}\frac{\lambda-\lambda_i^*}{\lambda-\lambda_i}\hat{\mathbf{R}}(\lambda)$.
     
     \item Asymptotic condition: 
    $\M(\lambda)=\mathbb{I}_{3}+\mathcal{O}\left(\lambda^{-1}\right)$, as $\lambda\rightarrow\infty$.
  \end{enumerate} 
\end{rhp}

According to \cite[Theorem 1]{HYBLLMZXE2025}, we can also achieve a better formula for 
recovering the solution $\mathbf{q}(x,t)$ for the CNLS equation \eqref{CNLS} 
by evaluating $\M(\lambda;x,t)$ at $\lambda=\infty$ and $\lambda=\lambda_i\in\mathcal{Z}$:
\begin{equation}\label{better-recovery-formula}
  \mathbf{q}(x,t)=\lim_{\lambda\rightarrow\infty} \left(2\lambda \M(\lambda;x,t)\right)_{12}-
  2\mathbf{X}_1\mathbf{G}^{-1}(\mathbf{X}_2)^\dagger,
\end{equation}
where 
\begin{equation*}
  \begin{aligned}
    \mathbf{X} &= \begin{bmatrix} |x_1\rangle, \cdots, |x_n\rangle \end{bmatrix},\quad 
    \mathbf{G}=\left(\mathbf{G}_{ji}\right)_{1\le j,i\le n},\quad
    \mathbf{G}_{ji}=\frac{\langle x_j|x_i\rangle}{\lambda_i-\lambda_j^*},\\
    \mathbf{v}_{i} &= \begin{bmatrix}
    \prod_{j=i+1}^{n}\frac{\lambda_i-\lambda_j^*}{\lambda_i-\lambda_j}, &
    \frac{\mathbf{c}_i^{\top}}{\lambda_i-\lambda_i^*}
    \prod_{j=1}^{i-1}\frac{\lambda_i-\lambda_j}{\lambda_i-\lambda_j^*}
  \end{bmatrix}^{\top},
    \quad \langle x_j|=|x_j\rangle^\dagger,\\
  |x_i\rangle &= \M(\lambda_i; x, t)\mathrm{diag}\left(\prod_{j=i+1}^{n}\frac{\lambda_i-\lambda_{j}^{*}}{\lambda_i-\lambda_{j}}, 1,1\right)
    \mathrm{e}^{-\mathrm{i}\lambda_i(x+\lambda_i t)\mathbf{\Lambda}_3}\mathbf{v}_i,
  \end{aligned}
\end{equation*}
and $\mathbf{X}_1$ represents the first row of $\mathbf{X}$, while $\mathbf{X}_2$ represents the last two rows of $\mathbf{X}$. 

In the following subsections, we will focus on applying the $\bar{\partial}$-steepest 
descent method to obtain the long-time asymptotic behavior of the solution $\M(\lambda;x,t)$ 
at $\lambda\in\mathcal{Z}\cup\{\infty\}$, which is sufficient for analyzing the long-time asymptotics of 
the solution $\mathbf{q}(x,t)$.

\subsection{Factorization of the jump matrix}
The key idea of the $\bar{\partial}$-steepest descent method is to introduce a non-analytic extension to 
deform the jump condition \eqref{def-R-no-soliton} with the oscillatory factors $\mathrm{e}^{\pm 2\mathrm{i}t\theta(\lambda)}$ on the 
real axis into new contours along which the new jumps are decaying. Since both exponential factors 
$\mathrm{e}^{\pm 2\mathrm{i}t\theta(\lambda)}$ appear in the jump matrix \eqref{def-R-no-soliton}, we need to introduce two distinct 
factorizations such that the exponential factors $\mathrm{e}^{\pm 2\mathrm{i}t\theta(\lambda)}$ can be deformed into 
their corresponding decay regions.
\begin{equation}
  \V(\lambda)=
  \begin{cases}
    \begin{pmatrix}
      1&-\mathrm{e}^{-2\ii t\theta}\mathbf{R}(\lambda)\\
      \mathbf{0}&\mathbb{I}_2\end{pmatrix}
    \begin{pmatrix}1&\mathbf{0}\\
      -\mathrm{e}^{2\ii t\theta}
      \mathbf{R}^\dagger(\lambda)&\mathbb{I}_2
    \end{pmatrix},\\
    \begin{pmatrix}1&\mathbf{0}\\
      -\frac{\mathrm{e}^{2\ii t\theta}\mathbf{R}^\dagger(\lambda)}{1+\mathbf{R}
      (\lambda)\mathbf{R}^\dagger(\lambda)}&\mathbb{I}_2
    \end{pmatrix}
    \begin{pmatrix}
      1+\mathbf{R}(\lambda)\mathbf{R}^\dagger(\lambda)&\mathbf{0}\\
      \mathbf{0}&(\mathbb{I}_2+
      \mathbf{R}^\dagger(\lambda)\mathbf{R}(\lambda))^{-1}
    \end{pmatrix}
    \begin{pmatrix}1
        &-\frac{\mathrm{e}^{-2\ii t\theta}\mathbf{R}(\lambda)}{1+\mathbf{R}
        (\lambda)\mathbf{R}^\dagger(\lambda)}\\
        \mathbf{0}&\mathbb{I}_2
    \end{pmatrix}.
  \end{cases}
\end{equation}
    

    


\noindent
The two factorizations are applied to the contours 
$(\xi, +\infty)$ and $(-\infty, \xi)$ for $t>0$, respectively. 
When $t<0$, the contours corresponding to each 
factorization are interchanged. Now, we consider the case where $t>0$. 
In order to perform a conjugation on the matrix $\M(\lambda)$ such that 
the new jump \eqref{Jump1} exhibits the lower/upper factorization without a diagonal factor,  
we introduce a matrix function $\pmb{\delta}(\lambda)$ satisfying the following RHP.

\begin{rhp}\label{RHP-delta}
  Find a matrix function $\pmb{\delta}(\lambda)$ with the following properties: 
  \begin{enumerate}
    \item Analyticity: $\pmb{\delta}(\lambda)$ is analytic in 
    $\mathbb{C}\setminus(-\infty,\xi)$.
    
    \item Jump condition: $\pmb{\delta}(\lambda)$ has continuous boundary 
    values $\pmb{\delta}_\pm(\lambda)$ on $(-\infty,\xi)$ given by
    \begin{equation}
      \pmb{\delta}_\pm(\lambda)=\lim_{\varepsilon \downarrow 0} \pmb{\delta}(\lambda\pm\mathrm{i}\varepsilon), 
    \end{equation}
    satisfying 
    \begin{equation}
      \pmb{\delta}_+(\lambda)=(\mathbb{I}_2+\mathbf{R}^\dagger(\lambda)\mathbf{R}(\lambda))\pmb{\delta}_-(\lambda).
    \end{equation}
    
    \item Asymptotic condition: 
    $\pmb{\delta}(\lambda)=\mathbb{I}_{2}+\mathcal{O}\left(\lambda^{-1}\right)$, as $\lambda\rightarrow\infty$.
  \end{enumerate} 
\end{rhp}

\begin{prop}\label{proposition-delta}
  For $\mathbf{R}(\lambda)\in H^{1,1}(\mathbb{R})$, the solution 
  $\pmb{\delta}(\lambda)$ of RHP \ref{RHP-delta} has the following properties:
  \begin{enumerate}[label=(\Roman*)]
    \item The solution $\pmb{\delta}(\lambda)$ exists and is unique due to the 
    positive definiteness of the jump matrix $\mathbb{I}_2+\mathbf{R}^\dagger(\lambda)\mathbf{R}(\lambda)$ 
    and the vanishing lemma \cite{ablowitz2003complex}.
    
    \item The determinant $\det\pmb{\delta}$ of $\pmb{\delta}(\lambda)$ satisfies the 
    following scalar RHP:
    \begin{equation}\label{RHP-det-delta}
      \begin{cases}
        \det\pmb{\delta}_+(\lambda)=\det\pmb{\delta}_-(\lambda)(1+|\mathbf{R}(\lambda)|^2),\quad  & \lambda<\xi,\\
        \det\pmb{\delta}(\lambda)\rightarrow1, & \lambda\rightarrow\infty.
      \end{cases}
    \end{equation}
    Furthermore, $\det\pmb{\delta}(\lambda)$ can be expressed using the Plemelj formula 
    \cite{ablowitz2003complex}:
    \begin{equation}\label{def-kappa}
      \begin{aligned}
      \det\pmb{\delta}(\lambda) &= \exp\left(\mathrm{i}\int_{-\infty}^{\xi}\frac{\kappa(s)}{s-\lambda}\,\mathrm{d}s\right),\\
      \kappa(s) &= -\frac{1}{2\pi}\ln(1+|\mathbf{R}(s)|^2).
      \end{aligned}
    \end{equation}
    
    \item $\det\pmb{\delta}(\lambda)$ and $\pmb{\delta}(\lambda)$ satisfy the following symmetries:
  \begin{equation}\label{symmetry-delta}
    \det\pmb{\delta}(\lambda)(\det\pmb{\delta}(\lambda^*))^*=1,\quad
    \pmb{\delta}(\lambda)\pmb{\delta}^\dagger(\lambda^*)=\mathbb{I}_2.
  \end{equation}
  
  \item $\det\pmb{\delta}(\lambda)$ and $\pmb{\delta}(\lambda)$ satisfy the 
  following boundedness properties. For $\lambda\in\mathbb{R}$, 
  \begin{equation}\label{boundedness-delta1}
    |(\det\pmb{\delta}_\pm(\lambda))^{\pm 1}|\lesssim 1,
    \quad|\pmb{\delta}_\pm(\lambda)^{\pm 1}|\lesssim 1.
  \end{equation}
  For $\lambda\in\mathbb{C}\setminus\mathbb{R}$, 
  \begin{equation}\label{boundedness-delta2}
    |(\det\pmb{\delta}(\lambda))^{\pm 1}|\lesssim 1,
    \quad|\pmb{\delta}^{\pm 1}(\lambda)|\lesssim 1.
  \end{equation}
  
  \item As $\lambda\to\xi$ along any ray $\xi+\mathrm{e}^{\mathrm{i} \phi}$ with $|\phi|\le c < \pi$,
    \begin{equation}\label{boundedness-detdelta}
      |\det\pmb{\delta}(\lambda)-T_0(\xi)(\lambda-\xi)^{\mathrm{i}\kappa}|\lesssim
      |\lambda-\xi|^{\frac{1}{2}},
    \end{equation}
    where $\kappa:=\kappa(\xi)=-\frac{1}{2\pi}\ln(1+|\mathbf{R}(\xi)|^2)$,
    \begin{equation}
      T_0(\xi)=\mathrm{e}^{\mathrm{i}\beta(\xi,\xi)},\quad
      \beta(\lambda,\xi)=-\kappa(\xi)\ln(\lambda-\xi+1)+\int_{-\infty}^{\xi}
      \frac{\kappa(s)-\chi(s)\kappa(\xi)}{s-\lambda}\,\mathrm{d}s,
    \end{equation}
    and $\chi(s)$ is the characteristic function of the interval $(\xi-1,\xi)$. 
    Here we choose the branch of the logarithm with a cut along $(-\infty,\xi-1)$.
  \end{enumerate}
\end{prop}
\begin{proof}
  The proofs of (I)-(IV) can be found in \cite{HYBLLMZXE2025}. For (V), using the fact that 
  \begin{equation}\label{boundedness-lamda-xi}
    |(\lambda-\xi)^{\mathrm{i}\kappa}|\leq \mathrm{e}^{-\pi\kappa}=\sqrt{1+|\mathbf{R}(\xi)|^2},
   \end{equation}
   we can derive 
   \begin{equation}
    \begin{aligned}
      |\det\pmb{\delta}(\lambda)-T_0(\xi)(\lambda-\xi)^{\mathrm{i}\kappa}|
      &= |(\lambda-\xi)^{\mathrm{i}\kappa}(\mathrm{e}^{\mathrm{i}\beta(\lambda,\xi)}-
      \mathrm{e}^{\mathrm{i}\beta(\xi,\xi)})|\\
      &\lesssim |\mathrm{e}^{\mathrm{i}\beta(\lambda,\xi)-\mathrm{i}\beta(\xi,\xi)}-1|\\
      &= \left|\int_{0}^{1}\frac{\mathrm{d}}{\mathrm{d}s}
      \mathrm{e}^{\mathrm{i} s(\beta(\lambda,\xi)-\beta(\xi,\xi))}\,\mathrm{d}s\right|\\
      &\lesssim
      |\beta(\lambda,\xi)-\beta(\xi,\xi)|
      \sup_{0\le s\le 1}
      |\mathrm{e}^{\mathrm{i} s(\beta(\lambda,\xi)-\beta(\xi,\xi))}|.
    \end{aligned}
   \end{equation}
   By the properties of the logarithm, we deduce that 
   \begin{equation*}
    |\ln(\lambda-\xi+1)|\lesssim |\lambda-\xi|
    \lesssim |\lambda-\xi|^{1/2},\quad \text{as} \quad \lambda\to\xi.
   \end{equation*}
   It follows from $\mathbf{R}(\lambda)\in H^{1,1}(\mathbb{R})$ 
   and Lemma 23.3 in \cite{beals1988direct} that 
   \begin{equation*}
      \left|\int_{-\infty}^{\xi}
        \frac{\kappa(s)-\chi(s)\kappa(\xi)}{s-\lambda}\,\mathrm{d}s-
        \int_{-\infty}^{\xi}
      \frac{\kappa(s)-\chi(s)\kappa(\xi)}{s-\xi}\,\mathrm{d}s\right|
      \lesssim |\lambda-\xi|^{1/2}\|\kappa'\|_{L^2(\mathbb{R})}
      \lesssim |\lambda-\xi|^{1/2}\|\mathbf{R}\|_{H^{1,1}(\mathbb{R})}.
   \end{equation*}
   Therefore, we conclude that 
   \begin{equation}
    |\beta(\lambda,\xi)-\beta(\xi,\xi)|\lesssim |\lambda-\xi|^{1/2}.
   \end{equation}
   The result then follows immediately, which completes the proof.
\end{proof}
We define a matrix function
\begin{equation}\label{def-Delta}
  \pmb{\Delta}(\lambda)=
  \begin{pmatrix}
    \det\pmb{\delta}(\lambda) & \mathbf{0}\\
    \mathbf{0} & \pmb{\delta}^{-1}(\lambda)
  \end{pmatrix}
\end{equation}
and apply the following transformation: 
\begin{equation}\label{transform1}
  \M^{(1)}(\lambda)= \M(\lambda)\pmb{\Delta}^{-1}(\lambda).
\end{equation}
It is straightforward to verify that if $\M(\lambda)$ solves RHP \ref{initial-rhp-without-role}, 
then the new matrix-valued function $\M^{(1)}(\lambda)$ is the solution to the following RHP.

\begin{rhp}\label{rhp-M1}
  Find a matrix function $\M^{(1)}(\lambda)=\M^{(1)}(\lambda;x,t)$ with the following properties: 
  \begin{enumerate}
    \item Analyticity: $\M^{(1)}(\lambda)$ is analytic in $\mathbb{C}\setminus\mathbb{R}$.
    
    \item Jump condition: $\M^{(1)}(\lambda)$ has continuous boundary 
    values $\M^{(1)}_\pm(\lambda)$ on $\mathbb{R}$ given by
    \begin{equation}
      \M^{(1)}_\pm(\lambda)=\lim_{\varepsilon \downarrow 0} \M^{(1)}(\lambda\pm\mathrm{i}\varepsilon), 
    \end{equation}
    satisfying $\M^{(1)}_+(\lambda)=\M^{(1)}_-(\lambda)\V^{(1)}(\lambda)$, where
    \begin{equation}\label{Jump1}
      \V^{(1)}(\lambda)=
      \begin{cases}
       \begin{pmatrix}
         1 & -\det\pmb{\delta}\ee^{-2\ii t\theta}\mathbf{R}(\lambda)\pmb{\delta}\\
         \mathbf{0} & \mathbb{I}_2
       \end{pmatrix}
       \begin{pmatrix}
         1 & \mathbf{0}\\
         -(\det\pmb{\delta})^{-1}\ee^{2\ii t\theta}\pmb{\delta}^{-1}\mathbf{R}^\dagger(\lambda) & \mathbb{I}_2
       \end{pmatrix},
       &\lambda\in(\xi,+\infty),\\[15pt]
       \begin{pmatrix}
         1 & \mathbf{0}\\
         -(\det\pmb{\delta}_-)^{-1}\ee^{2\ii t\theta}\pmb{\delta}_-^{-1}
         \frac{\mathbf{R}^\dagger(\lambda)}{1+\mathbf{R}\mathbf{R}^\dagger} & \mathbb{I}_2
       \end{pmatrix}
       \begin{pmatrix}
         1 & -\det\pmb{\delta}_+\ee^{-2\ii t\theta}
         \frac{\mathbf{R}(\lambda)}{1+\mathbf{R}\mathbf{R}^\dagger}\pmb{\delta}_+\\
         \mathbf{0} & \mathbb{I}_2
       \end{pmatrix},
       &\lambda\in(-\infty,\xi).
      \end{cases}
     \end{equation}
     
     \item Asymptotic condition: 
    $\M^{(1)}(\lambda)=\mathbb{I}_{3}+\mathcal{O}\left(\lambda^{-1}\right)$, as $\lambda\rightarrow\infty$.
  \end{enumerate} 
\end{rhp}

In the next section, we intend to utilize $T_0(\xi)(\lambda-\xi)^{\mathrm{i}\kappa}$ as
a suitable approximation for $\det\pmb{\delta}(\lambda)$ to construct 
the non-analytic extension necessary for the contour deformation.

\section{Long-time asymptotics for the focusing CNLS equation}
In this section, we employ the $\bar{\partial}$-steepest descent method to 
study the long-time asymptotic behavior of the solution $\M^{(1)}(\lambda)$ of 
RHP \ref{rhp-M1}. 
The crucial step is to introduce factorizations of the jump matrix 
whose factors admit continuous---but not necessarily analytic-extensions 
off the real axis. This allows us to deform the oscillatory jump along 
the real axis onto new contours along which the jumps are decaying, using 
non-analytic extensions. A consequence of this non-analytic transformation is that 
the new unknown $\M^{(2)}(\lambda)$ has nonzero $\bar{\partial}$-derivatives inside 
the regions where the extensions are introduced, and thus satisfies a 
hybrid $\bar{\partial}$-RHP.

\subsection{$\bar{\partial}$-extensions for contour deformation}
We now introduce the non-analytic extension $\pmb{\mathcal{R}}^{(2)}(\lambda)$ defined in Lemma 
\ref{def-extension} to perform the contour deformation on RHP \ref{rhp-M1}. According to the decay 
regions of the exponential factors $\mathrm{e}^{\pm 2\mathrm{i} t\theta}$, the new contour $\Sigma$ 
is chosen to be 
\begin{equation}
  \Sigma=\bigcup_{k=1}^4\Sigma_k,
  \quad\Sigma_k=\xi+\mathrm{e}^{\frac{(2k-1)\pi\mathrm{i}}{4}}\mathbb{R}_+,\quad k=1,2,3,4.
\end{equation}
The complex plane $\mathbb{C}$ is divided by the real axis and the new contour $\Sigma$ into 
six open sectors $\Omega_k$, $k = 1,\ldots,6$, as shown in Figure \ref{region-extension}.
Additionally, let 
\begin{equation}
  \rho=\frac{1}{2}\min_{\lambda_1\neq\lambda_2\in\mathcal{Z}\cup\mathcal{Z}^*}
  |\lambda_1-\lambda_2|.
\end{equation}
Since the discrete eigenvalues appear in conjugate pairs and do not lie on the real axis by 
Assumption \ref{assumption-q0}, we deduce that $\operatorname{dist}(\mathcal{Z},\mathbb{R})\ge \rho$.
In what follows, we introduce a smooth cutoff function $\chi(\lambda)\in C^\infty_0(\mathbb{C},[0,1])$, 
which is supported near the discrete spectrum such that 
\begin{equation}
  \chi(\lambda)=\begin{cases}
    1, & \operatorname{dist}(\lambda,\mathcal{Z}\cup\mathcal{Z}^*)<\rho/3,\\
    0, & \operatorname{dist}(\lambda,\mathcal{Z}\cup\mathcal{Z}^*)>2\rho/3.
  \end{cases}
\end{equation}

\begin{figure}[h]
  \centering
  \begin{tikzpicture}
    \fill[gray!30] (0,0) -- (5,0) -- (5,3) -- (3,3) -- cycle;
    \fill[gray!30] (0,0) -- (-5,0) -- (-5,3) -- (-3,3) -- cycle;
    \fill[gray!30] (0,0) -- (5,0) -- (5,-3) -- (3,-3) -- cycle;
    \fill[gray!30] (0,0) -- (-5,0) -- (-5,-3) -- (-3,-3) -- cycle;
    \draw[-] (-5, 0) -- (5, 0) ;
    \draw[->] (-1.5, 1.5) -- (1.5, -1.5) ;
    \draw[->] (-1.5, -1.5) -- (1.5, 1.5) ;
    \draw[->] (-3, 3) -- (-1.5, 1.5) ;
    \draw[->] (-3, -3) -- (-1.5, -1.5) ;
    \draw[-] (1.5, 1.5) -- (3, 3) ;
    \draw[-] (1.5, -1.5) -- (3, -3) ;
    \node at (1, 0.5)  {$\Omega_1$};
    \node at (0, 1)  {$\Omega_2$};
    \node at (-1, 0.5)  {$\Omega_3$};
    \node at (-1, -0.5)  {$\Omega_4$};
    \node at (0, -1)  {$\Omega_5$};
    \node at (0, -0.5)  {$\xi$};
    \node at (1, -0.5)  {$\Omega_6$};
    \node at (3, 2.5)  {$\Sigma_1$};
    \node at (3, -2.5)  {$\Sigma_4$};
    \node at (-3, 2.5)  {$\Sigma_2$};
    \node at (-3, -2.5)  {$\Sigma_3$};
    
    \node at (3, 1)  {$\pmb{\mathcal{R}}^{ (2)}=\begin{pmatrix}1&\mathbf{0}\\-\pmb{\mathcal{R}}_1\ee^{2\ii t\theta}&\mathbb{I}_2\end{pmatrix}$};
    \node at (3, -1)  {$\pmb{\mathcal{R}}^{ (2)}=\begin{pmatrix}1&
      \pmb{\mathcal{R}}_6\ee^{-2\ii t\theta}\\\mathbf{0}&\mathbb{I}_2\end{pmatrix}$};
    \node at (-3.2, -1)  {$\pmb{\mathcal{R}}^{ (2)}=\begin{pmatrix}1&\mathbf{0}\\\pmb{\mathcal{R}}_4\ee^{2\ii t\theta}&\mathbb{I}_2\end{pmatrix}$};
    \node at (-3.2, 1)  {$\pmb{\mathcal{R}}^{ (2)}=\begin{pmatrix}1&
      -\pmb{\mathcal{R}}_3\ee^{-2\ii t\theta}\\\mathbf{0}&\mathbb{I}_2\end{pmatrix}$};
    \node at (0,2) {$\pmb{\mathcal{R}}^{ (2)}=\mathbb{I}_3$};
    \node at (0,-2) {$\pmb{\mathcal{R}}^{ (2)}=\mathbb{I}_3$};
    
    \fill[white] (-2, 2) circle (3pt);
    \fill[white] (-2, -2) circle (3pt);

    \fill[white] (4, 2.5) circle (3pt);\fill[white] (4, -2.5) circle (3pt);
    
\end{tikzpicture}
\caption{\small{The contours $\Sigma_i$ and regions $\Omega_k$, for $i=1,\ldots,4$ and 
$k=1,\ldots,6$. 
The $\bar{\partial}$ derivative
$\bar{\partial}\M^{(2)}$ is supported in the gray region, excluding the neighborhoods of the discrete spectrums.}
}
\label{region-extension}
\end{figure}
Next, we introduce another matrix-valued function $\M^{(2)}(\lambda)$ as follows:
\begin{equation}\label{transform2}
  \M^{(2)}(\lambda)=\M^{(1)}(\lambda)\pmb{\mathcal{R}}^{(2)}(\lambda),
\end{equation}
where $\pmb{\mathcal{R}}^{(2)}(\lambda)$ is rigorously defined in Lemma \ref{def-extension}, 
taking the following piecewise form:
\begin{equation}\label{def-R2}
  \pmb{\mathcal{R}}^{(2)}(\lambda)=
  \begin{cases}
    \begin{pmatrix}1 & \mathbf{0}\\ -\pmb{\mathcal{R}}_1(\lambda)\mathrm{e}^{2\mathrm{i} t\theta(\lambda)} & \mathbb{I}_2\end{pmatrix}, & \lambda\in\Omega_1,\\[12pt]
    \begin{pmatrix}1 & -\pmb{\mathcal{R}}_3(\lambda)\mathrm{e}^{-2\mathrm{i} t\theta(\lambda)}\\ \mathbf{0} & \mathbb{I}_2\end{pmatrix}, & \lambda\in\Omega_3,\\[12pt]
    \begin{pmatrix}1 & \mathbf{0}\\ \pmb{\mathcal{R}}_4(\lambda)\mathrm{e}^{2\mathrm{i} t\theta(\lambda)} & \mathbb{I}_2\end{pmatrix}, & \lambda\in\Omega_4,\\[12pt]
    \begin{pmatrix}1 & \pmb{\mathcal{R}}_6(\lambda)\mathrm{e}^{-2\mathrm{i} t\theta(\lambda)}\\ \mathbf{0} & \mathbb{I}_2\end{pmatrix}, & \lambda\in\Omega_6,\\[12pt]
    \begin{pmatrix}1 & \mathbf{0}\\ \mathbf{0} & \mathbb{I}_2\end{pmatrix}, & \lambda\in\Omega_2\cup\Omega_5.
  \end{cases}
\end{equation}

\begin{remark}
The non-analytic extension $\pmb{\mathcal{R}}^{(2)}(\lambda)$ is designed to remove the jump \eqref{Jump1} on the real axis 
and to establish new analytic jumps with the desired exponential decay along the contour $\Sigma$. 
Compared to the scalar case presented in \cite{McLaughlin_2018}, the boundary value of the extension 
$\pmb{\mathcal{R}}^{(2)}(\lambda)$ on the contour $\Sigma$, defined in equation \eqref{boundary-R}, includes $\pmb{\delta}(\lambda)$, 
which lacks a closed form. 
Consequently, the jump matrix of the subsequent $\bar{\partial}$-RHP \ref{dbar-rhp} 
does not directly match the jump matrix of a solvable model RHP. 
Nonetheless, this specific design of the extension $\pmb{\mathcal{R}}^{(2)}(\lambda)$ has the advantage of enabling us to 
derive its boundedness and the corresponding estimate \eqref{condition-R} 
for its $\bar{\partial}$-derivative. This, in turn, guarantees the 
existence and the optimal estimation of the solution $\M^{(3)}(\lambda)$ for the pure $\bar{\partial}$-problem \ref{rhp-without-dbar}. 
\end{remark}

\begin{lemma}\label{def-extension}
  There exist functions $\pmb{\mathcal{R}}_{j}:\Omega_j\to\mathbb{C}$, for $j = 1,3,4,6$, whose boundary values satisfy
  \begin{equation}\label{boundary-R}
    \begin{aligned}
        \pmb{\mathcal{R}}_{1}(\lambda) &= 
        \begin{cases}
          -(\det\pmb{\delta})^{-1}\pmb{\delta}^{-1}\mathbf{R}^{\dagger}(\lambda),  \\
          -\pmb{\delta}^{-1}\mathbf{R}^\dagger(\xi) T^{-1}_{0}(\xi)(\lambda-\xi)^{-\mathrm{i} \kappa}
          (1-\chi(\lambda)),
        \end{cases} 
        &&
        \begin{aligned}
          &\lambda \in(\xi, +\infty),\\
          &\lambda \in \Sigma_{1},
        \end{aligned}\\[8pt]
        \pmb{\mathcal{R}}_{3}(\lambda) &= 
        \begin{cases}
          -(\det\pmb{\delta}_{+})\frac{\mathbf{R}(\lambda)}{1+|\mathbf{R}(\lambda)|^{2}} \pmb{\delta}_+, \\
          -\frac{\mathbf{R}(\xi)}{1+|\mathbf{R}(\xi)|^{2}}\pmb{\delta}T_{0}(\xi)(\lambda-\xi)^{\mathrm{i} \kappa}
          (1-\chi(\lambda)),
        \end{cases}
        &&
        \begin{aligned}
          & \lambda \in(-\infty, \xi),\\
          & \lambda \in \Sigma_{2},
        \end{aligned}\\[8pt]
        \pmb{\mathcal{R}}_{4}(\lambda) &= 
        \begin{cases}
          -(\det\pmb{\delta}_-)^{-1}\pmb{\delta}^{-1}_{-}\frac{\mathbf{R}^\dagger(\lambda)}{1+|\mathbf{R}(\lambda)|^{2}}, \\
          -\pmb{\delta}^{-1}\frac{\mathbf{R}^\dagger(\xi)}{1+|\mathbf{R}(\xi)|^{2}}T^{-1}_{0}(\xi)(\lambda-\xi)^{-\mathrm{i} \kappa}
          (1-\chi(\lambda)),
        \end{cases}
        &&
        \begin{aligned}
          & \lambda \in(-\infty, \xi), \\
          & \lambda \in \Sigma_{3},
        \end{aligned}\\[8pt]
        \pmb{\mathcal{R}}_{6}(\lambda) &= 
        \begin{cases}
          -\det\pmb{\delta}\mathbf{R}(\lambda)\pmb{\delta},  \\
          -\mathbf{R}(\xi)\pmb{\delta} T_{0}(\xi)(\lambda-\xi)^{\mathrm{i} \kappa}
          (1-\chi(\lambda)),
        \end{cases} 
        &&
        \begin{aligned}
          & \lambda \in(\xi, +\infty), \\
          & \lambda \in \Sigma_{4},
        \end{aligned}
    \end{aligned}
    \end{equation}
    and satisfy the following estimates:
    \begin{equation}\label{condition-R}
      \begin{aligned}
        |\pmb{\mathcal{R}}_j(\lambda)| &\lesssim \sin^2(\arg(\lambda-\xi))+\langle\operatorname{Re}\lambda\rangle^{-1/2},\\
        |\bar{\partial}\pmb{\mathcal{R}}_j(\lambda)| &\lesssim |\bar{\partial}\chi(\lambda)|+
        |\mathbf{R}'(\operatorname{Re}\lambda)|+
        |\lambda-\xi|^{-1/2},\quad \forall\lambda\in\Omega_j,\\
        \bar{\partial}\pmb{\mathcal{R}}_j(\lambda) &= 0,\quad
        \text{if } \operatorname{dist}(\lambda,\mathcal{Z}\cup\mathcal{Z}^*)<\rho/3,
      \end{aligned}
    \end{equation}
    where $\langle x\rangle=\sqrt{1+x^2}$.
\end{lemma}
\begin{proof}
  Only the details for $\pmb{\mathcal{R}}_1(\lambda)$ will be presented here, 
  while the other cases can be easily inferred. 
  Define the function $\pmb{f}_1(\lambda)$ on the region $\overline{\Omega}_1$:
  \begin{equation}
    \pmb{f}_1(\lambda)=\mathbf{R}^\dagger(\xi)\det\pmb{\delta}(\lambda)\,T_0^{-1}(\xi)(\lambda-\xi)^{-\mathrm{i}\kappa}.
  \end{equation}
  Consider $\lambda=\xi+s\mathrm{e}^{\mathrm{i}\phi}\in\overline{\Omega}_1$. 
  The extension $\pmb{\mathcal{R}}_1(\lambda)$ can be constructed as 
  \begin{equation}
    \pmb{\mathcal{R}}_1(\lambda)=-(\det\pmb{\delta})^{-1}\pmb{\delta}^{-1}[\pmb{f}_1(\lambda)+(\mathbf{R}^\dagger(\operatorname{Re}\lambda)-\pmb{f}_1(\lambda))\cos 2\phi]
    (1-\chi(\lambda)),\quad \phi\in\left[0,\frac{\pi}{4}\right].
  \end{equation}
  It is evident that $\pmb{\mathcal{R}}_1(\lambda)$ fulfills the boundary conditions \eqref{boundary-R} 
  since $\cos 2\phi$ vanishes on $\Sigma_1$ 
  and $\chi(\lambda)$ equals zero on the real axis.
  
  We first consider the boundedness of $\pmb{\mathcal{R}}_1(\lambda)$. 
  By part (IV) of Proposition \ref{proposition-delta} and the 
  fact \eqref{boundedness-lamda-xi}, we derive 
 \begin{equation}\label{bdd-R}
  \begin{aligned}
    |\pmb{\mathcal{R}}_1(\lambda)| &\lesssim 2\sin^2\phi\,|(\lambda-\xi)^{-\mathrm{i}\kappa}|+
    |\mathbf{R}(\operatorname{Re}\lambda)|\cos 2\phi\\
    &\lesssim \sin^2\phi+|\mathbf{R}(\operatorname{Re}\lambda)|.
  \end{aligned}
 \end{equation}
 As a consequence of $\mathbf{R}(\lambda)\in H^{1,1}(\mathbb{R})$, 
 $\mathbf{R}(\lambda)$ is H\"older-$1/2$ continuous. Moreover, we can derive 
 \begin{equation}
  \begin{aligned}
    \langle x\rangle|\mathbf{R}(x)|^2
    =\int_{-\infty}^{x}\partial_y
    (\langle y\rangle|\mathbf{R}(y)|^2)\,\mathrm{d}y
    &\lesssim\int_{-\infty}^{x}\frac{|y||\mathbf{R}(y)|^2}{\sqrt{1+y^2}}+
    |\langle y\rangle(\mathbf{R}'\mathbf{R}^\dagger+
    \mathbf{R}\mathbf{R}^{\dagger \prime})|\,\mathrm{d}y\\
    &\lesssim\|\mathbf{R}\|_{L^2}+\|\langle y\rangle\mathbf{R}\|_{L^2}
    \|\mathbf{R}'\|_{L^2}\lesssim 1,
  \end{aligned}
 \end{equation} 
 which implies 
 $|\mathbf{R}(\operatorname{Re}\lambda)|\lesssim\langle\operatorname{Re}\lambda\rangle^{-1/2}$. 
 Then we obtain the first inequality in equation \eqref{condition-R}.

 We have
\begin{equation}
  \bar{\partial}=\frac{1}{2}(\partial_x+\mathrm{i}\partial_y)=\mathrm{e}^{\mathrm{i}\phi}\frac{\partial_s+\mathrm{i} s^{-1}\partial_\phi}{2},
\end{equation}
so the $\bar{\partial}$-derivative of $\pmb{\mathcal{R}}_1(\lambda)$ can be calculated by
\begin{equation}
  \begin{aligned}
    \bar{\partial}\pmb{\mathcal{R}}_1(\lambda) &=
    (\det\pmb{\delta})^{-1}\pmb{\delta}^{-1}[\pmb{f}_1(\lambda)+(\mathbf{R}^\dagger(\operatorname{Re}\lambda)-\pmb{f}_1(\lambda))\cos 2\phi]
    \bar{\partial}\chi(\lambda)\\
    &\quad -(\det\pmb{\delta})^{-1}\pmb{\delta}^{-1}[(\mathbf{R}^\dagger(\operatorname{Re}\lambda)
    -\pmb{f}_1(\lambda))\bar{\partial}\cos 2\phi+\bar{\partial}\mathbf{R}^\dagger(\operatorname{Re}\lambda)\cos 2\phi]
    (1-\chi(\lambda))\\
    &=(\det\pmb{\delta})^{-1}\pmb{\delta}^{-1}[\pmb{f}_1(\lambda)+(\mathbf{R}^\dagger(\operatorname{Re}\lambda)-\pmb{f}_1(\lambda))\cos 2\phi]
    \bar{\partial}\chi(\lambda)\\
    &\quad -(\det\pmb{\delta})^{-1}\pmb{\delta}^{-1}\left[-\mathrm{i}\mathrm{e}^{\mathrm{i}\phi}\frac{\mathbf{R}^\dagger(\operatorname{Re}\lambda)
    -\pmb{f}_1(\lambda)}{|\lambda-\xi|}\sin 2\phi+\frac{1}{2}\mathbf{R}^{\dagger \prime}(\operatorname{Re}\lambda)
    \cos 2\phi\right](1-\chi(\lambda))\\
    &\equiv \mathrm{I}_1+\mathrm{I}_2+\mathrm{I}_3.
  \end{aligned}
\end{equation}
In a manner similar to equation \eqref{bdd-R}, we can bound $\mathrm{I}_1$ and $\mathrm{I}_3$ by 
\begin{equation}
  |\mathrm{I}_1|\lesssim|\bar{\partial}\chi(\lambda)|,\quad
  |\mathrm{I}_3|\lesssim|\mathbf{R}'(\operatorname{Re}\lambda)|.
\end{equation}
For the estimate on $\mathrm{I}_2$, it suffices to provide a proper estimate between 
$\mathbf{R}^\dagger(\operatorname{Re}\lambda)$ and $\pmb{f}_1(\lambda)$. We decompose it 
into two parts as follows:
\begin{equation}
  |\mathbf{R}^\dagger(\operatorname{Re}\lambda)
  -\pmb{f}_1(\lambda)|\le |\mathbf{R}^\dagger(\operatorname{Re}\lambda)-\mathbf{R}^\dagger(\xi)|+
  |\mathbf{R}^\dagger(\xi)-\pmb{f}_1(\lambda)|.
\end{equation}
By the Cauchy-Schwarz inequality, the first part can be estimated by 
\begin{equation}
  |\mathbf{R}^\dagger(\operatorname{Re}\lambda)-\mathbf{R}^\dagger(\xi)|
  \leq\left|\int_{\xi}^{\operatorname{Re}\lambda}\mathbf{R}'(s)\,\mathrm{d}s\right| \leq \|\mathbf{R}\|_{H^1(\mathbb{R})}
  |\lambda-\xi|^{1/2}.
\end{equation}
For the last term, by part (V) of Proposition \ref{proposition-delta}, we deduce that 
\begin{equation}
  |\mathbf{R}^\dagger(\xi)-\pmb{f}_1(\lambda)|\lesssim
  |\det\pmb{\delta}(\lambda)-T_0(\xi)(\lambda-\xi)^{\mathrm{i}\kappa}|\lesssim
  |\lambda-\xi|^{1/2}
\end{equation}
in a neighborhood of $\lambda=\xi$. Together with part (IV) of Proposition \ref{proposition-delta} and 
the fact \eqref{boundedness-lamda-xi}, we can also derive  
\begin{equation}
  \frac{|\mathbf{R}^\dagger(\xi)-\pmb{f}_1(\lambda)|}{|\lambda-\xi|}
  \lesssim \frac{1}{|\lambda-\xi|}\lesssim 
  |\lambda-\xi|^{-1/2}
\end{equation}
as $\lambda$ moves away from $\xi$. 
Hence, the bound \eqref{condition-R} for $\lambda\in\Omega_1$ follows immediately. This completes the proof.
\end{proof}
It follows directly from Lemma \ref{def-extension} and RHP \ref{rhp-M1} that 
$\M^{(2)}(\lambda)$ satisfies the following $\bar{\partial}$-RHP.

\begin{rhp}\label{dbar-rhp}
  ($\bar{\partial}$-RHP) Find a matrix function $\M^{(2)}(\lambda)=\M^{(2)}(\lambda;x,t)$ with the following properties: 
  \begin{enumerate}
    \item Jump condition: $\M^{(2)}(\lambda)$ has continuous boundary 
    values $\M^{(2)}_\pm(\lambda)$ on $\Sigma$ given by
    \begin{equation}
      \M^{(2)}_\pm(\lambda)=\lim_{\varepsilon \downarrow 0} \M^{(2)}(\lambda\pm\mathrm{i}\varepsilon), 
    \end{equation}
    satisfying $\M^{(2)}_+(\lambda)=\M^{(2)}_-(\lambda)\V^{(2)}(\lambda)$, where
    \begin{equation}\label{jump-M2}
      \V^{(2)}(\lambda)=
      \begin{cases}
       \begin{pmatrix}
         1 & \mathbf{0}\\
         -\pmb{\delta}^{-1}(\lambda)\mathbf{R}^\dagger(\xi) T^{-1}_{0}(\xi)(\lambda-\xi)^{-\mathrm{i} \kappa}
         \mathrm{e}^{2\mathrm{i} t\theta(\lambda)}(1-\chi(\lambda)) & \mathbb{I}_2
       \end{pmatrix},
       & \lambda\in\Sigma_1,\\[15pt]
       \begin{pmatrix}
         1 & -\frac{\mathbf{R}(\xi)}{1+|\mathbf{R}(\xi)|^{2}}
         \pmb{\delta}(\lambda)T_{0}(\xi)(\lambda-\xi)^{\mathrm{i} \kappa}\mathrm{e}^{-2\mathrm{i} t\theta(\lambda)}(1-\chi(\lambda))\\
         \mathbf{0} & \mathbb{I}_2
       \end{pmatrix},
       & \lambda\in\Sigma_2,\\[15pt]
       \begin{pmatrix}
         1 & \mathbf{0}\\
         -\pmb{\delta}^{-1}(\lambda)\frac{\mathbf{R}^\dagger(\xi)}
         {1+|\mathbf{R}(\xi)|^{2}}T^{-1}_{0}(\xi)(\lambda-\xi)^{-\mathrm{i} \kappa}
         \mathrm{e}^{2\mathrm{i} t\theta(\lambda)}(1-\chi(\lambda)) & \mathbb{I}_2
       \end{pmatrix},
       & \lambda\in\Sigma_3,\\[15pt]
       \begin{pmatrix}
         1 & -\mathbf{R}(\xi)\pmb{\delta}(\lambda) T_{0}(\xi)(\lambda-\xi)^{\mathrm{i} \kappa}
         \mathrm{e}^{-2\mathrm{i} t\theta(\lambda)}(1-\chi(\lambda))\\
         \mathbf{0} & \mathbb{I}_2
       \end{pmatrix},
       & \lambda\in\Sigma_4.
      \end{cases}
     \end{equation}
     
     \item $\bar{\partial}$-condition: 
     For $\lambda\in\mathbb{C}\setminus\Sigma$, $\M^{(2)}(\lambda)$ has nonzero 
     $\bar{\partial}$-derivatives in each region $\Omega_j$, for $j=1,3,4,6$,
     \begin{equation}
      \bar{\partial}\M^{(2)}(\lambda)=\M^{(2)}(\lambda)\bar{\partial}\pmb{\mathcal{R}}^{(2)}(\lambda),
     \end{equation}
     where
  \begin{equation}\label{dbar-R}
    \bar{\partial}\pmb{\mathcal{R}}^{(2)}(\lambda)=
    \begin{cases}
      \begin{pmatrix}0 & \mathbf{0}\\ -\bar{\partial}\pmb{\mathcal{R}}_1(\lambda)\mathrm{e}^{2\mathrm{i} t\theta(\lambda)} & \mathbf{0}\end{pmatrix}, & \lambda\in\Omega_1,\\[10pt]
      \begin{pmatrix}0 & -\bar{\partial}\pmb{\mathcal{R}}_3(\lambda)\mathrm{e}^{-2\mathrm{i} t\theta(\lambda)}\\ \mathbf{0} & \mathbf{0}\end{pmatrix}, & \lambda\in\Omega_3,\\[10pt]
      \begin{pmatrix}0 & \mathbf{0}\\ \bar{\partial}\pmb{\mathcal{R}}_4(\lambda)\mathrm{e}^{2\mathrm{i} t\theta(\lambda)} & \mathbf{0}\end{pmatrix}, & \lambda\in\Omega_4,\\[10pt]
      \begin{pmatrix}0 & \bar{\partial}\pmb{\mathcal{R}}_6(\lambda)\mathrm{e}^{-2\mathrm{i} t\theta(\lambda)}\\ \mathbf{0} & \mathbf{0}\end{pmatrix}, & \lambda\in\Omega_6,\\[10pt]
      \mathbf{0}, & \lambda\in\Omega_2\cup\Omega_5.
    \end{cases}
  \end{equation}
  
  \item Asymptotic condition: 
    $\M^{(2)}(\lambda)=\mathbb{I}_{3}+\mathcal{O}\left(\lambda^{-1}\right)$, as $\lambda\rightarrow\infty$.
  \end{enumerate} 
\end{rhp}

\subsection{Reducing the $\bar{\partial}$-RHP to a pure $\bar{\partial}$-problem}
In this part, we aim to construct a solution $\M^{(3)}(\lambda)$ of RHP \ref{rhp-without-dbar}, derived from 
the $\bar{\partial}$-RHP \ref{dbar-rhp} by dropping the $\bar{\partial}$-component. 

\begin{rhp}\label{rhp-without-dbar}
  Find a matrix function $\M^{(3)}(\lambda)=\M^{(3)}(\lambda;x,t)$ with the following properties: 
  \begin{enumerate}
    \item Analyticity: $\M^{(3)}(\lambda)$ is analytic in $\mathbb{C}\setminus\Sigma$.
    
    \item Jump condition: $\M^{(3)}(\lambda)$ has continuous boundary 
    values $\M^{(3)}_\pm(\lambda)$ on $\Sigma$ given by
    \begin{equation}
      \M^{(3)}_\pm(\lambda)=\lim_{\varepsilon \downarrow 0} \M^{(3)}(\lambda\pm\mathrm{i}\varepsilon), 
    \end{equation}
    satisfying $\M^{(3)}_+(\lambda)=\M^{(3)}_-(\lambda)\V^{(3)}(\lambda)$, where 
    $\V^{(3)}(\lambda)=\V^{(2)}(\lambda)$.
    
    \item Asymptotic condition: 
    $\M^{(3)}(\lambda)=\mathbb{I}_{3}+\mathcal{O}(\lambda^{-1})$, as $\lambda\rightarrow\infty$.
  \end{enumerate} 
\end{rhp}

\noindent
If $\M^{(3)}(\lambda)$ exists, we can reduce the $\bar{\partial}$-RHP 
\ref{dbar-rhp} to a pure $\bar{\partial}$-problem.

\begin{lemma}
  Suppose that $\M^{(3)}(\lambda)$ is a solution of RHP \ref{rhp-without-dbar}. Then the unknown matrix 
  function 
  \begin{equation}\label{transform3}
  \M^{(4)}(\lambda)=\M^{(2)}(\lambda)\M^{(3)}(\lambda)^{-1}
  \end{equation} 
  satisfies the following pure $\bar{\partial}$-problem.
  
  \begin{rhp}\label{pure-dbar}
    (Pure $\bar{\partial}$-problem) Find a matrix function $\M^{(4)}(\lambda)=\M^{(4)}(\lambda;x,t)$ with the following properties: 
    \begin{enumerate}
      \item $\bar{\partial}$-condition: $\M^{(4)}(\lambda)$ is continuous with piecewise 
      $\bar{\partial}$-derivatives in $\mathbb{C}\setminus\Sigma$ given by
      \begin{equation}\label{def-V4}
        \bar{\partial}\M^{(4)}(\lambda)=\M^{(4)}(\lambda)\V^{(4)}(\lambda),
      \end{equation}
      where $\V^{(4)}(\lambda):=\M^{(3)}(\lambda)\bar{\partial}\pmb{\mathcal{R}}^{(2)}(\lambda)\M^{(3)}(\lambda)^{-1}$, 
      and $\bar{\partial}\pmb{\mathcal{R}}^{(2)}(\lambda)$ is defined in equation \eqref{dbar-R}.
      
      \item Asymptotic condition: 
      $\M^{(4)}(\lambda)=\mathbb{I}_{3}+\mathcal{O}(\lambda^{-1})$, as $\lambda\rightarrow\infty$.
    \end{enumerate} 
  \end{rhp}
\end{lemma}
  
\begin{proof}
  Both $\M^{(2)}(\lambda)$ and $\M^{(3)}(\lambda)$ approach the identity matrix as $\lambda$ tends to infinity. 
  It follows from the definition $\M^{(4)}(\lambda)=\M^{(2)}(\lambda)\M^{(3)}(\lambda)^{-1}$ 
  that $\M^{(4)}(\lambda)$ satisfies the asymptotic condition in RHP \ref{pure-dbar}.
  Since both $\M^{(2)}(\lambda)$ and $\M^{(3)}(\lambda)$ share the same jump condition 
  on the contour $\Sigma$, we deduce
  \begin{equation*}
    \M^{(4)}_+(\lambda)=\M^{(2)}_+(\lambda)\M^{(3)}_+(\lambda)^{-1}
    =\M^{(2)}_-(\lambda)\V^{(2)}(\lambda)\V^{(3)}(\lambda)^{-1}
    \M^{(3)}_-(\lambda)^{-1}=\M^{(4)}_-(\lambda).
  \end{equation*}
This implies that $\M^{(4)}(\lambda)$ has no jump across $\Sigma$. 
Using the fact that $\M^{(3)}(\lambda)$ has zero $\bar{\partial}$-derivative 
(since it is analytic outside $\Sigma$), we can derive the $\bar{\partial}$-condition \eqref{def-V4} as follows:
\begin{equation*}
  \bar{\partial}\M^{(4)}(\lambda)=\bar{\partial}\M^{(2)}(\lambda)\M^{(3)}(\lambda)^{-1}=
  \M^{(2)}(\lambda)\bar{\partial}\pmb{\mathcal{R}}^{(2)}(\lambda)\M^{(3)}(\lambda)^{-1}=
  \M^{(4)}(\lambda)\V^{(4)}(\lambda).
\end{equation*}
\end{proof}

Next, we construct the solution $\M^{(3)}(\lambda)$ piecewise in two distinct regions:
\begin{equation}\label{construct-M3}
  \M^{(3)}(\lambda)=\begin{cases}
    \E(\lambda), & |\lambda-\xi|>\rho/2,\\
    \E(\lambda)\M^{(5)}(\lambda), & |\lambda-\xi|<\rho/2.
  \end{cases}
\end{equation}
Here, $\M^{(5)}(\lambda)$ is a solvable local model which will be 
determined by the asymptotic behavior of $\V^{(3)}(\lambda)$ as $\lambda\to\xi$, and 
the error function $\E(\lambda)$ is utilized to 
describe the asymptotic error between 
the unknown solution $\M^{(3)}(\lambda)$ and the solvable model $\M^{(5)}(\lambda)$. 
It is necessary to demonstrate the existence of $\E(\lambda)$ and 
establish an asymptotic bound for it. We seek to obtain the 
asymptotic expansion of $\M^{(3)}(\lambda)$ through 
the local model $\M^{(5)}(\lambda)$ and the error function $\E(\lambda)$. 

In order to determine the local model $\M^{(5)}(\lambda)$, 
it suffices to analyze the behavior of $\V^{(3)}(\lambda)$ 
as $\lambda\to\xi$. Since the jump $\V^{(3)}(\lambda)$ contains 
$\pmb{\delta}(\lambda)$, which lacks a closed form, we must first establish 
an estimate for the difference between $\mathbf{R}(\xi)\pmb{\delta}(\lambda)$ and 
$\mathbf{R}(\xi)\det\pmb{\delta}(\lambda)$. 
Using part (IV) of Proposition \ref{proposition-delta}, we can then derive  
the behavior of $\mathbf{R}(\xi)\pmb{\delta}(\lambda)$ as 
$\lambda\to\xi$, which is summarized in the following lemma. 

\begin{lemma}\label{estimate-R-delta}
  As $\lambda\to\xi$ along the contour $\Sigma$, we have the following 
  estimates:
  \begin{equation}
    \begin{aligned}
      |\mathbf{R}(\xi)\pmb{\delta}(\lambda) - T_{0}(\xi)(\lambda-\xi)^{\mathrm{i} \kappa}\pmb{A}(\xi)|
      &\lesssim |\lambda-\xi|^{1/2}, \quad && \lambda\in\Sigma_2\cup\Sigma_4,\\
      |\pmb{\delta}^{-1}(\lambda)\mathbf{R}^\dagger(\xi) - T_{0}^{-1}(\xi)(\lambda-\xi)^{-\mathrm{i} \kappa}\pmb{A}^\dagger(\xi)|
      &\lesssim |\lambda-\xi|^{1/2}, \quad && \lambda\in\Sigma_1\cup\Sigma_3,
    \end{aligned}
  \end{equation}
  where $\pmb{A}(\xi)$ is a constant matrix related to $\xi$.
\end{lemma}
\begin{proof}
  For $\lambda$ on the contours $\Sigma_2\cup\Sigma_4$, 
  we first provide the estimate for 
  $\tilde{\pmb{\delta}}(\lambda)=\mathbf{R}(\xi)(\pmb{\delta}(\lambda)-\det\pmb{\delta}(\lambda))$. 
  It follows from RHP \ref{RHP-delta} and part (II) of Proposition \ref{proposition-delta} 
  that $\tilde{\pmb{\delta}}(\lambda)$ satisfies the following RHP:
  \begin{equation}
    \begin{cases}
      \tilde{\pmb{\delta}}_+(\lambda)=\tilde{\pmb{\delta}}_-(\lambda)
      (1+|\mathbf{R}(\lambda)|^2)+\pmb{f}(\lambda)\pmb{\delta}_-(\lambda),
      & \quad \lambda\in(-\infty,\xi),\\
      \tilde{\pmb{\delta}}(\lambda)\to 0,
      & \quad \lambda\to\infty,
    \end{cases}
  \end{equation}
  where 
  $\pmb{f}(\lambda)=\mathbf{R}(\xi)(\mathbf{R}^\dagger\mathbf{R}-\mathbf{R}\mathbf{R}^\dagger\mathbb{I}_2)(\lambda)$. 
  The solution $\tilde{\pmb{\delta}}(\lambda)$ for the above non-homogeneous RHP can be expressed as
  \begin{equation}\label{expression-tilde-delta}
    \begin{aligned}
      \tilde{\pmb{\delta}}(\lambda) &= \frac{X(\lambda)}{2\pi \mathrm{i}}\int_{-\infty}^{\xi}
      \frac{\pmb{f}(s)\pmb{\delta}_-(s)}{X_+(s)(s-\lambda)}\,\mathrm{d}s,\\
      X(\lambda) &= \exp\left\{\frac{1}{2\pi \mathrm{i}}\int_{-\infty}^{\xi}
      \frac{\ln(1+|\mathbf{R}(s)|^2)}{s-\lambda}\,\mathrm{d}s\right\}.
    \end{aligned}
  \end{equation}
 Moreover, $\tilde{\pmb{\delta}}(\lambda)$ can be rewritten using the Cauchy 
 projection operator:
 \begin{equation}
  \tilde{\pmb{\delta}}(\lambda)=X(\lambda)C_{\Gamma}\pmb{g}(\lambda),\quad 
  \pmb{g}(s)=X^{-1}_+(s)\pmb{f}(s)\pmb{\delta}_-(s).
 \end{equation}
 Here, the contour $\Gamma$ denotes $(-\infty,\xi)$, and the 
 Cauchy projection operator $C_{\Gamma}\pmb{f}(\lambda)$ 
 is defined by 
\begin{equation}
  \begin{aligned}
    C_{\Gamma}\pmb{f}(\lambda) &=
  \frac{1}{2\pi \mathrm{i}}\int_{\Gamma}\frac{\pmb{f}(s)}
  {s-\lambda}\,\mathrm{d}s,
  \quad \lambda\in\mathbb{C}\setminus\Gamma,\\
  C^\pm_{\Gamma}\pmb{f}(\lambda) &= \lim_{z\to\lambda_\pm}
    \frac{1}{2\pi \mathrm{i}}\int_{\Gamma}\frac{\pmb{f}(s)}
    {s-z}\,\mathrm{d}s,\quad \lambda\in\Gamma.
  \end{aligned}
\end{equation}

Since the reflection coefficient 
$\mathbf{R}(\lambda)$ belongs to the Sobolev space 
$H^{1}(\mathbb{R})$, we have 
$\pmb{f}(s)\in H^{1}(\mathbb{R})$. Then, for 
$\lambda=\xi+r\mathrm{e}^{\mathrm{i}\phi}$ with $r>0$ and $\phi=\frac{3\pi}{4}$, 
\begin{equation}\label{estimate-C-Gamma-g}
        |C_\Gamma\pmb{g}(\lambda)-C_\Gamma\pmb{g}(\xi)| \lesssim
        \int_{-\infty}^{\xi}\frac{|\lambda-\xi|
          |\pmb{f}(s)|
        }{|s-\lambda||s-\xi|}\,\mathrm{d}s
        \lesssim r\int_{-\infty}^{\xi}\frac{1
        }{|s-\xi-r\mathrm{e}^{\mathrm{i}\phi}||s-\xi|^{1/2}}\,\mathrm{d}s
        \lesssim r \int_{0}^{\infty}\frac{1
        }{|\tau+r\mathrm{e}^{\mathrm{i}\phi}|
        \tau^{1/2}}\,\mathrm{d}\tau.
    \end{equation}
    Together with the following fact 
    \begin{equation}
      |\tau+r\mathrm{e}^{\mathrm{i}\phi}|^2
      =\tau^2+r^2+2\tau r\cos\phi
      =\frac{1+\cos\phi}{2}(\tau+r)^2
      +\frac{1-\cos\phi}{2}(\tau-r)^2,
    \end{equation}
    we can derive that 
    \begin{equation}
      \int_{0}^{\infty}\frac{1}{|\tau+r\mathrm{e}^{\mathrm{i}\phi}|\tau^{1/2}}\,\mathrm{d}\tau
        \lesssim\int_{0}^{\infty}\frac{1}{\cos\frac{\phi}{2}(\tau+r)\tau^{1/2}}\,\mathrm{d}\tau
        \lesssim \frac{\pi}{\cos\frac{\phi}{2}}r^{-1/2}.
    \end{equation}
    Hence, as $\lambda\rightarrow\xi$ along the contour $\Sigma_2$, we conclude that 
    \begin{equation}
      |C_\Gamma\pmb{g}(\lambda)-C_\Gamma\pmb{g}(\xi)|\lesssim 
        r^{1/2}=|\lambda-\xi|^{1/2}.
    \end{equation}
    The case on the contour $\Sigma_4$ can be 
    estimated in a similar way. 
  Then it follows from equation \eqref{estimate-C-Gamma-g} and 
  part (V) of Proposition \ref{proposition-delta} that 
  \begin{equation}
    |\tilde{\pmb{\delta}}(\lambda)-
    T_{0}(\xi)(\lambda-\xi)^{\mathrm{i} \kappa}C_\Gamma\pmb{g}(\xi)|
    \lesssim |\lambda-\xi|^{1/2},
  \end{equation}
  and hence 
  \begin{equation}
    |\mathbf{R}(\xi)\pmb{\delta}(\lambda)
    -T_{0}(\xi)(\lambda-\xi)^{\mathrm{i} \kappa}
    (\mathbf{R}(\xi)+C_\Gamma\pmb{g}(\xi))|
    \lesssim |\lambda-\xi|^{1/2}.
  \end{equation}
  This implies that $\pmb{A}(\xi)=\mathbf{R}(\xi)+C_\Gamma\pmb{g}(\xi)$. 
  Then we can utilize the symmetry of $\pmb{\delta}(\lambda)$ to 
  provide the estimate for 
  $\pmb{\delta}^{-1}(\lambda)\mathbf{R}^\dagger(\xi)$ on the contours 
  $\Sigma_1\cup\Sigma_3$. It follows from part 
  (III) of Proposition \ref{proposition-delta} that 
  \begin{equation}
    \pmb{\delta}^{-1}(\lambda)\mathbf{R}^\dagger(\xi)
    =\pmb{\delta}^\dagger(\lambda^*)\mathbf{R}^\dagger(\xi)
    =[\mathbf{R}(\xi)\pmb{\delta}(\lambda^*)]^\dagger,
  \end{equation}
  and then 
  \begin{equation}
    \begin{aligned}
      |\pmb{\delta}^{-1}(\lambda)\mathbf{R}^\dagger(\xi)
      -T_{0}^{-1}(\xi)(\lambda-\xi)^{-\mathrm{i} \kappa}\pmb{A}^\dagger(\xi)|
      &= |[\mathbf{R}(\xi)\pmb{\delta}(\lambda^*)-
      T_{0}(\xi)(\lambda^*-\xi)^{\mathrm{i} \kappa}\pmb{A}(\xi)]^\dagger|\\
      &\lesssim |\lambda-\xi|^{1/2},
      \quad \lambda\in\Sigma_1\cup\Sigma_3.
    \end{aligned}
  \end{equation}
  Consequently, we complete the proof.

\end{proof}
By Lemma \ref{estimate-R-delta}, we can construct the following model RHP such that 
its solution $\M^{(5)}(\lambda)$ is a good approximation to 
$\M^{(3)}(\lambda)$ in the neighborhood of the stationary point $\xi$.

\begin{rhp}\label{solvable-model}
  Find a matrix function $\M^{(5)}(\lambda)=\M^{(5)}(\lambda;x,t)$ with the following properties: 
  \begin{enumerate}
    \item Analyticity: $\M^{(5)}(\lambda)$ is analytic in $\mathbb{C}\setminus\Sigma$.
    
    \item Jump condition: $\M^{(5)}(\lambda)$ has continuous boundary 
    values $\M^{(5)}_\pm(\lambda)$ on $\Sigma$ given by
    \begin{equation}
      \M^{(5)}_\pm(\lambda)=\lim_{\varepsilon \downarrow 0} \M^{(5)}(\lambda\pm\mathrm{i}\varepsilon), 
    \end{equation}
    satisfying $\M^{(5)}_+(\lambda)=\M^{(5)}_-(\lambda)\V^{(5)}(\lambda)$, where
    \begin{equation}
      \V^{(5)}(\lambda)=
      \begin{cases}
       \begin{pmatrix}
         1 & \mathbf{0}\\
         -\pmb{A}^\dagger(\xi) T^{-2}_{0}(\xi)(\lambda-\xi)^{-2\mathrm{i} \kappa}\mathrm{e}^{2\mathrm{i} t\theta} & \mathbb{I}_2
       \end{pmatrix},
       & \lambda\in\Sigma_1,\\[15pt]
       \begin{pmatrix}
         1 & -\frac{\pmb{A}(\xi)}{1+|\mathbf{R}(\xi)|^{2}}
         T^2_{0}(\xi)(\lambda-\xi)^{2\mathrm{i} \kappa}\mathrm{e}^{-2\mathrm{i} t\theta}\\
         \mathbf{0} & \mathbb{I}_2
       \end{pmatrix},
       & \lambda\in\Sigma_2,\\[15pt]
       \begin{pmatrix}
         1 & \mathbf{0}\\
         -\frac{\pmb{A}^\dagger(\xi)}
         {1+|\mathbf{R}(\xi)|^{2}}T^{-2}_{0}(\xi)(\lambda-\xi)^{-2\mathrm{i} \kappa}\mathrm{e}^{2\mathrm{i} t\theta} & \mathbb{I}_2
       \end{pmatrix},
       & \lambda\in\Sigma_3,\\[15pt]
       \begin{pmatrix}
         1 & -\pmb{A}(\xi)T^2_{0}(\xi)(\lambda-\xi)^{2\mathrm{i} \kappa}\mathrm{e}^{-2\mathrm{i} t\theta}\\
         \mathbf{0} & \mathbb{I}_2
       \end{pmatrix},
       & \lambda\in\Sigma_4.
      \end{cases}
     \end{equation}
     
     \item Asymptotic condition: 
    $\M^{(5)}(\lambda)=\mathbb{I}_{3}+\mathcal{O}(\lambda^{-1})$, as $\lambda\rightarrow\infty$.
  \end{enumerate} 
\end{rhp}

\begin{figure}[h]
\centering
\begin{tikzpicture}
  
  \draw[dashed] (-5, 0) -- (5, 0) ;
  \draw[red,-,thick] (-1.5, 1.5) -- (-1.0606, 1.0606) ;
  \draw[red,thick] (-1.5, -1.5) -- (-1.0606, -1.0606) ;
  \draw[red,->,thick] (1.0606, -1.0606) -- (1.5, -1.5);
  \draw[->,thick] (-1.0606, 1.0606) -- (-0.3,0.3);
  \draw[->,thick] (-0.3,0.3)--(0.3,-0.3);
  \draw[-,thick] (0.3,-0.3)--(1.0606, -1.0606) ;
  \draw[->,thick] (-1.0606, -1.0606) -- (-0.3,-0.3);
  \draw[->,thick] (-0.3,-0.3)--(0.3,0.3);
  \draw[-,thick] (0.3,0.3)--(1.0606, 1.0606) ;

  \draw[red,->,thick] (1.0606, 1.0606) --  (1.5, 1.5);
  \draw[red, ->,thick] (-3, 3) -- (-1.5, 1.5) ;
  \draw[red,->,thick] (-3, -3) -- (-1.5, -1.5) ;
  \draw[red,-,thick] (1.5, 1.5) -- (3, 3) ;
  \draw[red,-,thick] (1.5, -1.5) -- (3, -3) ;

  \node at (4.7, -0.3)  {${\rm Re}\lambda$};
  \node at (0, -0.5)  {$\xi$};
  \draw[blue, ->,thick] (1.5,0) arc[start angle=360, end angle=0, radius=1.5];
  \node at (-2, 0.3)  {$\partial\mathcal{U}_\xi$};
  \node at (-2.1, 2.7)  {$\Sigma\setminus\mathcal{U}_\xi$};
  \node at (0, 0.7)  {$\Sigma\cap\mathcal{U}_\xi$};
\end{tikzpicture}
\caption{\small{The jump contour $\Sigma_E$. The red contour represents 
$\Sigma\setminus\mathcal{U}_\xi$; The black contour represents 
$\Sigma\cap\mathcal{U}_\xi$; The blue contour represents 
$\partial\mathcal{U}_\xi$.}
}
\label{jump-contour-VE}
\end{figure}
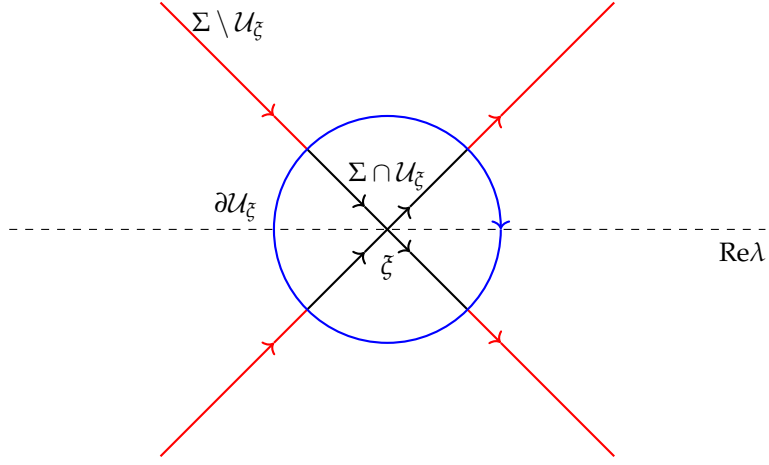

To solve the local model RHP \ref{solvable-model}, we introduce 
the following scaling transformation:
\begin{equation}\label{variable-change}
  y=\phi(\lambda)=2\sqrt{t}(\lambda-\xi).
\end{equation}
Then RHP \ref{solvable-model} can be transformed into a new RHP whose jumps 
match those of the parabolic cylinder model problem 
\cite{HYBLLMZXE2025}. It is easy to verify that $\M^{(6)}(y)=\M^{(5)}(\phi^{-1}(y))$ 
satisfies the following RHP.

\begin{rhp}
  Find a matrix function $\M^{(6)}(y)=\M^{(6)}(y;x,t)$ with the following properties: 
  \begin{enumerate}
    \item Analyticity: $\M^{(6)}(y)$ is analytic in 
    $\mathbb{C}\setminus\Sigma'$, where  
    $\Sigma'=\bigcup_{j=1}^{4}\Sigma'_j$ and 
    $\Sigma'_j=\mathrm{e}^{\frac{(2j-1)\pi\mathrm{i}}{4}}\mathbb{R}_+$.
    
    \item Jump condition: $\M^{(6)}(y)$ has continuous boundary 
    values $\M^{(6)}_\pm(y)$ on $\Sigma'$ given by
    \begin{equation}
      \M^{(6)}_\pm(y)=\lim_{\varepsilon \downarrow 0} \M^{(6)}(y\pm\mathrm{i}\varepsilon), 
    \end{equation}
    satisfying $\M^{(6)}_+(y)=\M^{(6)}_-(y)\V^{(6)}(y)$, where
    \begin{equation}
      \V^{(6)}(y)=
      \begin{cases}
       \begin{pmatrix}
         1 & \mathbf{0}\\
         -\pmb{A}^\dagger(\xi) \eta^{-2}y^{-2\mathrm{i}\kappa} \mathrm{e}^{\frac{1}{2}\mathrm{i} y^2} & \mathbb{I}_2
       \end{pmatrix},
       & y\in\Sigma_1',\\[15pt]
       \begin{pmatrix}
         1 & -\frac{\pmb{A}(\xi)}{1+|\mathbf{R}(\xi)|^{2}}
         \eta^{2}y^{2\mathrm{i}\kappa} \mathrm{e}^{-\frac{1}{2}\mathrm{i} y^2}\\
         \mathbf{0} & \mathbb{I}_2
       \end{pmatrix},
       & y\in\Sigma_2',\\[15pt]
       \begin{pmatrix}
         1 & \mathbf{0}\\
         -\frac{\pmb{A}^\dagger(\xi)}
         {1+|\mathbf{R}(\xi)|^{2}}\eta^{-2}y^{-2\mathrm{i}\kappa} \mathrm{e}^{\frac{1}{2}\mathrm{i} y^2} & \mathbb{I}_2
       \end{pmatrix},
       & y\in\Sigma_3',\\[15pt]
       \begin{pmatrix}
         1 & -\pmb{A}(\xi)\eta^{2}y^{2\mathrm{i}\kappa} \mathrm{e}^{-\frac{1}{2}\mathrm{i} y^2}\\
         \mathbf{0} & \mathbb{I}_2
       \end{pmatrix},
       & y\in\Sigma_4',
      \end{cases}
     \end{equation}
     where 
     \begin{equation}\label{scalar-eta}
      \eta=(2\sqrt{t})^{-\mathrm{i}\kappa(\xi)}\mathrm{e}^{\mathrm{i} t\xi^2}T_0(\xi).
     \end{equation}
     
     \item Asymptotic condition: 
    $\M^{(6)}(y)=\mathbb{I}_{3}+\mathcal{O}(y^{-1})$, as $y\rightarrow\infty$.
  \end{enumerate} 
\end{rhp}
It is clear that $\M^{(6)}(y)$ can be solved via the Weber equation and expressed in terms of 
the parabolic cylinder functions $D_a(\cdot)$. Moreover, the expansion of $\M^{(6)}(y)$ 
is given by \cite{HYBLLMZXE2025}:
\begin{equation}\label{residue-M6}
  \begin{aligned}
    \M^{(6)}(y) &= \mathbb{I}_3+\frac{\M^{(6)}_1}{y}+\mathcal{O}(y^{-2}),
    \quad y\to\infty,\\
    \M^{(6)}_1 &= \begin{pmatrix}
      0 & -\mathrm{i}\eta^2\pmb{\beta}_{12}\\
      \mathrm{i}\eta^{-2}\pmb{\beta}_{21} & \mathbf{0}
    \end{pmatrix},\quad
    \pmb{\beta}_{12}=\frac{\mathrm{e}^{\frac{\pi\kappa}{2}+\frac{\pi \mathrm{i}}{4}}\kappa
      \Gamma(\mathrm{i}\kappa)}{\sqrt{2\pi}}\pmb{A}(\xi),\quad \pmb{\beta}_{21}=-\pmb{\beta}_{12}^\dagger,
  \end{aligned}
\end{equation}
where $\Gamma(\cdot)$ is the Gamma function.
It follows from the variable change \eqref{variable-change} and equation \eqref{residue-M6} that 
we can derive the following two expansions of $\M^{(5)}(\lambda)$:
\begin{equation}\label{expansion-M5-lambda}
    \M^{(5)}(\lambda)=\mathbb{I}_3+\frac{\M^{(5)}_1}{\lambda}+\mathcal{O}(\lambda^{-2}),
    \quad \M^{(5)}_1=\frac{1}{2\sqrt{t}}\M^{(6)}_1,\quad \lambda\to\infty,
\end{equation}
\begin{equation}\label{expansion-M5-t}
  \M^{(5)}(\lambda)=\mathbb{I}_3+\frac{\M^{(6)}_1}{2\sqrt{t}(\lambda-\xi)}+\mathcal{O}(t^{-1}),
    \quad t\to\infty.
\end{equation}

We now consider the existence of the error function $\E(\lambda)$. 
Let $\mathcal{U}_\xi$ denote the open neighborhood of the stationary point $\xi$, defined as
\begin{equation}
  \mathcal{U}_\xi=\{\lambda : |\lambda-\xi|<\rho/2\}.
\end{equation}
By equation \eqref{construct-M3} and RHP \ref{rhp-without-dbar}, 
it is straightforward to show that $\E(\lambda)$ must satisfy the following RHP.

\begin{rhp}\label{rhp-error}
  Find a matrix function $\E(\lambda)=\E(\lambda;x,t)$ with the following properties: 
  \begin{enumerate}
    \item Analyticity: $\E(\lambda)$ is analytic in 
    $\mathbb{C}\setminus\Sigma_E$, where  
    $\Sigma_E=\Sigma\cup\partial\mathcal{U}_\xi$, 
    as shown in Figure \ref{jump-contour-VE}.
    
    \item Jump condition: $\E(\lambda)$ has continuous boundary 
    values $\E_\pm(\lambda)$ on $\Sigma_E$ given by
    \begin{equation}
      \E_\pm(\lambda)=\lim_{\varepsilon \downarrow 0} \E(\lambda\pm\mathrm{i}\varepsilon), 
    \end{equation}
    satisfying $\E_+(\lambda)=\E_-(\lambda)\V_E(\lambda)$, where
    \begin{equation}\label{def-VE}
      \V_E(\lambda)=\begin{cases}
        \V^{(3)}(\lambda),
        & \lambda\in\Sigma\setminus\mathcal{U}_\xi,\\[8pt]
        \M^{(5)}_-(\lambda)\V^{(3)}(\lambda)\V^{(5)}(\lambda)^{-1}\M^{(5)}_-(\lambda)^{-1},
        & \lambda\in\Sigma\cap\mathcal{U}_\xi,\\[8pt]
        \M^{(5)}(\lambda),
        & \lambda\in\partial\mathcal{U}_\xi,
      \end{cases}
    \end{equation}
    and the boundary $\partial\mathcal{U}_\xi$ is oriented clockwise.
    
    \item Asymptotic condition: 
    $\E(\lambda)=\mathbb{I}_{3}+\mathcal{O}(\lambda^{-1})$, as $\lambda\rightarrow\infty$.
  \end{enumerate} 
\end{rhp}

To show the existence of $\E(\lambda)$, 
we need to provide a proper estimate 
for $\V_E(\lambda)-\mathbb{I}_3$ along three different contours as shown in Figure \ref{jump-contour-VE}, 
ensuring that RHP \ref{rhp-error} is a small-norm RHP. Initially, 
we focus on the most complicated case along the contour $\Sigma\cap\mathcal{U}_\xi$.
Since $\|\M^{(5)}_-(\lambda)^{\pm 1}\|_{L^\infty(\Sigma\cap\mathcal{U}_\xi)}$ 
and $\|\V^{(5)}(\lambda)^{\pm 1}\|_{L^\infty(\Sigma\cap\mathcal{U}_\xi)}$ 
are finite, it suffices to estimate $|\V^{(3)}(\lambda)-\V^{(5)}(\lambda)|$ along the contour 
$\Sigma\cap\mathcal{U}_\xi$, using Lemma \ref{estimate-R-delta}. 
In the following lemma, we establish a uniformly vanishing 
bound for $\V_E(\lambda)-\mathbb{I}_3$.

\begin{lemma}\label{estimate-VE}
  For sufficiently large $t>0$, we establish the following 
  estimates for the jump matrix $\V_E(\lambda)$ of RHP \ref{rhp-error}:
  \begin{equation}\label{bdd-VE}
    \begin{aligned}
      \|\V_E(\lambda)-\mathbb{I}_3\|_{L^p(\partial\mathcal{U}_\xi)} &\lesssim t^{-1/2},\quad 1\le p\le\infty,\\
      \|\V_E(\lambda)-\mathbb{I}_3\|_{L^p(\Sigma)} &\lesssim t^{-1/4-1/2p},\quad 1\le p\le\infty.
    \end{aligned}
  \end{equation}
\end{lemma}

\begin{proof}
  For $\lambda\in\Sigma\cap\mathcal{U}_\xi$, using Lemma \ref{estimate-R-delta}, 
  we can derive 
\begin{equation}
  |\V_E(\lambda)-\mathbb{I}_3| \lesssim
  |\V^{(3)}(\lambda)-\V^{(5)}(\lambda)|
  \lesssim |\lambda-\xi|^{1/2}
  \mathrm{e}^{-2t|\lambda-\xi|^2}
  \lesssim t^{-1/4}
  \mathrm{e}^{-t|\lambda-\xi|^2}.
\end{equation}
Here, we use the fact that $\sup_{w>0}w^{1/2}\mathrm{e}^{-w^2}\le c$.
For $\lambda\in\Sigma\setminus\mathcal{U}_\xi$, by the definition \eqref{jump-M2} of 
$\V^{(3)}(\lambda)$, the fact \eqref{boundedness-lamda-xi}, and part (IV) of 
Proposition \ref{proposition-delta}, we can deduce 
\begin{equation}
  |\V_E(\lambda)-\mathbb{I}_3| \lesssim \mathrm{e}^{-2t|\lambda-\xi|^2}
  \lesssim t^{-1/4}\mathrm{e}^{-t|\lambda-\xi|^2}.
\end{equation}
Hence, we obtain the estimate for 
$\V_E(\lambda)-\mathbb{I}_3$ on the contour $\Sigma$:
\begin{equation}
  \|\V_E(\lambda)-\mathbb{I}_3\|_{L^p(\Sigma)}
  \lesssim  t^{-1/4}\left(\int_{\Sigma}
  \mathrm{e}^{-tp|\lambda-\xi|^2}\,\mathrm{d}\lambda
  \right)^{1/p}
  \lesssim t^{-1/4-1/2p},\quad 1\le p<\infty.
\end{equation}
It follows from equation \eqref{expansion-M5-t}  
that for $\lambda\in\partial\mathcal{U}_\xi$, 
\begin{equation}
  |\V_E(\lambda)-\mathbb{I}_3| = |\M^{(5)}(\lambda)-\mathbb{I}_3|
  \lesssim t^{-1/2}.
\end{equation}
The result \eqref{bdd-VE} then follows immediately, which completes the proof.
\end{proof}

By the Beals-Coifman theorem, the solution $\E(\lambda)$ can be expressed as
\begin{equation}\label{expression-E}
  \E(\lambda)=\mathbb{I}_3+\frac{1}{2\pi \mathrm{i}}\int_{\Sigma_E}
  \frac{\pmb{\mu}_E(s)(\V_E(s)-\mathbb{I}_3)}{s-\lambda}\,\mathrm{d}s,
\end{equation}
where $\pmb{\mu}_E\in L^2(\Sigma_E) + \mathbb{I}_3$ satisfies 
\begin{equation}\label{def-mu-E}
  (1-C_{\mathbf{w}_E})\pmb{\mu}_E=\mathbb{I}_3.
\end{equation}
The singular integral operator $C_{\mathbf{w}_E}: L^2(\Sigma_E) \to L^2(\Sigma_E)$ is defined by 
\begin{equation}\label{def-Cw}
  C_{\mathbf{w}_E}\pmb{f} = C^-_{\Sigma_E}(\pmb{f}(\V_E-\mathbb{I}_3)),\quad \lambda\in\Sigma_E,
\end{equation}
where $C^\pm_{\Sigma_E}$ are the well-known Cauchy projection operators given by 
\begin{equation}
  C^\pm_{\Sigma_E}\pmb{f}(\lambda)=\lim_{z\to\lambda_\pm}
  \frac{1}{2\pi \mathrm{i}}\int_{\Sigma_E}\frac{\pmb{f}(s)}{s-z}\,\mathrm{d}s,
\end{equation} 
and the Cauchy operator $C_{\Sigma_E}$ is expressed by 
\begin{equation}
  C_{\Sigma_E}\pmb{f}(\lambda)=
  \frac{1}{2\pi \mathrm{i}}\int_{\Sigma_E}\frac{\pmb{f}(s)}{s-\lambda}\,\mathrm{d}s,
  \quad \lambda\in\mathbb{C}\setminus\Sigma_{E}.
\end{equation}
It is well known that $\|C^-_{\Sigma_E}\|_{L^2(\Sigma_E)}$ is bounded.
It follows from Lemma \ref{estimate-VE} that 
\begin{equation}
  \|C_{\mathbf{w}_E}\|_{L^2(\Sigma_E)} \lesssim 
  \|C^-_{\Sigma_E}\|_{L^2(\Sigma_E)}\|\V_E-\mathbb{I}_3\|_{L^\infty(\Sigma_E)}
  \lesssim t^{-1/4},
\end{equation}
which guarantees the existence of the resolvent operator $(1-C_{\mathbf{w}_E})^{-1}$. As a result, we 
can also ensure the existence of both $\pmb{\mu}_E$ and $\E(\lambda)$. This allows us to complete the definition of 
$\M^{(3)}(\lambda)$ given by equation \eqref{construct-M3}.

In order to reconstruct the solution $\mathbf{q}(x,t)$ for the CNLS equation \eqref{CNLS}, 
it is necessary to analyze the long-time asymptotics of $\E(\lambda)$.

\begin{prop}
  The asymptotic expansion of the solution $\E(\lambda)$ for RHP \ref{rhp-error} can be 
  expressed as follows: 
  \begin{equation}\label{long-time-expression-E}
    \E(\lambda)=\mathbb{I}_3+\frac{\E_1}{\lambda}+\mathcal{O}(\lambda^{-2}),
  \end{equation}
  where 
  \begin{equation}
    \E_1=-\frac{1}{2\pi \mathrm{i}}\int_{\Sigma_E}\pmb{\mu}_E(s)(\V_E(s)-\mathbb{I}_3)\,\mathrm{d}s.
  \end{equation}
  Furthermore, we obtain the following optimal estimates for $\E_1$ and $\E(\lambda)$ evaluated at 
  $\lambda=\lambda_i\in\mathcal{Z}$ for large $t>0$: 
  \begin{equation}\label{expression-E1}
    \begin{aligned}
      \E(\lambda_i) &= \mathbb{I}_3+t^{-1/2}\tilde{\mathbf{P}}_i+\mathcal{O}(t^{-3/4}),\\
      \E_1 &= \frac{1}{2\sqrt{t}}\begin{pmatrix}
        0 & -\mathrm{i}\eta^2\pmb{\beta}_{12}\\
        \mathrm{i}\eta^{-2}\pmb{\beta}_{21} & \mathbf{0}
      \end{pmatrix}+\mathcal{O}(t^{-3/4}),
    \end{aligned}
  \end{equation}
  where 
  \begin{equation}\label{def-tilde-P}
    \tilde{\mathbf{P}}_i=\frac{1}{2(\lambda_i-\xi)}
    \begin{pmatrix}
      0 & -\mathrm{i}\eta^2\pmb{\beta}_{12}\\
      \mathrm{i}\eta^{-2}\pmb{\beta}_{21} & \mathbf{0}
    \end{pmatrix}.
  \end{equation}
\end{prop}

\begin{proof}
  By expanding the term $(s-\lambda)^{-1}$ in a geometric series for large $\lambda$ in 
  expression \eqref{expression-E}, we derive the asymptotic expansion 
  \eqref{long-time-expression-E} of $\E(\lambda)$. 
  It follows from equations \eqref{bdd-VE} and \eqref{def-mu-E} that 
\begin{equation}
  \|\pmb{\mu}_E-\mathbb{I}_3\|_{L^2(\Sigma_E)}=
  \|(1-C_{\mathbf{w}_E})^{-1}C_{\mathbf{w}_E}\mathbb{I}_3\|_{L^2(\Sigma_E)}
  \lesssim\|\V_E-\mathbb{I}_3\|_{L^2(\Sigma_E)}\lesssim t^{-1/2},
\end{equation}
and we thus derive:
\begin{equation}
  \begin{aligned}
    \left|\E_1+\frac{1}{2\pi\mathrm{i}}\int_{\partial\mathcal{U}_\xi}(\V_E(s)-\mathbb{I}_3)\,\mathrm{d}s\right| &\lesssim
    \left|\int_{\Sigma_E}(\pmb{\mu}_E(s)-\mathbb{I}_3)(\V_E(s)-\mathbb{I}_3)\,\mathrm{d}s\right|+
    \left|\int_{\Sigma}(\V_E(s)-\mathbb{I}_3)\,\mathrm{d}s\right|\\
    &\lesssim \|\pmb{\mu}_E-\mathbb{I}_3\|_{L^2(\Sigma_E)} \|\V_E-\mathbb{I}_3\|_{L^2(\Sigma_E)}+\|\V_E-\mathbb{I}_3\|_{L^1(\Sigma)}\\
    &\lesssim t^{-1/2}t^{-1/2}+t^{-3/4}\lesssim t^{-3/4},
  \end{aligned}
\end{equation}
by Lemma \ref{estimate-VE}. 
Combining this with the expansion \eqref{expansion-M5-t} and the 
Cauchy residue theorem (recalling that 
$\partial\mathcal{U}_\xi$ is clockwise oriented), the estimate 
for $\E_1$ follows immediately. 
By equation \eqref{expression-E}, we obtain:
\begin{equation}
  \E(\lambda_i)=\mathbb{I}_3+\frac{1}{2\pi \mathrm{i}}\left(\int_{\partial\mathcal{U}_\xi}+\int_{\Sigma}
  \frac{\V_E(s)-\mathbb{I}_3}{s-\lambda_i}\,\mathrm{d}s\right)+
  \frac{1}{2\pi \mathrm{i}}\int_{\Sigma_E}
  \frac{(\pmb{\mu}_E(s)-\mathbb{I}_3)(\V_E(s)-\mathbb{I}_3)}{s-\lambda_i}\,\mathrm{d}s.
\end{equation}
Since $\V_E(s)-\mathbb{I}_3$ is not supported near $\lambda=\lambda_i$, we deduce:
\begin{equation}
  \begin{aligned}
    \left|\int_{\Sigma_E}
    \frac{(\pmb{\mu}_E(s)-\mathbb{I}_3)(\V_E(s)-\mathbb{I}_3)}{s-\lambda_i}\,\mathrm{d}s\right| &\lesssim
    \int_{\Sigma_E}|(\pmb{\mu}_E(s)-\mathbb{I}_3)(\V_E(s)-\mathbb{I}_3)|\,\mathrm{d}s \lesssim t^{-1},\\
    \left|\int_{\Sigma}\frac{\V_E(s)-\mathbb{I}_3}{s-\lambda_i}\,\mathrm{d}s\right| &\lesssim
    \int_{\Sigma}|\V_E(s)-\mathbb{I}_3|\,\mathrm{d}s \lesssim t^{-3/4}.
  \end{aligned}
\end{equation}
For $s\in\partial\mathcal{U}_\xi$, it follows from equations 
\eqref{expansion-M5-lambda} and \eqref{def-VE} that:
\begin{equation}
  \frac{1}{2\pi \mathrm{i}}\int_{\partial\mathcal{U}_\xi}
  \frac{\V_E(s)-\mathbb{I}_3}{s-\lambda_i}\,\mathrm{d}s = \frac{1}{2\pi \mathrm{i}}\int_{\partial\mathcal{U}_\xi}
  \frac{\M^{(5)}(s)-\mathbb{I}_3}{s-\lambda_i}\,\mathrm{d}s = 
  \frac{\M^{(6)}_1}{2\sqrt{t}(\lambda_i-\xi)}.
\end{equation}
This completes the proof.
\end{proof}

\subsection{Analysis of the pure $\bar{\partial}$-problem}
In this subsection, we primarily focus on the existence of the solution 
$\M^{(4)}(\lambda)$ for the pure $\bar{\partial}$-problem \ref{pure-dbar} and provide 
an optimal estimate for $\M^{(4)}(\lambda)$. It is well known that 
the $\bar{\partial}$-problem \ref{pure-dbar} is equivalent to the following integral equation:
\begin{equation}\label{M4-integral-form}
  \M^{(4)}(\lambda)=\mathbb{I}_3-\frac{1}{\pi}\iint_{\mathbb{C}}
  \frac{\bar{\partial}\M^{(4)}(s)}{s-\lambda}\,\mathrm{d} A(s)
  =\mathbb{I}_3-\frac{1}{\pi}\iint_{\mathbb{C}}\frac{\M^{(4)}(s)\V^{(4)}(s)}{s-\lambda}\,\mathrm{d} A(s),
\end{equation}
where $\mathrm{d} A(s)$ is the Lebesgue measure on the complex plane. Furthermore, we can rewrite the 
integral \eqref{M4-integral-form} using the solid Cauchy operator $S$:
\begin{equation}\label{def-S}
  (1-S)[\M^{(4)}](\lambda)=\mathbb{I}_3,\quad 
  S[\pmb{f}](\lambda)=-\frac{1}{\pi}\iint_{\mathbb{C}}\frac
  {\pmb{f}(s)\V^{(4)}(s)}{s-\lambda}\,\mathrm{d} A(s).
\end{equation}

\begin{lemma}\label{existence-S}
  For sufficiently large $t>0$, the operator $S$ defined in 
  equation \eqref{def-S} satisfies the inequality
  \begin{equation}
    \|S\|_{L^\infty\to L^\infty}\lesssim t^{-1/4}.
  \end{equation}
\end{lemma}

Lemma \ref{existence-S} implies that for sufficiently large $t>0$, 
the operator $S$ is a small-norm operator, guaranteeing the existence of 
both the resolvent operator $(1-S)^{-1}$ and the solution $\M^{(4)}(\lambda)$. 
By expanding the term $(s-\lambda)^{-1}$ in a geometric series for large $\lambda$ in 
expression \eqref{M4-integral-form}, we derive the asymptotic expansion for $\M^{(4)}(\lambda)$:
\begin{equation}
  \M^{(4)}(\lambda)=\mathbb{I}_3+\frac{\M^{(4)}_1}{\lambda}+\mathcal{O}(\lambda^{-2}),
\end{equation}
where 
\begin{equation}\label{residue-M4}
  \M^{(4)}_1=\frac{1}{\pi}\iint_{\mathbb{C}}
  \M^{(4)}(s)\V^{(4)}(s)\,\mathrm{d} A(s).
\end{equation}
To construct the solution $\mathbf{q}(x,t)$ for the CNLS equation \eqref{CNLS}, it is necessary 
to analyze the long-time asymptotic behavior of $\M^{(4)}_1$ in the following lemma.

\begin{lemma}\label{bdd-residue-M4}
  For sufficiently large $t>0$, the following estimate holds for $\M^{(4)}_1$: 
  \begin{equation}
    |\M^{(4)}_1| \lesssim t^{-3/4}.
  \end{equation}
\end{lemma}

The proofs of Lemma \ref{existence-S} and Lemma \ref{bdd-residue-M4} follow 
a standardized procedure \cite{McLaughlin_2018}, which is based on the relative estimate for 
the $\bar{\partial}$-derivative of the extension $\pmb{\mathcal{R}}^{(2)}(\lambda)$ provided in equation 
\eqref{condition-R}. 
We provide the proofs of Lemma \ref{existence-S} and 
Lemma \ref{bdd-residue-M4} in Appendix A 
for the sake of completeness.

\subsection{Reconstruction of the solution $\mathbf{q}(x,t)$}
By inverting all the transformations \eqref{transform1}, \eqref{transform2}, and 
\eqref{transform3} carried out previously, the solution $\M(\lambda)$ of 
RHP \ref{initial-rhp-without-role} is given by:
\begin{equation}\label{inverse-transformation}
  \M(\lambda)=\M^{(4)}(\lambda)\M^{(3)}(\lambda)
  \pmb{\mathcal{R}}^{(2)}(\lambda)^{-1}\pmb{\Delta}(\lambda).
\end{equation}
The asymptotic expansion of $\M(\lambda)$ can be expressed as:
\begin{equation}
  \M(\lambda)=\mathbb{I}_3+\frac{\M_1}{\lambda}+\mathcal{O}(\lambda^{-2}),
  \quad \lambda\to\infty.
\end{equation}
Based on the reconstruction formula \eqref{better-recovery-formula}, 
the long-time asymptotics of the solution $\mathbf{q}(x,t)$ for the CNLS 
equation \eqref{CNLS} is determined by the long-time asymptotics of 
$(\M_1)_{12}$ and $\M(\lambda_i)$. 
Since the discrete eigenvalues $\lambda_i$ are 
located outside the neighborhood $\mathcal{U}_\xi$, it suffices to analyze 
the long-time asymptotics of $\M(\lambda)$ 
for $\lambda\in\mathbb{C}\setminus\mathcal{U}_\xi$. Thus, the equation 
\eqref{inverse-transformation} reduces to:
\begin{equation}\label{inverse-transformation2}
  \M(\lambda)=\M^{(4)}(\lambda)\E(\lambda)
  \pmb{\mathcal{R}}^{(2)}(\lambda)^{-1}\pmb{\Delta}(\lambda),\quad
  \lambda\in\mathbb{C}\setminus\mathcal{U}_\xi,
\end{equation}
by the construction \eqref{construct-M3} of $\M^{(3)}(\lambda)$. 
Using the fact that $\pmb{\mathcal{R}}^{(2)}(\lambda)=\mathbb{I}_3$ 
in $\Omega_2$ and that $\pmb{\Delta}(\lambda)$ is the diagonal matrix defined in \eqref{def-Delta}, 
we derive:
\begin{equation}
  (\M_1)_{12} = (\M_1^{(4)})_{12} + (\E_1)_{12}.
\end{equation}
The long-time asymptotics of $(\M_1)_{12}$ then follows from equation  
\eqref{expression-E1} and Lemma \ref{bdd-residue-M4}:
\begin{equation}\label{expansion-M1}
  (\M_1)_{12}=\frac{-\mathrm{i}\eta^2\pmb{\beta}_{12}}{2\sqrt{t}}+\mathcal{O}(t^{-3/4}).
\end{equation}
From equation \eqref{M4-integral-form} and the fact that $\bar{\partial}\pmb{\mathcal{R}}^{(2)}(\lambda)$ is 
supported away from the discrete spectral points $\lambda_i$, we have:
\begin{equation}
  |\M^{(4)}(\lambda_i)-\mathbb{I}_3| \lesssim \iint_{\mathbb{C}'}
  \frac{|\bar{\partial}\pmb{\mathcal{R}}^{(2)}(s)|}{|s-\lambda_i|}\,\mathrm{d} A(s)
  \lesssim \iint_{\mathbb{C}}
  |\bar{\partial}\pmb{\mathcal{R}}^{(2)}(s)|\,\mathrm{d} A(s),
\end{equation}
where $\mathbb{C}'=\mathbb{C}\setminus D$ and 
$D = \{\lambda : |\lambda-\lambda_i|<\rho/3, \ \lambda_i\in\mathcal{Z}\cup\mathcal{Z}^*\}$.
Similar to the proof of Lemma \ref{bdd-residue-M4}, it can be shown that
\begin{equation}\label{expansion-M4-lambda}
  \M^{(4)}(\lambda_i)=\mathbb{I}_3+\mathcal{O}(t^{-3/4}).
\end{equation}
According to the definition \eqref{def-R2} of $\pmb{\mathcal{R}}^{(2)}(\lambda)$ 
and Lemma \ref{def-extension}, it follows that
\begin{equation}\label{expansion-extension-R}
  \pmb{\mathcal{R}}^{(2)}(\lambda_i)=\mathbb{I}_3+\mathcal{O}(\mathrm{e}^{-ct}).
\end{equation}
Therefore, we obtain the long-time asymptotic behavior of $\M(\lambda_i)$ as follows:
\begin{equation}\label{expansion-M-lambda}
  \M(\lambda_i)=\left(\mathbb{I}_3+t^{-1/2}\tilde{\mathbf{P}}_i\right)
  \pmb{\Delta}(\lambda_i)+\mathcal{O}(t^{-3/4}),
\end{equation}
using equations \eqref{expression-E1}, \eqref{inverse-transformation2}, 
\eqref{expansion-M4-lambda}, and \eqref{expansion-extension-R}.

Finally, combining equations \eqref{better-recovery-formula}, \eqref{expansion-M1}, 
and \eqref{expansion-M-lambda}, we obtain Theorem \ref{main-result-dbar-CNLS} immediately.

\section{Asymptotic stability of solitons in the focusing CNLS equation}
From Theorem \ref{main-result-dbar-CNLS}, we obtain the long-time asymptotic 
expansion of the solutions to the CNLS equation \eqref{CNLS}. 
In the expansion \eqref{recover-formula-dbar}, 
the leading term is the part associated with the discrete spectrum $\{\lambda_i\}$, 
which represents the solitonic contribution. 
To establish the asymptotic stability of multi-soliton solutions for 
the CNLS equation, a further analysis of this leading term is required.

Many key results regarding the stability of solitary waves 
are derived using inverse scattering methods. In order to state 
our results, we define the scattering map
\begin{equation}
    \mathcal{D}_n: H^{1,1}(\mathbb{R})\cap\mathcal{G}_n \to 
    H^{1,1}(\mathbb{R}) \times \mathbb{C}_+^n \times \mathbb{C}^n,
\end{equation}
given by $\mathcal{D}_n(\mathbf{q}_0) = \{\hat{\mathbf{R}}(\lambda), \{\lambda_i\}_{i=1}^n, \{\mathbf{c}_i\}_{i=1}^n\}$,
where $n\in\{0\}\cup\mathbb{N}$ and 
\begin{equation}\label{def-g-n}
    \mathcal{G}_n = \left\{ \mathbf{q}_0(x) \in L^1(\mathbb{R}) \mid 
    \text{the corresponding } \bar{a}(\lambda) 
    \text{ admits exactly $n$ simple zeros at } \{\lambda_i\}_{i=1}^n \right\}.
\end{equation}
It follows from \cite{beals1984scattering} and \cite{Liu_2019} that 
we can derive the following conclusions: 
\begin{enumerate}
    \item The map $\mathcal{D}_n$ is bijective and Lipschitz continuous 
    for each $n\in\{0\}\cup\mathbb{N}$.
    \item The set $\mathcal{G}_n$ is indeed open for each $n\in\{0\}\cup\mathbb{N}$.
\end{enumerate} 
The openness of the sets $\mathcal{G}_n$ in the scalar case was demonstrated 
by Beals and Coifman \cite{beals1984scattering}. 
Here, we extend this result to the CNLS system; 
the details are provided in Appendix C.
  
This implies that every $n$-soliton solution $\mathbf{q}^n_{\mathrm{sol}}(x,0)$ 
belongs to $\mathcal{G}_n$. Moreover, if there exists a sufficiently small $\varepsilon > 0$ 
and a function $\mathbf{q}_0(x) \in H^{1,1}(\mathbb{R})$ such that 
\begin{equation}
    \|\mathbf{q}_0(\cdot) - \mathbf{q}^n_{\mathrm{sol}}(\cdot,0)\|_{H^{1,1}(\mathbb{R})} \le \varepsilon,
\end{equation}
then $\mathbf{q}_0(x) \in \mathcal{G}_n$.
We first consider the asymptotic stability of the single-soliton solution, 
i.e., the case where $\mathbf{q}_0(x) \in \mathcal{G}_1$.

\subsection{Asymptotic stability of single-soliton}\label{sec-5.1}
Suppose the initial data $\mathbf{q}_0$ satisfies
\begin{equation}
  \mathcal{D}_1(\mathbf{q}_0) = \{\hat{\mathbf{R}}(\lambda), \lambda_1, \mathbf{c}_1\},
\end{equation}
where $\lambda_1=\xi_1+\mathrm{i}\eta_1$. 
It follows from equation \eqref{better-recovery-formula} and 
Theorem \ref{main-result-dbar-CNLS} that we can deduce:
\begin{equation}\label{reduction-q-1}
  \begin{aligned}
    \mathbf{q}(x,t) &= -4\mathrm{i} \eta_1 \frac{\pmb{\varphi}_1\pmb{\varphi}_2^\dagger}{\pmb{\varphi}^\dagger\pmb{\varphi}}
      + \mathcal{O}(t^{-1/2}),\\
    \pmb{\varphi} &= \begin{pmatrix} \pmb{\varphi}_1 \\ \pmb{\varphi}_2 \end{pmatrix} = 
    \begin{pmatrix}
      \mathrm{e}^{-\mathrm{i}\lambda_1 x-\mathrm{i}\lambda_1^2 t}\M_{11}(\lambda_1) +
      \mathrm{e}^{\mathrm{i}\lambda_1 x+\mathrm{i}\lambda_1^2 t}\M_{12}(\lambda_1)\frac{\mathbf{c}_1}{2\mathrm{i}\eta_1} \\
      \mathrm{e}^{-\mathrm{i}\lambda_1 x-\mathrm{i}\lambda_1^2 t}\M_{21}(\lambda_1) +
      \mathrm{e}^{\mathrm{i}\lambda_1 x+\mathrm{i}\lambda_1^2 t}\M_{22}(\lambda_1)\frac{\mathbf{c}_1}{2\mathrm{i}\eta_1}
    \end{pmatrix}.
  \end{aligned}
\end{equation}
For convenience, we introduce the following notation: 
\begin{equation}
  -\mathrm{i}\lambda_1 x-\mathrm{i}\lambda_1^2 t = \eta_1(x+2\xi_1 t) + \mathrm{i}(-x\xi_1+t(\eta_1^2-\xi_1^2)) = \varphi_1 + \mathrm{i}\psi_1.
\end{equation}
We now provide the detailed analysis for the asymptotic stability of the single-soliton solution 
as $t \to +\infty$.

If $\xi - \xi_1 > t^{-1}$, then $\varphi_1 = 2t\eta_1(\xi_1 - \xi) < 0$. Consequently, 
$\mathrm{e}^{-2\varphi_1}$ is the leading term in equation \eqref{reduction-q-1}, leading to:
\begin{equation}
  \begin{aligned}
    \pmb{\varphi}_1\pmb{\varphi}_2^\dagger &= \mathrm{e}^{-2\varphi_1} \left(
      \M_{12}(\lambda_1)\frac{\mathbf{c}_1}{2\mathrm{i}\eta_1}
      \frac{\mathbf{c}^\dagger_1}{-2\mathrm{i}\eta_1}\M^\dagger_{22}(\lambda_1) +
      \mathcal{O}(\mathrm{e}^{-ct}) \right),\\
    \pmb{\varphi}^\dagger\pmb{\varphi} &= \mathrm{e}^{-2\varphi_1} \left(
      \left| \M_{12}(\lambda_1)\frac{\mathbf{c}_1}{2\mathrm{i}\eta_1} \right|^2 +
      \left| \M_{22}(\lambda_1)\frac{\mathbf{c}_1}{2\mathrm{i}\eta_1} \right|^2 +
      \mathcal{O}(\mathrm{e}^{-ct}) \right).
  \end{aligned}
\end{equation}
Combining this with the estimate \eqref{expansion-M-lambda} 
for $\M(\lambda)$ at $\lambda=\lambda_1$, equation \eqref{reduction-q-1} 
reduces to: 
\begin{equation}
  \mathbf{q}(x,t) = -4\mathrm{i} \eta_1 \frac{\mathcal{O}(t^{-1/2})}{ |\M_{22}(\lambda_1)\frac{\mathbf{c}_1}{2\mathrm{i}\eta_1}|^2 + \mathcal{O}(t^{-1})} + \mathcal{O}(t^{-1/2}) = \mathcal{O}(t^{-1/2}), \quad \text{as } t \to +\infty.
\end{equation}
In a similar manner, we find that $\mathbf{q}(x,t) = \mathcal{O}(t^{-1/2})$ in the region $\xi_1 - \xi > t^{-1}$.

Next, we consider the case where $|\xi_1 - \xi| \le t^{-1}$, in which 
$\mathrm{e}^{\pm 2\varphi_1}$ remain bounded. Using equations \eqref{expansion-M-lambda} and \eqref{symmetry-delta}, 
we obtain:
\begin{equation}
  \begin{aligned}
    \pmb{\varphi}_1\pmb{\varphi}_2^\dagger &= \frac{\mathrm{e}^{2\mathrm{i}\psi_1}}{-2\mathrm{i}\eta_1}\det
    \pmb{\delta}(\lambda_1)\mathbf{c}_1^\dagger\pmb{\delta}(\lambda^*_1) + \mathcal{O}(t^{-1/2}),\\
    \pmb{\varphi}^\dagger\pmb{\varphi} &= \mathrm{e}^{2\varphi_1} |\det\pmb{\delta}(\lambda_1)|^2 +
       \mathrm{e}^{-2\varphi_1} \left| \pmb{\delta}^{-1}(\lambda_1)\frac{\mathbf{c}_1}{2\mathrm{i}\eta_1} \right|^2 + \mathcal{O}(t^{-1/2}).
  \end{aligned}
\end{equation}
Consequently, equation \eqref{reduction-q-1} can be reduced to:
\begin{equation}\label{single-soliton-expansion1}
  \mathbf{q}(x,t) = \frac{2\mathrm{e}^{-2\mathrm{i}\xi_1x-2\mathrm{i} t(\xi_1^2-\eta_1^2)}\det
  \pmb{\delta}(\lambda_1)\mathbf{c}_1^\dagger\pmb{\delta}(\lambda^*_1)}{\mathrm{e}^{2\eta_1(x+2\xi_1t)} | \det\pmb{\delta}(\lambda_1) |^2 + \mathrm{e}^{-2\eta_1(x+2\xi_1t)} \frac{|\mathbf{c}_1^\dagger\pmb{\delta}(\lambda^*_1)|^2}{4\eta_1^2}}
    + \mathcal{O}(t^{-1/2}).
\end{equation}
It is necessary for further analysis to provide some estimates 
for $\det\pmb{\delta}(\lambda_1)$ and $\pmb{\delta}(\lambda_1)$ in the following lemmas.

\begin{lemma}\label{det-delta-lambda1}
  For fixed $\lambda_1=\xi_1+\mathrm{i}\eta_1$, we have 
  \begin{equation}\label{det-delta-lambda1-1}
    |\det\pmb{\delta}(\lambda_1)-\Delta_{\xi_1}(\lambda_1)| \lesssim 1,
  \end{equation}
  where 
  \begin{equation}\label{det-delta-lambda1-2}
    \Delta_{\xi_1}(\lambda_1) = \exp\left(\frac{1}{2\pi\mathrm{i}}\int_{-\infty}^{\xi_1}
    \frac{\ln(1+|\mathbf{R}(s)|^2)}{s-\lambda_1}\,\mathrm{d}s\right).
  \end{equation}
  Moreover, for $|\xi-\xi_1| \lesssim t^{-1/2}$, we have 
  \begin{equation}\label{det-delta-lambda1-3}
    |\det\pmb{\delta}(\lambda_1)-\Delta_{\xi_1}(\lambda_1)| \lesssim t^{-1/2}.
  \end{equation}
\end{lemma}

\begin{proof}
  Equation \eqref{det-delta-lambda1-1} follows from part (IV) of Proposition \ref{proposition-delta}. 
  By combining equations \eqref{def-kappa} and \eqref{det-delta-lambda1-2}, we find:
  \begin{equation}
    |\det\pmb{\delta}(\lambda_1)-\Delta_{\xi_1}(\lambda_1)| \lesssim
    \left|\frac{1}{2\pi\mathrm{i}}\int_{\xi_1}^{\xi}
    \frac{\ln(1+|\mathbf{R}(s)|^2)}{s-\lambda_1}\,\mathrm{d}s\right|
    \lesssim \frac{\|\mathbf{R}\|^2_{L^{\infty}}}{|\eta_1|} |\xi-\xi_1| \lesssim t^{-1/2}.
  \end{equation}
  This completes the proof.
\end{proof}

\begin{lemma}\label{delta-lambda1}
  For fixed $\lambda_1=\xi_1+\mathrm{i}\eta_1$, we have 
  \begin{equation}\label{delta-lambda1-1}
    |\pmb{\delta}(\lambda_1)-\pmb{\Delta}_{\xi_1}(\lambda_1)| \lesssim 1,
  \end{equation}
  where $\pmb{\Delta}_{\xi_1}(\lambda)$ is the solution to the following RHP.
  \begin{rhp}\label{RHP-Delta}
  Find a matrix function $\pmb{\Delta}_{\xi_1}(\lambda)$ with the following properties: 
  \begin{enumerate}
    \item Analyticity: $\pmb{\Delta}_{\xi_1}(\lambda)$ is analytic in 
    $\mathbb{C}\setminus(-\infty, \xi_1)$.
    \item Jump condition: $\pmb{\Delta}_{\xi_1}(\lambda)$ has continuous boundary 
    values $\pmb{\Delta}_{\xi_1\pm}(\lambda)$ on $(-\infty, \xi_1)$ defined by
    \begin{equation}
      \pmb{\Delta}_{\xi_1\pm}(\lambda) = \lim_{\varepsilon\downarrow 0} \pmb{\Delta}_{\xi_1}(\lambda\pm\mathrm{i}\varepsilon), 
    \end{equation}
    satisfying 
    \begin{equation}
      \pmb{\Delta}_{\xi_1+}(\lambda) = (\mathbb{I}_2 + \mathbf{R}^\dagger\mathbf{R})\pmb{\Delta}_{\xi_1-}(\lambda).
    \end{equation}
    \item Asymptotic condition: 
    $\pmb{\Delta}_{\xi_1}(\lambda) = \mathbb{I}_{2} + \mathcal{O}(\lambda^{-1})$ as $\lambda \to \infty$.
  \end{enumerate} 
\end{rhp}

\noindent
  Moreover, for $|\xi-\xi_1| \lesssim t^{-1/2}$, we have 
  \begin{equation}\label{delta-lambda1-2}
    |\pmb{\delta}(\lambda_1)-\pmb{\Delta}_{\xi_1}(\lambda_1)| \lesssim t^{-1/2}.
  \end{equation}
\end{lemma}

\begin{proof}
  Equation \eqref{delta-lambda1-1} follows from part (IV) of Proposition \ref{proposition-delta}. 
  Without loss of generality, assume $\xi_1 < \xi$ and define 
  $\pmb{\tilde{\delta}}(\lambda) = \pmb{\delta}(\lambda) - \pmb{\Delta}_{\xi_1}(\lambda)$. 
  Then $\pmb{\tilde{\delta}}(\lambda)$ satisfies the following non-homogeneous RHP:
  \begin{equation}
    \begin{cases}
      \pmb{\tilde{\delta}}_+(\lambda) = (\mathbb{I}_2 + \mathbf{R}^\dagger\mathbf{R}(\lambda))\pmb{\tilde{\delta}}_-(\lambda) + \pmb{f}(\lambda),
      & \lambda \in (-\infty, \xi), \\
      \pmb{\tilde{\delta}}(\lambda) \to 0,
      & \lambda \to \infty,
    \end{cases}
  \end{equation}
  where $\pmb{f}(\lambda) = \mathbf{R}^\dagger\mathbf{R}(\lambda)\pmb{\Delta}_{\xi_1-}(\lambda)\chi_{(\xi_1, \xi)}(\lambda)$, and $\chi_{(\xi_1, \xi)}(\lambda)$ denotes the characteristic function of the interval $(\xi_1, \xi)$. 
  The solution $\pmb{\tilde{\delta}}(\lambda)$ can be expressed via the Plemelj formula as:
  \begin{equation}
    \begin{aligned}
      \pmb{\tilde{\delta}}(\lambda) &= \pmb{\delta}(\lambda)\int_{-\infty}^{\xi}
      \frac{\pmb{\delta}_+^{-1}(s)\pmb{f}(s)}{s-\lambda}\,\mathrm{d}s \\
      &= \pmb{\delta}(\lambda)\int_{\xi_1}^{\xi}
      \frac{\pmb{\delta}^{-1}_+(s)\mathbf{R}^\dagger\mathbf{R}(s)\pmb{\Delta}_{\xi_1-}(s)}{s-\lambda}\,\mathrm{d}s.
    \end{aligned}
  \end{equation}
  Using part (IV) of Proposition \ref{proposition-delta} and the fact that $\mathbf{R}(\lambda) \in H^1(\mathbb{R})$, we have:
  \begin{equation}
    |\pmb{\tilde{\delta}}(\lambda_1)| \lesssim \int_{\xi_1}^{\xi}
    \frac{1}{|s-\lambda_1|}\,\mathrm{d}s \lesssim \frac{|\xi_1-\xi|}{|\eta_1|} \lesssim t^{-1/2}.
  \end{equation}
  The estimate \eqref{delta-lambda1-2} follows immediately, completing the proof.
\end{proof}
Through Lemma \ref{delta-lambda1} and Lemma \ref{det-delta-lambda1}, we 
deduce the following long-time asymptotic expansion of the solution $\mathbf{q}(x,t)$ 
of the CNLS equation \eqref{CNLS} when the initial data $\mathbf{q}_0(x)$ lies in 
$\mathcal{G}_1$.

\begin{lemma}\label{estimate-soliton-solution}
For initial data $\mathbf{q}_0(x)$ satisfying 
$\mathcal{D}_1(\mathbf{q}_0)=\{\hat{\mathbf{R}}(\lambda),\lambda_1=\xi_1+\mathrm{i}\eta_1, \mathbf{c}_1\}$, 
the asymptotic expansion of the solution $\mathbf{q}(x,t)$ for the CNLS equation \eqref{CNLS} is given by:
  \begin{equation}\label{estimate-q-single-soliton}
  \mathbf{q}(x,t)=2\eta_1\operatorname{sech}(2
  \eta_1(x-x_1^++2t\xi_1))
  \mathrm{e}^{-2\mathrm{i}(\xi_1x-(\eta_1^2-\xi_1^2
  )t)+\mathrm{i}\phi_1^+}
  \mathbf{\tilde{c}_1^+}
    +\mathcal{O}(t^{-1/2}),\quad \text{as } t\to\infty,
\end{equation}
where 
\begin{equation}
  \begin{aligned}
    x_1^+ &= \frac{1}{2\eta_1}\left(\ln\frac{|\mathbf{c}_1^\dagger\pmb{\Delta}_{\xi_1}(\lambda^*_1)|}{2\eta_1}-\int_{-\infty}^{\xi_1}
    \frac{\eta_1\ln(1+\mathbf{R}\mathbf{R}^\dagger(s))}{(s-\xi_1)^2+\eta_1^2}
    \frac{\mathrm{d}s}{2\pi}\right),\\
    \phi_1^+ &= \int_{-\infty}^{\xi_1}
    \frac{\ln(1+\mathbf{R}\mathbf{R}^\dagger(s))}{(s-\xi_1)^2+\eta_1^2}
    \frac{\xi_1-s}{2\pi}\,\mathrm{d}s,\quad 
    \mathbf{\tilde{c}_1^+} = \frac{\mathbf{c}_1^\dagger\pmb{\Delta}_{\xi_1}(\lambda^*_1)}{|\mathbf{c}_1^\dagger\pmb{\Delta}_{\xi_1}(\lambda^*_1)|}.
  \end{aligned}
\end{equation}
\end{lemma}

\begin{proof}
  We first consider the region where $(x,t)$ satisfies 
  $|\xi_1-\xi|\le t^{-1}$. 
  It follows from equation \eqref{single-soliton-expansion1}, Lemma 
  \ref{det-delta-lambda1}, and Lemma \ref{delta-lambda1} that:
  \begin{equation}
   \mathbf{q}(x,t)=\frac{2\mathrm{e}^{-2\mathrm{i}\xi_1x-2\mathrm{i} t(\xi_1^2-\eta_1^2)}\Delta_{\xi_1}(\lambda_1)
   \mathbf{c}_1^\dagger\pmb{\Delta}_{\xi_1}(\lambda^*_1)}{\mathrm{e}^{4\eta_1t(\xi_1-\xi)}|
    \Delta_{\xi_1}(\lambda_1)|^2+\mathrm{e}^{-4\eta_1t(\xi_1-\xi)
    }\frac{|\mathbf{c}_1^\dagger\pmb{\Delta}_{\xi_1}(\lambda^*_1)|^2}{4\eta_1^2}}
    + \mathcal{O}(t^{-1/2}),
  \end{equation}
  and the result \eqref{estimate-q-single-soliton} follows by rearranging the exponential terms into the hyperbolic secant form. 

  Now, we consider the case where $(x,t)$ is such that 
  $|\xi_1-\xi|> t^{-1}$. We can derive the following estimate for the solitonic part:
 \begin{equation}
  \left|2\eta_1\operatorname{sech}(2\eta_1(x-x_1^++2t\xi_1))
  \mathrm{e}^{-2\mathrm{i}(\xi_1x-(\eta_1^2-\xi_1^2)t)
  +\mathrm{i}\phi_1^+}\mathbf{\tilde{c}_1^+}\right|\lesssim \mathrm{e}^{\pm 2\eta_1(x+2\xi_1t)}
  \lesssim t^{-1/2},
 \end{equation}
 for $\pm(\xi-\xi_1)>t^{-1}$. 
 Together with the established conclusion that $\mathbf{q}(x,t)=\mathcal{O}(t^{-1/2})$ in these regions, we obtain:
 \begin{equation}
  \left|\mathbf{q}(x,t)-2\eta_1\operatorname{sech}(2\eta_1(x-x_1^++2t\xi_1))
  \mathrm{e}^{-2\mathrm{i}(\xi_1x-(\eta_1^2-\xi_1^2)t)
  +\mathrm{i}\phi_1^+}\mathbf{\tilde{c}_1^+}\right|\lesssim t^{-1/2}.
 \end{equation}
 This completes the proof.
\end{proof}

\noindent
A similar technique can be applied to the case $t\to-\infty$. Combined with equation 
\eqref{M-lambda-neg}, we obtain the following lemma.

\begin{lemma}\label{estimate-soliton-solution-neg}
  If the discrete spectrum $\mathcal{Z}$ consists of only a single 
point $\lambda_1=\xi_1+\mathrm{i}\eta_1$, then the asymptotic expansion of the 
solution $\mathbf{q}(x,t)$ for the CNLS equation \eqref{CNLS} is given by:
  \begin{equation}\label{estimate-q-single-soliton-neg}
  \mathbf{q}(x,t)=2\eta_1\operatorname{sech}(2\eta_1(x-x_1^-+2t\xi_1))
  \mathrm{e}^{-2\mathrm{i}(\xi_1x-(\eta_1^2-\xi_1^2)t)+\mathrm{i}\phi_1^-}\mathbf{\tilde{c}_1^-}
    +\mathcal{O}(t^{-1/2}),\quad \text{as } t\to-\infty,
\end{equation}
where 
\begin{equation}
  \begin{aligned}
    x_1^- &= \frac{1}{2\eta_1}\left(\ln\frac{|\mathbf{c}_1^\dagger\pmb{\Delta}^-_{\xi_1}(\lambda^*_1)|}{2\eta_1}-\int_{\xi_1}^{+\infty}
    \frac{\eta_1\ln(1+\mathbf{R}\mathbf{R}^\dagger(s))}{(s-\xi_1)^2+\eta_1^2} \frac{\mathrm{d}s}{2\pi}\right),\\
    \phi_1^- &= \int_{\xi_1}^{+\infty} \frac{\ln(1+\mathbf{R}\mathbf{R}^\dagger(s))}{(s-\xi_1)^2+\eta_1^2} \frac{\xi_1-s}{2\pi}\,\mathrm{d}s,\quad 
    \mathbf{\tilde{c}_1}^- = \frac{\mathbf{c}_1^\dagger\pmb{\Delta}^-_{\xi_1}(\lambda^*_1)}{|\mathbf{c}_1^\dagger\pmb{\Delta}^-_{\xi_1}(\lambda^*_1)|}.
  \end{aligned}
\end{equation}
\end{lemma}

Now, we can complete the proof of Theorem \ref{main-result-asymptotic-stability}.

\begin{proof}[Proof of Theorem \ref{main-result-asymptotic-stability}]
  Since $\mathcal{G}_1$ is an open set, there exists a sufficiently small 
  $\varepsilon>0$ such that $\mathbf{q}_0(x)\in\mathcal{G}_1$. Thus, it follows 
  from Theorem \ref{main-result-dbar-CNLS}, Lemma \ref{estimate-soliton-solution}, and 
  Lemma \ref{estimate-soliton-solution-neg} that:
  \begin{equation}
    \left|\mathbf{q}(x,t)-\mathbf{q}^{\mathrm{sol}\pm}_{\omega_1,\gamma_1^\pm,v_1}(x-x_1^\pm,t)\right|
    \lesssim t^{-1/2},\quad \text{as } t\to\pm\infty,
  \end{equation}
  where $\mathbf{q}^{\mathrm{sol}\pm}_{\omega_1,\gamma_1^\pm,v_1}(x-x_1^\pm,t)$ are the asymptotic solitons described by 
  equations \eqref{estimate-q-single-soliton} and \eqref{estimate-q-single-soliton-neg}
  with parameters $\{\omega_1, \gamma_1^\pm, v_1\} = \{2\eta_1, \phi_1^\pm, -2\xi_1\}$. 
  This completes the proof.
\end{proof}
Next, we consider the asymptotic stability of soliton solutions 
with respect to a set of $n$ discrete eigenvalues, i.e., $\mathbf{q}_0(x) \in \mathcal{G}_n$.

\subsection{Asymptotic stability of $n$-soliton}
Assume that $\mathcal{D}_n(\mathbf{q}_0)$ consists of $n$ discrete eigenvalues
$\lambda_j=\xi_j+\mathrm{i} \eta_j$ ($j=1,2,\ldots,n$), satisfying 
$\xi_1 < \xi_2 < \dots < \xi_n$. 
From equation \eqref{better-recovery-formula} and 
Theorem \ref{main-result-dbar-CNLS}, we deduce:
\begin{equation}\label{reduction-q-n}
    \mathbf{q}(x,t) = -2\sum_{k=1}^{n}\sum_{m=1}^{n}\frac{\pmb{N}_{m,k}}{\det \pmb{G}}
  \pmb{y}_m^1\pmb{y}_k^{2\dagger}
      + \mathcal{O}(t^{-1/2}),
\end{equation}
where 
\begin{equation}\label{def-G}
  \begin{aligned}
    \pmb{y}_k &= \begin{pmatrix}
      \pmb{y}_k^1 \\[5pt]
      \pmb{y}_k^2
    \end{pmatrix}
    = \M(\lambda_k;x,t)\mathrm{e}^{a_k\pmb{\Lambda}_3}\pmb{v}_k, \\
    a_k &= \eta_k(x+2\xi_k t)+\mathrm{i}(-x\xi_k+t(\eta_k^2-\xi_k^2)) = \varphi_k+\mathrm{i}\psi_k, \\
    \mathbf{v}_{i} &= \begin{bmatrix}
        \prod_{j=i+1}^{n}\frac{\lambda_i-\lambda_j^*}{\lambda_i-\lambda_j} \\
        \frac{\pmb{c}_i^{\top}}{\lambda_i-\lambda_i^*} \prod_{j=1}^{i-1}\frac{\lambda_i-\lambda_j}{\lambda_i-\lambda_j^*}
      \end{bmatrix}, \quad
    \mathbf{v}_{1} = \begin{bmatrix}
         \prod_{j=2}^{n}\frac{\lambda_1-\lambda_j^*}{\lambda_1-\lambda_j} \\
        \frac{\pmb{c}_1^{\top}}{\lambda_1-\lambda_1^*}
      \end{bmatrix}, \quad
    \mathbf{v}_{n} = \begin{bmatrix}
        1 \\
        \frac{\pmb{c}_n^{\top}}{\lambda_n-\lambda_n^*} \prod_{j=1}^{n-1}\frac{\lambda_n-\lambda_j}{\lambda_n-\lambda_j^*}
      \end{bmatrix}, \\
    \pmb{G}_{km} &= \frac{\pmb{y}_k^\dagger\pmb{y}_m}{\lambda_m-\lambda_k^*}, \quad
    \pmb{N} = \operatorname{adj}(\pmb{G}).
  \end{aligned}
\end{equation}
To analyze the first term on the right-hand side of equation \eqref{reduction-q-n} 
in different regions, we rewrite it as follows: 
\begin{equation}\label{reduction-q-n-det}
  \mathbf{q}(x,t) = \frac{2}{\det \pmb{G}}\begin{bmatrix}
  \det \pmb{F}_1, \det \pmb{F}_2
  \end{bmatrix} + \mathcal{O}(t^{-1/2}),
\end{equation}
where $\pmb{F}_j$ is the following $(n+1)\times(n+1)$ matrix:
\begin{equation}\label{def-F}
  \pmb{F}_j = \begin{bmatrix}
        0 & \pmb{Y}_1 \\
        \pmb{Y}_2^{[j]\dagger} & \pmb{G}
      \end{bmatrix}, \quad
  \pmb{Y}_1 = \begin{bmatrix} \pmb{y}_1^1, \dots, \pmb{y}_n^1 \end{bmatrix}, \quad
  \pmb{Y}_2^{[j]} = \begin{bmatrix} (\pmb{y}_1^2)^{[j]}, \dots, (\pmb{y}_n^2)^{[j]} \end{bmatrix} \quad \text{for } j=1,2.
\end{equation}

\begin{remark}
  In the scalar case, it can be shown that an $n$-soliton solution can split into $n$ separated single solitons 
  with respect to the stationary point $\xi=-x/2t$ and discrete eigenvalues $\lambda_j$
  \cite{Exact_theory_Zakharov_1972}. In 2017, Saalmann grouped the poles with respect to their real
  parts and introduced $T(\lambda,x,t)=\prod_{\mathrm{Re}\lambda_j < \xi} \frac{\lambda-\lambda_j}{\lambda-\lambda_j^*}$ 
  to replace the poles with jumps on small contours in the RHP transformation \cite{stability-n-soliton-nls}. 
  It was shown that all jumps on these small contours decay exponentially as $t\to\infty$, 
  due to the construction of $T(\lambda,x,t)$. This implies that the initial RHP with $n$ poles can be reduced to a RHP with only a single pole, allowing the long-time asymptotics of $\mathbf{q}(x,t)$ to be described by a single-soliton solution as $\xi \to \mathrm{Re}\lambda_j$. 
  However, for the CNLS equation \eqref{CNLS}, we cannot utilize a scalar $T(\lambda,x,t)$ to transform the lower-triangular residue conditions into upper-triangular ones at $\lambda=\lambda_j$ for $\mathrm{Re}\lambda_j < \xi$ in RHP \ref{initial-rhp-with-pole}. 
  Consequently, directly applying the method of \cite{stability-n-soliton-nls} to reduce the multi-component RHP to a single-pole problem is challenging.
\end{remark}

  If $\left|\xi-\xi_l\right|<t^{-1}(1\le l\le n)$, we can deduce that 
  $\ee^{2\varphi_l}$ and $\ee^{-2\varphi_l}$ 
  are bounded and the leading term is 
  $\ee^{-2(\varphi_1+\cdots+\varphi_{l-1})+2(\varphi_{l+1}+\cdots+\varphi_n)}$. 
  It follows from equation \eqref{expansion-M-lambda} that
  \begin{equation*}
    \pmb{y}_l=\begin{bmatrix}
      \ee^{a_l}b_l+\ee^{-a_l}\oo(t^{-1/2})\\
      \ee^{a_l}\oo(t^{-1/2})+\ee^{-a_l}\pmb{d}_l
    \end{bmatrix},\ \ t\to\infty.
  \end{equation*}
  Then we can obtain the following estimates for $\pmb{G}_{km}$
  \begin{equation}\label{estimate-G-km-l}
  \begin{cases}
    \pmb{G}_{ll}=\frac{1}{\lambda_l-\lambda_l^*}
    (\ee^{2\varphi_l}|b_l|^2+
    \ee^{-2\varphi_l}|\pmb{d}_l|^2+\oo(t^{-1/2})),
    &k=m=l,\\
    \pmb{G}_{lm}=\frac{\ee^{-\varphi_m}}{\lambda_m-\lambda_l^*}
    (\ee^{-\varphi_l+\ii(\psi_l-\psi_m)}\pmb{d}_l^\dagger\pmb{d}_m
    +\oo(t^{-1/2})),
    &k=l,1\le m\le l-1,\\
    \pmb{G}_{lm}=\frac{\ee^{\varphi_m}}{\lambda_m-\lambda_l^*}
    (\ee^{\varphi_l+\ii(\psi_m-\psi_l)}b_l^*b_m
    +\oo(t^{-1/2})),
    &k=l,l+1\le m\le n,\\
    \pmb{G}_{kl}=\frac{\ee^{-\varphi_k}}{\lambda_l-\lambda_k^*}
    (\ee^{-\varphi_l+\ii(\psi_k-\psi_l)}\pmb{d}_k^\dagger\pmb{d}_l
    +\oo(t^{-1/2})),
    &1\le k\le l-1,m=l,\\
    \pmb{G}_{kl}=\frac{\ee^{\varphi_k}}{\lambda_l-\lambda_k^*}
    (\ee^{\varphi_l+\ii(\psi_l-\psi_k)}b_k^*b_l
    +\oo(t^{-1/2})),
    &l+1\le k\le n,m=l,
  \end{cases}
\end{equation}
\begin{equation}\label{estimate-G-km-l-2}
    \begin{cases}
      \pmb{G}_{km}=\frac{\ee^{-\varphi_k-\varphi_m}}
      {\lambda_m-\lambda_k^*}
      (\ee^{\ii(\psi_k-\psi_m)}\pmb{d}_k^\dagger\pmb{d}_m
      +\oo(t^{-1/2})),
      &1\le k,m\le l-1,\\
      \pmb{G}_{km}=\frac{\ee^{\varphi_k+\varphi_m}}
      {\lambda_m-\lambda_k^*}
      (\ee^{\ii(\psi_m-\psi_k)}\pmb{b}_k^*\pmb{b}_m
      +\oo(t^{-1/2})),
      &l+1\le k,m\le n,\\
      \pmb{G}_{km}=\frac{\ee^{\varphi_k-\varphi_m}}
      {\lambda_m-\lambda_k^*}
      \oo(t^{-1/2}),
      &1\le m< l< k\le n,\\
      \pmb{G}_{km}=\frac{\ee^{\varphi_m-\varphi_k}}
      {\lambda_m-\lambda_k^*}
      \oo(t^{-1/2}),
      &1\le k< l< m\le n.
    \end{cases}
\end{equation}
By extracting the dominant exponential factors, we derive the asymptotic expansion of $\det \pmb{G}$:
  \begin{equation}
    \begin{aligned}
      \det \pmb{G}&=\ee^{-2(\varphi_1+\cdots+\varphi_{l-1})+
      2(\varphi_{l+1}+\cdots+\varphi_n)} \left(\det\tilde{\pmb{G}}+\oo(t^{-1/2})
      \right)\\
      &=\ee^{-2(\varphi_1+\cdots+\varphi_{l-1})+
      2(\varphi_{l+1}+\cdots+\varphi_n)}\\
      &\ \
      \left(
        \det \pmb{A}_{l:n}\det \pmb{B}_{1:l-1}\ee^{2\varphi_l}+
        \det \pmb{A}_{l+1:n}\det \pmb{B}_{1:l} \ee^{-2\varphi_l}+\oo(t^{-1/2})
      \right),
    \end{aligned}
\end{equation}
where 
  \begin{equation}
  \tilde{\pmb{G}}=\begin{bmatrix}
    \begin{array}{ccccccc}
          \pmb{B}_{11} & \cdots &\pmb{B}_{1,l-1} 
          &\pmb{B}_{1l}\ee^{-\varphi_l-\ii\psi_l} &0 &\cdots &0\\
          \vdots & \ddots & \vdots & \vdots &\vdots & \ddots & \vdots\\
          \pmb{B}_{l-1,1} & \cdots & \pmb{B}_{l-1,l-1} 
          & \pmb{B}_{l-1,l}\ee^{-\varphi_l-\ii\psi_l} &0 & \cdots & 0\\
          \pmb{B}_{l1}\ee^{-\varphi_l+\ii\psi_l} & \cdots &\pmb{B}_{l, l-1}\ee^{-\varphi_l+\ii\psi_l} & \pmb{B}_{ll}\ee^{-2\varphi_l}+\pmb{A}_{ll}
          \ee^{2\varphi_l} & \pmb{A}_{l, l+1}\ee^{\varphi_l-\ii\psi_l}& \cdots & \pmb{A}_{ln}\ee^{\varphi_l-\ii\psi_l}\\
          0 &\cdots & 0 & \pmb{A}_{l+1, l}\ee^{\varphi_l+\ii\psi_l} 
          & \pmb{A}_{l+1, l+1} & \cdots & \pmb{A}_{l+1, n}\\
          \vdots & \ddots & \vdots & \vdots &\vdots & \ddots & \vdots\\
          0 &\cdots& 0 & \pmb{A}_{nl}\ee^{\varphi_l+\ii\psi_l} 
          &\pmb{A}_{n, l+1}& \cdots & \pmb{A}_{nn}
        \end{array}
  \end{bmatrix},
\end{equation}
and
\begin{equation}
  \pmb{A}_{km}=\frac{b_k^*b_m}{\lambda_m-\lambda_k^*},
  \,\ \pmb{B}_{km}=\frac{\pmb{d}_k^\dagger\pmb{d}_m}
  {\lambda_m-\lambda_k^*}.
\end{equation}
Next, we consider the asymptotic expansion of the product $\pmb{y}_m^1\pmb{y}_k^{2\dagger}$:
\begin{equation}\label{def-ym-yk}
  \begin{aligned}
    \pmb{y}_m^1\pmb{y}_k^{2\dagger}
    &= \mathrm{e}^{\varphi_m-\varphi_k+\mathrm{i}(\psi_m+\psi_k)} b_m\pmb{d}_k^\dagger \\
    &\quad + \left( \mathrm{e}^{\varphi_m+\varphi_k+\mathrm{i}(\psi_m-\psi_k)}
    + \mathrm{e}^{-\varphi_m-\varphi_k+\mathrm{i}(\psi_k-\psi_m)}
    + \mathrm{e}^{\varphi_k-\varphi_m-\mathrm{i}(\psi_k+\psi_m)} \right) \mathcal{O}(t^{-1/2}).
  \end{aligned}
\end{equation}
It follows from equation \eqref{def-ym-yk} that the product terms $\pmb{y}_k^1\pmb{y}_m^{2\dagger}$ admit the following case-wise estimates as $t \to \infty$:
\begin{equation}\label{estimate-y-km-l}
  \begin{cases}
    \pmb{y}_k^1\pmb{y}_m^{2\dagger} = \mathrm{e}^{2\mathrm{i}\psi_l} b_l\pmb{d}_l^\dagger + \mathcal{O}(t^{-1/2}),
    & k=m=l, \\[5pt]
    \pmb{y}_m^1\pmb{y}_k^{2\dagger} = \mathrm{e}^{-\varphi_m}\mathcal{O}(t^{-1/2}),
    & k=l, \, 1 \le m \le l-1, \\[5pt]
    \pmb{y}_m^1\pmb{y}_k^{2\dagger} = \mathrm{e}^{\varphi_m} \left( \mathrm{e}^{-\varphi_l+\mathrm{i}(\psi_m+\psi_l)} b_m\pmb{d}_l^\dagger + \mathcal{O}(t^{-1/2}) \right),
    & k=l, \, l+1 \le m \le n, \\[5pt]
    \pmb{y}_m^1\pmb{y}_k^{2\dagger} = \mathrm{e}^{-\varphi_k} \left( \mathrm{e}^{\varphi_l+\mathrm{i}(\psi_l+\psi_k)} b_l\pmb{d}_k^\dagger + \mathcal{O}(t^{-1/2}) \right),
    & m=l, \, 1 \le k \le l-1, \\[5pt]
    \pmb{y}_m^1\pmb{y}_k^{2\dagger} = \mathrm{e}^{\varphi_k}\mathcal{O}(t^{-1/2}),
    & m=l, \, l+1 \le k \le n, \\[5pt]
    \pmb{y}_m^1\pmb{y}_k^{2\dagger} = \mathrm{e}^{-\varphi_m-\varphi_k}\mathcal{O}(t^{-1/2}),
    & 1 \le k, m \le l-1, \\[5pt]
    \pmb{y}_m^1\pmb{y}_k^{2\dagger} = \mathrm{e}^{\varphi_m+\varphi_k} \mathcal{O}(t^{-1/2}),
    & l+1 \le k, m \le n, \\[5pt]
    \pmb{y}_m^1\pmb{y}_k^{2\dagger} = \mathrm{e}^{-\varphi_m+\varphi_k}\mathcal{O}(t^{-1/2}),
    & 1 \le m < l < k \le n, \\[5pt]
    \pmb{y}_m^1\pmb{y}_k^{2\dagger} = \mathrm{e}^{\varphi_m-\varphi_k} \left( \mathrm{e}^{\mathrm{i}(\psi_m+\psi_k)} b_m\pmb{d}_k^\dagger + \mathcal{O}(t^{-1/2}) \right),
    & 1 \le k < l < m \le n.
  \end{cases}
\end{equation}
Together with the estimates \eqref{estimate-G-km-l} 
and \eqref{estimate-G-km-l-2} for 
$\pmb{G}_{km}$, we can derive the following expansion:
 \begin{equation}
    \begin{aligned}
      \sum_{k=1}^{n}\sum_{m=1}^{n}\pmb{N}_{m,k}
      \pmb{y}_m^1\pmb{y}_k^{2\dagger} &= -
    \begin{bmatrix}
      \begin{vmatrix}
        \pmb{G} & (\pmb{y}^2)_1^* \\
        \pmb{y}^1 & 0
      \end{vmatrix}, &
      \begin{vmatrix}
        \pmb{G} & (\pmb{y}^2)_2^* \\
        \pmb{y}^1 & 0
      \end{vmatrix}
    \end{bmatrix} \\
      &= \mathrm{e}^{-2(\varphi_1+\cdots+\varphi_{l-1})+2(\varphi_{l+1}+\cdots+\varphi_n)} \left( -\begin{bmatrix} \pmb{G}_1, \pmb{G}_2 \end{bmatrix} + \mathcal{O}(t^{-1/2}) \right),
    \end{aligned}
\end{equation}
where $\pmb{y}^1 = [\pmb{y}^1_1, \dots, \pmb{y}^1_n]$, $(\pmb{y}^2)_j = [(\pmb{y}^2_1)_j, \dots, (\pmb{y}^2_n)_j]^\top$ for $j=1,2$, and 
\begin{equation}
  \begin{aligned}
    \pmb{G}_k &= \begin{vmatrix}
      \tilde{\pmb{G}}_{11} & \cdots & \tilde{\pmb{G}}_{1l} & \cdots & \tilde{\pmb{G}}_{1n} & d_{1}^{[k]*} \\
      \vdots & \ddots & \vdots & \ddots & \vdots & \vdots \\
      \tilde{\pmb{G}}_{l1} & \cdots & \tilde{\pmb{G}}_{ll} & \cdots & \tilde{\pmb{G}}_{ln} & d_{l}^{[k]*} \\
      \vdots & \ddots & \vdots & \ddots & \vdots & \vdots \\
      \tilde{\pmb{G}}_{n1} & \cdots & \tilde{\pmb{G}}_{nl} & \cdots & \tilde{\pmb{G}}_{nn} & 0 \\
      0 & \cdots & b_l & \cdots & b_n & 0 
    \end{vmatrix} \\
    &= -\mathrm{e}^{2\mathrm{i}\psi_l}
    \begin{vmatrix}
      \pmb{B}_{11} & \cdots & \pmb{B}_{1l} & d_{1}^{[k]*} \\
      \vdots & \ddots & \vdots & \vdots \\
      \pmb{B}_{l1} & \cdots & \pmb{B}_{ll} & d_{l}^{[k]*} \\
      0 & \cdots & 1 & 0
    \end{vmatrix}
    \begin{vmatrix}
      0 & b_l & \cdots & b_n \\
      1 & \pmb{A}_{ll} & \cdots & \pmb{A}_{ln} \\
      \vdots & \vdots & \ddots & \vdots \\
      0 & \pmb{A}_{nl} & \cdots & \pmb{A}_{nn}
    \end{vmatrix}, \quad k=1,2.
  \end{aligned}
\end{equation}

Hence, we can derive:
\begin{equation}
    \sum_{k=1}^{n}\sum_{m=1}^{n}\pmb{N}_{m,k}
    \pmb{y}_m^1\pmb{y}_k^{2\dagger} = \mathrm{e}^{-2(\varphi_1+\cdots+\varphi_{l-1})+2(\varphi_{l+1}+\cdots+\varphi_n)}
    \left(\mathrm{e}^{2\mathrm{i}\psi_l}\pmb{h}_l + \mathcal{O}(t^{-1/2})\right),
\end{equation}
where the vector-valued function $\pmb{h}_l$ is given by:
\begin{equation}
  \pmb{h}_l = \begin{bmatrix}
      \begin{vmatrix}
        \pmb{B}_{1:l} & \pmb{d}_{[1]}^* \\
        \pmb{e}_l^\top & 0
      \end{vmatrix}, &
      \begin{vmatrix}
        \pmb{B}_{1:l} & \pmb{d}_{[2]}^* \\
        \pmb{e}_l^\top & 0
      \end{vmatrix}
    \end{bmatrix}
    \begin{vmatrix}
        0 & \pmb{b} \\
        \pmb{e}_1 & \pmb{A}_{l:n}
    \end{vmatrix},
\end{equation}
with the vectors defined as:
\begin{equation}
  \pmb{b} = [b_l, b_{l+1}, \dots, b_n], \quad
  \pmb{d}_{[j]} = [d_1^{[j]}, d_2^{[j]}, \dots, d_l^{[j]}]^\top,
\end{equation}
and $\pmb{e}_1, \pmb{e}_l$ represent the standard basis vectors of appropriate dimensions.
It follows from equation \eqref{expansion-M-lambda}, Lemma \ref{det-delta-lambda1}, 
and Lemma \ref{delta-lambda1} that:
\begin{equation}\label{estimate-b-d-l}
  \begin{pmatrix}
    b_k \\
    \pmb{d}_k
  \end{pmatrix} = \begin{pmatrix}
    \Delta_{\xi_l}(\lambda_k) & \mathbf{0} \\
    \mathbf{0} & \pmb{\Delta}_{\xi_l}^{-1}(\lambda_k)
  \end{pmatrix}\pmb{v}_k + \mathcal{O}(t^{-1/2}),
\end{equation}
and hence the solution $\mathbf{q}(x,t)$ admits the following asymptotic form:
\begin{equation}\label{asy-q-l}
  \begin{aligned}
    \mathbf{q}(x,t) &= \frac{-2\mathrm{e}^{2\mathrm{i}\psi_l}\pmb{h}}{\det \pmb{A}_{l:n}\det \pmb{B}_{1:l-1}\mathrm{e}^{2\varphi_l} + \det \pmb{A}_{l+1:n}\det \pmb{B}_{1:l} \mathrm{e}^{-2\varphi_l} + \mathcal{O}(t^{-1/2})} + \mathcal{O}(t^{-1/2}) \\
    &= 2\eta_l\operatorname{sech}\left(2\eta_l(x-x_l^++2t\xi_l)\right) \mathrm{e}^{-2\mathrm{i}(\xi_lx-(\eta_l^2-\xi_l^2)t)+\mathrm{i}\phi_l^+} \tilde{\pmb{c}}_l^+ + \mathcal{O}(t^{-1/2}) \\
    &= \mathbf{q}_{\mathrm{sol}}^{l+}(x,t) + \mathcal{O}(t^{-1/2}),
  \end{aligned}
\end{equation}
where the auxiliary variables and parameters are defined as:
\begin{equation}\label{def-A-B}
  \begin{pmatrix} f_k \\ \pmb{g}_k \end{pmatrix} = \begin{pmatrix} \Delta_{\xi_l}(\lambda_k) & \mathbf{0} \\ \mathbf{0} & \pmb{\Delta}_{\xi_l}^{-1}(\lambda_k) \end{pmatrix}\pmb{v}_k, \quad 
  \pmb{A}_{km} = \frac{f_k^*f_m}{\lambda_m-\lambda_k^*}, \quad \pmb{B}_{km} = \frac{\pmb{g}_k^\dagger\pmb{g}_m}{\lambda_m-\lambda_k^*}, \quad \phi_l^+ = \frac{n+2}{2}\pi,
\end{equation}
\begin{equation}\label{def-c-l}
  x_l^+ = \frac{1}{4\eta_l}\ln\frac{|\det \pmb{A}_{l+1:n}\det \pmb{B}_{1:l}|}{|\det \pmb{A}_{l:n}\det \pmb{B}_{1:l-1}|}, \quad 
  \tilde{\pmb{c}}_l^+ = \frac{\pmb{h}}{2\eta_l \sqrt{|\det \pmb{A}_{l:n}\det \pmb{B}_{1:l-1}\det \pmb{A}_{l+1:n}\det \pmb{B}_{1:l}|}},
\end{equation}
and the vector $\pmb{h}$ is given by:
\begin{equation}\label{def-h-l}
  \pmb{h} = \begin{bmatrix}
      \begin{vmatrix} \pmb{B}_{1:l} & \pmb{g}_{[1]}^* \\ \pmb{e}_l^\top & 0 \end{vmatrix}, &
      \begin{vmatrix} \pmb{B}_{1:l} & \pmb{g}_{[2]}^* \\ \pmb{e}_l^\top & 0 \end{vmatrix}
    \end{bmatrix} \begin{vmatrix} 0 & \pmb{f} \\ \pmb{e}_1 & \pmb{A}_{l:n} \end{vmatrix},
\end{equation}
with $\pmb{f} = [f_l, f_{l+1}, \dots, f_{n}]$ and $\pmb{g}_{[j]} = [\pmb{g}_1^{[j]}, \pmb{g}_{2}^{[j]}, \dots, \pmb{g}_{l}^{[j]}]^\top$.
  Away from the neighborhoods of the stationary points, we can trivially obtain $\mathbf{q}(x,t)=\mathcal{O}(t^{-1/2})$. 
Then, the solution $\mathbf{q}(x,t)$ undergoes soliton resolution in the sense that:
\begin{equation}
  \mathbf{q}(x,t)=\sum_{k=1}^{n}\mathbf{q}_{\mathrm{sol}}^{k+}(x,t)+\mathcal{O}(t^{-1/2}),\quad t\to+\infty,
\end{equation}
where $\mathbf{q}_{\mathrm{sol}}^{k+}(x,t)$ is provided by equation \eqref{asy-q-l}. 
Next, we demonstrate that each $\mathbf{q}_{\mathrm{sol}}^{k+}(x,t)$, for $1\le k\le n$, corresponds to a single-soliton solution of the CNLS equation \eqref{CNLS} by utilizing the following algebraic lemma.

\begin{lemma}\label{lem-diag}
  Consider an invertible matrix $\pmb{B}=[B_{ij}]$ with entries $B_{ij}=\frac{b_{ij}}{x_i+y_j}$. There exist matrices $\tilde{\pmb{B}}=[b_{ij}]$ 
  and $\pmb{C} = \operatorname{diag}(x_1+y_1, x_2+y_2, \dots, x_n+y_n)$
  such that:
  \begin{equation}
    (\pmb{B}^{-1}\tilde{\pmb{B}}\pmb{B}^{-1})^{\mathrm{diag}} = (\pmb{B}^{-1})^{\mathrm{diag}}\pmb{C},
  \end{equation}
  where $\pmb{B}^{\mathrm{diag}}$ denotes the diagonal part of the matrix $\pmb{B}$.
\end{lemma}
  \begin{proof}
    The element $b_{ij}$ can be decomposed into two parts:
    \begin{equation}
      b_{ij} = x_i B_{ij} + y_j B_{ij}.
    \end{equation}
    In matrix form, this leads to the following relation:
    \begin{equation}
      \tilde{\pmb{B}} = \pmb{X}\pmb{B} + \pmb{B}\pmb{Y},
    \end{equation}
    where $\pmb{X} = \operatorname{diag}(x_1, x_2, \dots, x_n)$ 
    and $\pmb{Y} = \operatorname{diag}(y_1, y_2, \dots, y_n)$.
    Multiplying by $\pmb{B}^{-1}$ from both the left and the right 
    handsides, we obtain
    \begin{equation}
      \pmb{B}^{-1}\tilde{\pmb{B}}\pmb{B}^{-1} = \pmb{B}^{-1}\pmb{X} + \pmb{Y}\pmb{B}^{-1}.
    \end{equation}
    Taking the $k$-th diagonal element by applying the standard basis vector $\pmb{e}_k$, we have:
    \begin{equation}
      \pmb{e}_k^\top \pmb{B}^{-1}\tilde{\pmb{B}}\pmb{B}^{-1} \pmb{e}_k =
      (x_k + y_k) \pmb{e}_k^\top \pmb{B}^{-1} \pmb{e}_k,
    \end{equation}
    where we used the fact that $\pmb{e}_k^\top \pmb{Y} = y_k \pmb{e}_k^\top$ and $\pmb{X}\pmb{e}_k = x_k \pmb{e}_k$.
    This completes the proof.
\end{proof}
\begin{theorem}\label{q-max}
    For $1\le k\le n$, each individual soliton $\mathbf{q}_{\mathrm{sol}}^{k+}(x,t)$ appearing in the resolution \eqref{asy-q-l} is indeed a single-soliton solution 
    $\mathbf{q}^{\mathrm{sol}}_{\omega,\gamma,v}(x-x_k^+,t)$ of the form \eqref{soliton-CNLS}, with 
    parameters $(\omega,\gamma,v) = (2\eta_k, \phi_k^+, -2\xi_k)$.
\end{theorem}

\begin{proof}
    According to the expansion \eqref{asy-q-l}, it suffices to demonstrate that 
    $|\tilde{\pmb{c}}_k|=1$. From equations \eqref{def-c-l} and \eqref{def-h-l}, the condition 
    $|\tilde{\pmb{c}}_k|=1$ is equivalent to showing:
    \begin{equation}\label{equal-c-l}
      4\eta_k^2 |\det \pmb{A}_{k:n}| \left| \det \begin{bmatrix} 0 & \pmb{e}_1^\top \\ \pmb{e}_1 & \pmb{A}_{k:n} \end{bmatrix} \right|
      |\det \pmb{B}_{1:k}| \left| \det \begin{bmatrix} \pmb{B}_{1:k} & \pmb{e}_k \\ \pmb{e}_k^\top & 0 \end{bmatrix} \right|
      = \left| \det \begin{bmatrix} 0 & \pmb{f} \\ \pmb{e}_1 & \pmb{A}_{k:n} \end{bmatrix} \right|^2
      \sum_{j=1}^2 \left| \det \begin{bmatrix} \pmb{B}_{1:k} & \pmb{g}_{[j]}^* \\ \pmb{e}_k^\top & 0 \end{bmatrix} \right|^2.
    \end{equation}
    By applying the block matrix determinant identity (Schur complement), we derive:
    \begin{equation}
      \left| \det \begin{bmatrix} 0 & \pmb{e}_1^\top \\ \pmb{e}_1 & \pmb{A}_{k:n} \end{bmatrix} \right| =
      |\det \pmb{A}_{k:n}| |\pmb{e}_1^\top \pmb{A}_{k:n}^{-1} \pmb{e}_1|, \quad
      \left| \det \begin{bmatrix} \pmb{B}_{1:k} & \pmb{e}_k \\ \pmb{e}_k^\top & 0 \end{bmatrix} \right| =
      |\det \pmb{B}_{1:k}| |\pmb{e}_k^\top \pmb{B}_{1:k}^{-1} \pmb{e}_k|,
    \end{equation}
    and
    \begin{equation}
      \left| \det \begin{bmatrix} 0 & \pmb{f} \\ \pmb{e}_1 & \pmb{A}_{k:n} \end{bmatrix} \right|^2 =
      |\det \pmb{A}_{k:n}|^2 |\pmb{f} \pmb{A}_{k:n}^{-1} \pmb{e}_1|^2, \quad
      \sum_{j=1}^2 \left| \det \begin{bmatrix} \pmb{B}_{1:k} & \pmb{g}_{[j]}^* \\ \pmb{e}_k^\top & 0 \end{bmatrix} \right|^2 =
      |\det \pmb{B}_{1:k}|^2 \sum_{j=1}^2 |\pmb{e}_k^\top \pmb{B}_{1:k}^{-1} \pmb{g}_{[j]}^*|^2.
    \end{equation}
    Utilizing the symmetries $\pmb{B}_{1:k} = -\pmb{B}_{1:k}^\dagger$ and 
    $\pmb{A}_{k:n} = -\pmb{A}_{k:n}^\dagger$, equation \eqref{equal-c-l} simplifies to:
    \begin{equation}
      4\eta_k^2 |\pmb{e}_1^\top \pmb{A}_{k:n}^{-1} \pmb{e}_1| |\pmb{e}_k^\top \pmb{B}_{1:k}^{-1} \pmb{e}_k| =
      |\pmb{e}_1^\top \pmb{A}_{k:n}^{-1} \pmb{f}^\dagger \pmb{f} \pmb{A}_{k:n}^{-1} \pmb{e}_1|
      \left| \pmb{e}_k^\top \pmb{B}_{1:k}^{-1} \sum_{j=1}^2 (\pmb{g}_{[j]}^* \pmb{g}_{[j]}^\top) \pmb{B}_{1:k}^{-1} \pmb{e}_k \right|.
    \end{equation}
    From the definition \eqref{def-A-B} of $\pmb{A}_{k:n}$ and $\pmb{B}_{1:k}$, and by applying Lemma \ref{lem-diag}, we obtain:
    \begin{equation}
      2\eta_k |\pmb{e}_1^\top \pmb{A}_{k:n}^{-1} \pmb{e}_1| = |\pmb{e}_1^\top \pmb{A}_{k:n}^{-1} \pmb{f}^\dagger \pmb{f} \pmb{A}_{k:n}^{-1} \pmb{e}_1|, \quad
      2\eta_k |\pmb{e}_k^\top \pmb{B}_{1:k}^{-1} \pmb{e}_k| = \left| \pmb{e}_k^\top \pmb{B}_{1:k}^{-1} \sum_{j=1}^2 (\pmb{g}_{[j]}^* \pmb{g}_{[j]}^\top) \pmb{B}_{1:k}^{-1} \pmb{e}_k \right|.
    \end{equation}
    Multiplying these two identities yields equation \eqref{equal-c-l}. This confirms $|\tilde{\pmb{c}}_k|=1$, which completes the proof.
\end{proof}
The long-time asymptotics of $\mathbf{q}(x,t)$ has been computed explicitly for $t \to +\infty$. Next, we consider the case for $t \to -\infty$. 

\begin{theorem}\label{asy-q-n-neg}
  Let $\mathbf{q}(x,t)$ be the solution to the CNLS equation \eqref{CNLS} with initial data 
  $\mathbf{q}_0(x) \in H^{1,1}(\mathbb{R})$ and associated 
  scattering data $\{\hat{\mathbf{R}}(\lambda), \{\lambda_j, \mathbf{c}_j\}_{j=1}^n\}$. 
  Suppose the discrete eigenvalues $\lambda_j = \xi_j + \mathrm{i}\eta_j$ are ordered such that 
  \begin{equation}
    \xi_1 < \xi_2 < \dots < \xi_n.
  \end{equation}
  As $t \to -\infty$, the long-time asymptotic behavior of $\mathbf{q}(x,t)$ is given by:
  \begin{equation}\label{asy-q-n-neg-1}
    \mathbf{q}(x,t) = \begin{cases}
      \mathbf{q}_{\mathrm{sol}}^{k-}(x,t) + \mathcal{O}(t^{-1/2}), & |\xi - \xi_k| < t^{-1}, \\
      \mathcal{O}(t^{-1/2}), & \text{otherwise},
    \end{cases}
  \end{equation}
  where 
  \begin{equation}\label{q-sol-k-neg}
    \mathbf{q}_{\mathrm{sol}}^{k-}(x,t) = 2\eta_k \operatorname{sech}\left(2\eta_k(x - x_k^- + 2t\xi_k)\right)
    \mathrm{e}^{-2\mathrm{i}(\xi_k x - (\eta_k^2 - \xi_k^2)t) + \mathrm{i}\phi_k^-} \tilde{\pmb{c}}_k^-
  \end{equation}
  is a single-soliton solution of the form \eqref{soliton-CNLS}. The parameters are defined as:
  \begin{equation}\label{def-A-B-neg}
    \begin{pmatrix} \pmb{f}_j \\ \pmb{g}_j \end{pmatrix} = \begin{pmatrix} \Delta_{\xi_k}^-(\lambda_j) & \mathbf{0} \\ \mathbf{0} & (\pmb{\Delta}^-_{\xi_k})^{-1}(\lambda_j) \end{pmatrix} \pmb{v}_j, \quad 
    \pmb{A}_{ij} = \frac{\pmb{f}_i^* \pmb{f}_j}{\lambda_j - \lambda_i^*}, \quad \pmb{B}_{ij} = \frac{\pmb{g}_i^\dagger \pmb{g}_j}{\lambda_j - \lambda_i^*}, \quad \phi_k^- = \frac{n+2}{2}\pi,
  \end{equation}
  and the phase shift $x_k^-$ and polarization vector $\tilde{\pmb{c}}_k^-$ are:
  \begin{equation}\label{def-c-l-neg}
    x_k^- = \frac{1}{4\eta_k} \ln \frac{|\det \pmb{B}_{k:n} \det \pmb{A}_{1:k-1}|}{|\det \pmb{B}_{k+1:n} \det \pmb{A}_{1:k}|}, \quad 
    \tilde{\pmb{c}}_k^- = \frac{\pmb{h}}{2\eta_k \sqrt{|\det \pmb{B}_{k+1:n} \det \pmb{A}_{1:k} \det \pmb{B}_{k:n} \det \pmb{A}_{1:k-1}|}},
  \end{equation}
  with the vector $\pmb{h}$ given by:
  \begin{equation}\label{def-h-l-neg}
    \pmb{h} = \begin{bmatrix}
        \begin{vmatrix} 0 & \pmb{e}_1^\top \\ \pmb{g}_{[1]}^* & \pmb{B}_{k:n} \end{vmatrix}, &
        \begin{vmatrix} 0 & \pmb{e}_1^\top \\ \pmb{g}_{[2]}^* & \pmb{B}_{k:n} \end{vmatrix}
      \end{bmatrix} \begin{vmatrix} \pmb{A}_{1:k} & \pmb{e}_k \\ \pmb{f} & 0 \end{vmatrix},
  \end{equation}
  where $\pmb{f} = [\pmb{f}_1, \dots, \pmb{f}_k]$ and $\pmb{g}_{[j]} = [\pmb{g}_k^{[j]}, \dots, \pmb{g}_n^{[j]}]^\top$.
\end{theorem}

The analysis for the case $t \to -\infty$ 
is analogous to that for $t \to \infty$; 
the detailed proof of Theorem \ref{asy-q-n-neg} 
is provided in Appendix B for completeness.

The unit vectors $\tilde{\pmb{c}}_k^\pm$ represent the asymptotic polarization vectors of soliton-$k$ before and after all its collisions with the other solitons, respectively. As established in previous works \cite{Nsoliton_collision,Soliton_interactions}, the final state $\{\tilde{\pmb{c}}_1^+, \tilde{\pmb{c}}_2^+, \dots, \tilde{\pmb{c}}_n^+\}$ is completely determined by the initial state $\{\tilde{\pmb{c}}_1^-, \tilde{\pmb{c}}_2^-, \dots, \tilde{\pmb{c}}_n^-\}$, regardless of the order of the pairwise collisions. In the CNLS model \eqref{CNLS}, an $n$-soliton collision can be factorized into a nonlinear superposition of $\binom{n}{2}$ pairwise collisions.

However, the factorization property associated with the underlying Yang-Baxter structure is often difficult to discern through the explicit formulas of general $n$-soliton solutions. Following the approach of Caudrelier and Zhang \cite{Yang-Baxter-reflection}, we utilize the dressing method to derive the long-time asymptotic expansion of general $n$-soliton solutions in a new form, allowing us to investigate the factorization property of the CNLS model \eqref{CNLS} more effectively.

\begin{definition}
  Let $\M(\lambda; x, t)$ be a solution to RHP \ref{initial-rhp-without-role}. Let $S_n$ be the symmetric group on the set $\{1, \dots, n\}$, and let $\{i_1, \dots, i_n\} \in S_n$. A general dressing factor of degree 1 is defined as follows for $1 \le k \le n$:
  \begin{equation}
  \begin{aligned}
    \mathbf{T}_{i_k, \{i_{k+1}, \dots, i_n\}}(\lambda) &= \mathbb{I}_3 - \frac{\lambda_{i_k} - \lambda_{i_k}^*}{\lambda - \lambda_{i_k}^*} \frac{\mathbf{\Phi}_{i_k, \{i_{k+1}, \dots, i_n\}} \mathbf{\Phi}_{i_k, \{i_{k+1}, \dots, i_n\}}^\dagger}{\mathbf{\Phi}_{i_k, \{i_{k+1}, \dots, i_n\}}^\dagger \mathbf{\Phi}_{i_k, \{i_{k+1}, \dots, i_n\}}}, \\
    \mathbf{\Phi}_{i_k, \{i_{k+1}, \dots, i_n\}} &= \mathbf{T}_{i_{k+1} \dots i_n}(\lambda_{i_k}) \M(\lambda_{i_k}; x, t) \mathrm{e}^{-\mathrm{i} \lambda_{i_k}(x + t\lambda_{i_k}) \mathbf{\Lambda}_3} \pmb{\beta}_{i_k}, \\
    \mathbf{T}_{i_{k+1} \dots i_n}(\lambda) &= \mathbf{T}_{i_{k+1}, \{i_{k+2}, \dots, i_n\}}(\lambda) \mathbf{T}_{i_{k+2}, \{i_{k+3}, \dots, i_n\}}(\lambda) \dots \mathbf{T}_{i_n}(\lambda), \\
    \pmb{\beta}_{i_k} &= \begin{pmatrix} 1 \\ \tilde{\pmb{\beta}}_{i_k} \end{pmatrix}, \quad \tilde{\pmb{\beta}}_{i_k} = \left( \prod_{l \neq i_k, 1 \le l \le n} \frac{\lambda_{i_k} - \lambda_l}{\lambda_{i_k} - \lambda_l^*} \right) \frac{\mathbf{c}_{i_k}}{(\lambda_{i_k} - \lambda_{i_k}^*)}.
  \end{aligned}
  \end{equation}
\end{definition}

According to Theorem 2.6 in \cite{Yang-Baxter-reflection}, we obtain the following property of a dressing factor of degree $n$.

\begin{theorem}\label{property-dressing-factor}
  (cf. \cite{Yang-Baxter-reflection}) A dressing factor of degree $n$ can be decomposed into $n!$ equivalent products of $n$ dressing factors of degree 1:
  \begin{equation}\label{def-T-1-n}
    \mathbf{T}_{1:n}(\lambda) = 
    \mathbf{T}_{i_1, \{i_{2}:i_n\}}
    (\lambda) \mathbf{T}_{i_2, 
    \{i_{3}:i_n\}}(\lambda) \dots \mathbf{T}_{i_n}(\lambda),
  \end{equation}
  where $\{i_1, \dots, i_n\}$ is an arbitrary permutation of $\{1, \dots, n\}$.
\end{theorem}

Theorem \ref{property-dressing-factor} is crucial for analyzing the long-time asymptotic expansion of multi-soliton solutions for the CNLS equations \eqref{CNLS} via the dressing transformation. We first consider a two-soliton collision. Suppose the discrete spectrum set $\mathcal{Z}$ consists of two poles $\lambda_j = \xi_j + \mathrm{i} \eta_j$ ($j=1,2$) with $\xi_1 < \xi_2$. 

For $|\xi - \xi_1| < t^{-1}$, we introduce the following dressing factors $\mathbf{T}_{1, \{2\}}^{(\pm)}(\lambda; x, t)$:
\begin{equation}\label{relation-M-M1}
    \begin{cases}
      \hat{\M}(\lambda; x, t) = \mathbf{T}_{1, \{2\}}^{(+)}(\lambda; x, t) \hat{\M}^{[1]}(\lambda; x, t) 
      \operatorname{diag}\left(\frac{\lambda - \lambda_{1}^{*}}{\lambda - \lambda_{1}}, 1, 1\right),
      \quad & \mathrm{Im}\, \lambda > 0, \\
      \hat{\M}(\lambda; x, t) = \mathbf{T}_{1, \{2\}}^{(-)}(\lambda; x, t) \hat{\M}^{[1]}(\lambda; x, t) 
      \operatorname{diag}\left(1, \frac{\lambda - \lambda_{1}}{\lambda - \lambda_{1}^{*}}\mathbb{I}_2\right),
      & \mathrm{Im}\, \lambda < 0,
    \end{cases}
\end{equation}
where 
\begin{equation}\label{relation-T+-T-}
		\begin{aligned}
      &\mathbf{T}_{1,\{2\}}^{(+)}(\lambda ; x, t)=\mathbb{I}_3-\frac{\lambda_{1}-\lambda_{1}^*}{\lambda-\lambda_{1}^*} 
      		\mathbf{P}_{1}(x, t), \quad \mathbf{P}_{1}(x, t)=\frac{\mathbf{\Phi}_{1} \mathbf{\Phi}_{1}^{\dagger}}
          {\mathbf{\Phi}_{1}^{\dagger} \mathbf{\Phi}_{1}},\ \ \mathbf{\Phi}_{1}=
          \hat{\M}_+^{[1]}(\lambda_{1} ; x, t) \ee^{-\mathrm{i} 
      			\lambda_{1}\left(x+\lambda_{1} t\right) \pmb{\Lambda}_{3}} 
      		\pmb{v}_{1},\\
      &\mathbf{T}_{1,\{2\}}^{(-)}(\lambda ; x, t)=
		\frac{\lambda-\lambda_{1}^*}
		{\lambda-\lambda_{1}} \mathbf{T}_1^{(+)}(\lambda ; x, t)=\mathbb{I}_3-\frac{\lambda_{1}^{*}-\lambda_{1}}
		{\lambda-\lambda_{1}}\left(\mathbb{I}_3-\mathbf{P}_{1}(x, t)\right),\ \
    \pmb{v}_{1}=\begin{pmatrix}
      1\\
      \frac{\mathbf{c}_1}{\lambda_1-\lambda_1^*}
    \end{pmatrix}
  \end{aligned}
	\end{equation}
and $\hat{\M}^{[1]}(\lambda; x, t)$ satisfies the following RHP:

\begin{rhp}
Find a matrix function $\hat{\M}^{[1]}(\lambda) = \hat{\M}^{[1]}(\lambda; x, t)$ such that:
\begin{enumerate}
  \item Analyticity: $\hat{\M}^{[1]}(\lambda)$ is meromorphic in $\mathbb{C} \setminus \mathbb{R}$.
  \item Jump condition: $\hat{\M}^{[1]}_+(\lambda) = \hat{\M}^{[1]}_-(\lambda) \hat{\V}^{[1]}(\lambda)$ for $\lambda \in \mathbb{R}$, where
    \begin{equation}
        \hat{\V}^{[1]}(\lambda) = \begin{pmatrix}
          1 + \hat{\mathbf{R}}_1(\lambda)\hat{\mathbf{R}}_1^\dagger(\lambda) & -\hat{\mathbf{R}}_1(\lambda)\mathrm{e}^{-2\mathrm{i} t\theta(\lambda)} \\
          -\hat{\mathbf{R}}_1^\dagger(\lambda)\mathrm{e}^{2\mathrm{i} t\theta(\lambda)} & \mathbb{I}_2
        \end{pmatrix}, \quad 
        \hat{\mathbf{R}}_1(\lambda) = \frac{\lambda - \lambda_{1}^{*}}{\lambda - \lambda_{1}} \hat{\mathbf{R}}(\lambda).
    \end{equation}
  \item Residue conditions: $\hat{\M}^{[1]}(\lambda)$ has simple poles at $\lambda_2$ and $\lambda_2^*$ satisfying:
    \begin{equation*}
        \begin{aligned}
          \operatorname*{Res}_{\lambda = \lambda_2} \hat{\M}^{[1]} &= \lim_{\lambda \to \lambda_2} \hat{\M}^{[1]} \begin{pmatrix} 0 & \mathbf{0} \\ -\mathbf{c}_{12}\mathrm{e}^{2\mathrm{i} t\theta} & \mathbf{0} \end{pmatrix}, \\
          \operatorname*{Res}_{\lambda = \lambda_2^*} \hat{\M}^{[1]} &= \lim_{\lambda \to \lambda_2^*} \hat{\M}^{[1]} \begin{pmatrix} 0 & \mathbf{c}_{12}^{\dagger}\mathrm{e}^{-2\mathrm{i} t\theta} \\ \mathbf{0} & \mathbf{0} \end{pmatrix},
        \end{aligned}
    \end{equation*}
    where $\mathbf{c}_{12} = \mathbf{c}_2 \frac{\lambda_2 - \lambda_1}{\lambda_2 - \lambda_1^*}$.
  \item Asymptotics: $\hat{\M}^{[1]}(\lambda) = \mathbb{I}_{3} + \mathcal{O}(\lambda^{-1})$ as $\lambda \to \infty$.
\end{enumerate}
\end{rhp}
The dressing factor $\mathbf{T}_2^{(\pm)}(\lambda ; x, t)$ is described as follows:
\begin{equation}\label{relation-M1-M2}
    \begin{cases}
    \hat{\M}^{[1]}(\lambda ; x, t)=\mathbf{T}_2^{(+)}(\lambda ; x, t) \M(\lambda ; x, t) 
      \operatorname{diag}\left(\frac{\lambda-\lambda_{2}^{*}}{\lambda-\lambda_{2}}, 1,1\right),
      \quad & \mathrm{Im}\, \lambda > 0, \\
      \hat{\M}^{[1]}(\lambda ; x, t)=\mathbf{T}_2^{(-)}(\lambda ; x, t) \M(\lambda ; x, t) 
      \operatorname{diag}\left(1, \frac{\lambda-\lambda_{2}}{\lambda-\lambda_{2}^{*}}\mathbb{I}_2\right),
      & \mathrm{Im}\, \lambda < 0,
    \end{cases}
\end{equation}
where 
\begin{equation}
    \begin{aligned}
      \mathbf{T}_2^{(+)}(\lambda ; x, t) &= \mathbb{I}_3-\frac{\lambda_{2}-\lambda_{2}^*}{\lambda-\lambda_{2}^*} 
          \mathbf{P}_{2}(x, t), \quad \mathbf{P}_{2}(x, t)=\frac{\mathbf{\Phi}_{2} \mathbf{\Phi}_{2}^{\dagger}}
          {\mathbf{\Phi}_{2}^{\dagger} \mathbf{\Phi}_{2}}, \quad
      \mathbf{\Phi}_{2} = \M(\lambda_{2} ; x, t) \mathrm{e}^{-\mathrm{i} \lambda_{2}(x+\lambda_{2} t) \pmb{\Lambda}_{3}} \pmb{v}_{2}, \\
      \mathbf{T}_2^{(-)}(\lambda ; x, t) &= \frac{\lambda-\lambda_{2}^*}{\lambda-\lambda_{2}} \mathbf{T}_2^{(+)}(\lambda ; x, t) = \mathbb{I}_3-\frac{\lambda_{2}^{*}-\lambda_{2}}{\lambda-\lambda_{2}}(\mathbb{I}_3-\mathbf{P}_{2}(x, t)), \quad
      \pmb{v}_{2} = \begin{pmatrix} 1 \\ \frac{\mathbf{c}_{12}}{\lambda_2-\lambda_2^*} \end{pmatrix},
    \end{aligned}
\end{equation}
and $\M(\lambda ; x, t)$ is the solution for RHP \ref{initial-rhp-without-role}.

Next, we evaluate the projectors $\mathbf{P}_{j}(x, t)$, for $j=1,2$, as $t \to \pm \infty$. 
It follows from equations \eqref{expansion-M-lambda} and \eqref{M-lambda-neg} that:
\begin{equation}
  \mathbf{P}_{2}(x, t) =
  \begin{cases}
    \begin{pmatrix} 1 & \mathbf{0} \\ \mathbf{0} & \mathbf{0} \end{pmatrix} + \mathcal{O}(t^{-1/2}), & t \to +\infty, \\
    \begin{pmatrix} 0 & \mathbf{0} \\ \mathbf{0} & \frac{\pmb{\beta}_2\pmb{\beta}_2^\dagger}{\pmb{\beta}_2^\dagger\pmb{\beta}_2} \end{pmatrix} + \mathcal{O}(t^{-1/2}), & t \to -\infty,
  \end{cases}
\end{equation}
where the vector $\pmb{\beta}_2$ is defined as:
\begin{equation}
  \pmb{\beta}_2 = \frac{\lambda_2-\lambda_1}{\lambda_2-\lambda_1^*} \frac{[\pmb{\Delta}_{\xi_1}^-(\lambda_2)]^{-1}\mathbf{c}_2}{\lambda_2-\lambda_2^*}.
\end{equation}
Hence, we can evaluate $\mathbf{T}_2^{(+)}(\lambda)$ as follows:
\begin{equation}\label{evaluate-T2}
   \mathbf{T}_2^{(+)}(\lambda) =
    \begin{cases}
      \begin{pmatrix} \frac{\lambda-\lambda_2}{\lambda-\lambda_2^*} & \mathbf{0} \\ \mathbf{0} & \mathbb{I}_2 \end{pmatrix} + \mathcal{O}(t^{-1/2}), & t \to +\infty, \\
      \begin{pmatrix} 1 & \mathbf{0} \\ \mathbf{0} & \mathbb{I}_2-\frac{\lambda_2-\lambda_2^*}{\lambda-\lambda_2^*} \frac{\pmb{\beta}_2\pmb{\beta}_2^\dagger}{\pmb{\beta}_2^\dagger\pmb{\beta}_2} \end{pmatrix} + \mathcal{O}(t^{-1/2}), & t \to -\infty,
    \end{cases}
\end{equation}
and then we obtain the asymptotic behavior of $\mathbf{\Phi}_{1}$:
\begin{equation}\label{evaluate-phi1}
    \mathbf{\Phi}_{1} =
    \begin{cases}
      \mathrm{e}^{-\mathrm{i} \lambda_{1}(x+\lambda_{1} t)\pmb{\Lambda}_{3}} \pmb{\beta}_1^+ + \mathcal{O}(t^{-1/2}), \quad \pmb{\beta}_1^+ = \begin{pmatrix} \Delta_{\xi_1}^+(\lambda_1) \\ \frac{[\pmb{\Delta}_{\xi_1}^+(\lambda_1)]^{-1}\mathbf{c}_1}{\lambda_1-\lambda_1^*} \end{pmatrix}, & t \to +\infty, \\
      \mathrm{e}^{-\mathrm{i} \lambda_{1}(x+\lambda_{1} t) \pmb{\Lambda}_{3}} \pmb{\beta}_1^- + \mathcal{O}(t^{-1/2}), \quad \pmb{\beta}_1^- = \begin{pmatrix} \frac{\lambda_1-\lambda_2^*}{\lambda_1-\lambda_2}\Delta_{\xi_1}^-(\lambda_1) \\ \left[\mathbb{I}_2-\frac{\lambda_2-\lambda_2^*}{\lambda_1-\lambda_2^*} \frac{\pmb{\beta}_2\pmb{\beta}_2^\dagger}{\pmb{\beta}_2^\dagger\pmb{\beta}_2}\right] \frac{[\pmb{\Delta}_{\xi_1}^-(\lambda_1)]^{-1}\mathbf{c}_1}{\lambda_1-\lambda_1^*} \end{pmatrix}, & t \to -\infty,
    \end{cases}
\end{equation}
by equation \eqref{relation-T+-T-}.
It follows from equations \eqref{recover-CNLS}, \eqref{expansion-M1}, \eqref{relation-M-M1}, 
\eqref{relation-M1-M2}, \eqref{evaluate-T2} and \eqref{evaluate-phi1} that, in the region 
$|\xi-\xi_1|<t^{-1}$, the solution $\mathbf{q}(x,t)$ can be expanded as:
\begin{equation}
  \begin{aligned}
    \mathbf{q}(x,t) &= \lim_{\lambda\rightarrow\infty} 
    (2\lambda \mathbf{T}_{1,\{2\}}^{(+)}(\lambda ; x, t))_{12} + \mathcal{O}(t^{-1/2}) \\
    &= -4\mathrm{i} \eta_1 \frac{\Phi_{1,1}\mathbf{\Phi}_{1,23}^{\dagger}}
    {\mathbf{\Phi}_{1}^\dagger\mathbf{\Phi}_{1}} + \mathcal{O}(t^{-1/2}) \\
    &= 2\eta_1\operatorname{sech}\left(2\eta_1\left(x-x_1^\pm+2t\xi_1\right)\right)
    \mathrm{e}^{-2\mathrm{i}\left(\xi_1x-(\eta_1^2-\xi_1^2)t\right)-\frac{\mathrm{i}\pi}{2}}\pmb{v}_1^\pm
    + \mathcal{O}(t^{-1/2}),
  \end{aligned}
\end{equation}
where $\Phi_{1,1}$ denotes the first element of $\mathbf{\Phi}_1$, 
$\mathbf{\Phi}_{1,23}$ denotes the vector formed by the last two elements of $\mathbf{\Phi}_1$, and 
the phase parameters are given by $x_1^\pm = \frac{1}{2\eta_1}\ln|\pmb{\gamma}^\pm_1|$ and 
$\pmb{v}_1^\pm = \frac{(\pmb{\gamma}^\pm_1)^\dagger}{|\pmb{\gamma}^\pm_1|}$. The vectors $\pmb{\gamma}^\pm_1$ are defined as follows:
\begin{equation}\label{relation-gamma-two-collision-1}
  \begin{aligned}
    \pmb{\gamma}^+_1 &= \frac{[\pmb{\Delta}_{\xi_1}^+(\lambda_1)]^{-1}\mathbf{c}_1}
    {\Delta_{\xi_1}^+(\lambda_1)(\lambda_1-\lambda_1^*)}, \\
    \pmb{\gamma}^-_1 &= \left[\mathbb{I}_2-\frac{\lambda_2-\lambda_2^*}{\lambda_1-\lambda_2^*}
    \frac{\pmb{\beta}_2\pmb{\beta}_2^\dagger}{\pmb{\beta}_2^\dagger\pmb{\beta}_2}\right]
    \frac{[\pmb{\Delta}_{\xi_1}^-(\lambda_1)]^{-1}\mathbf{c}_1}
    {\Delta_{\xi_1}^-(\lambda_1)(\lambda_1-\lambda_1^*)}
    \frac{\lambda_1-\lambda_2}{\lambda_1-\lambda_2^*}, \\
    \pmb{\beta}_2 &= \frac{\lambda_2-\lambda_1}{\lambda_2-\lambda_1^*}
    \frac{[\pmb{\Delta}_{\xi_1}^-(\lambda_2)]^{-1}\mathbf{c}_2}{\lambda_2-\lambda_2^*}.
  \end{aligned}
\end{equation}
Similarly, in the region $|\xi-\xi_2|<t^{-1}$, we obtain the long-time asymptotics:
\begin{equation}\label{relation-gamma-two-collision-2}
  \mathbf{q}(x,t) = 2\eta_2\operatorname{sech}\left(2\eta_2(x-x_2^\pm+2t\xi_2)\right)
    \mathrm{e}^{-2\mathrm{i}(\xi_2x-(\eta_2^2-\xi_2^2)t)-\frac{\mathrm{i}\pi}{2}}
    \pmb{v}_2^\pm + \mathcal{O}(t^{-1/2}),
\end{equation}
where $x_2^\pm = \frac{1}{2\eta_2}\ln|\pmb{\gamma}^\pm_2|$ and 
$\pmb{v}_2^\pm = \frac{(\pmb{\gamma}^\pm_2)^\dagger}{|\pmb{\gamma}^\pm_2|}$. The corresponding vectors $\pmb{\gamma}^\pm_2$ are:
\begin{equation}
 \begin{aligned}
  \pmb{\gamma}^+_2 &= \left[\mathbb{I}_2-\frac{\lambda_1-\lambda_1^*}{\lambda_2-\lambda_1^*}
      \frac{\pmb{\beta}_1\pmb{\beta}_1^\dagger}{\pmb{\beta}_1^\dagger\pmb{\beta}_1}\right]
      \frac{[\pmb{\Delta}_{\xi_2}^+(\lambda_2)]^{-1}\mathbf{c}_2}
      {\Delta_{\xi_2}^+(\lambda_2)(\lambda_2-\lambda_2^*)}
      \frac{\lambda_2-\lambda_1}{\lambda_2-\lambda_1^*}, \\
  \pmb{\gamma}^-_2 &= \frac{[\pmb{\Delta}_{\xi_2}^-(\lambda_2)]^{-1}\mathbf{c}_2}
        {\Delta_{\xi_2}^-(\lambda_2)(\lambda_2-\lambda_2^*)}, \\
  \pmb{\beta}_1 &= \frac{\lambda_1-\lambda_2}{\lambda_1-\lambda_2^*}
    \frac{[\pmb{\Delta}_{\xi_2}^+(\lambda_1)]^{-1}\mathbf{c}_1}{\lambda_1-\lambda_1^*}.
 \end{aligned}
\end{equation}
Inspired by the collision of two solitary waves, we introduce the following 
definition in preparation for studying the collisions of $n$ solitary waves:

\begin{definition}\label{def-D-pm-lemma}
  Define $\pmb{D}^{+}_{1:k-1}(\lambda)$ and $\pmb{D}^{-}_{k+1:n}(\lambda)$ 
  recursively as follows:
  \begin{equation}
     \begin{aligned}
      \pmb{D}^+_{1: k-1}(\lambda) &= \pmb{D}^+_{1,\{2: k-1\}}(\lambda) \pmb{D}^+_{2,\{3: k-1\}}(\lambda) \cdots \pmb{D}^+_{k-1}(\lambda), \\
      \pmb{D}^-_{k+1: n}(\lambda) &= \pmb{D}^-_{k+1,\{k+2: n\}}(\lambda) \pmb{D}^-_{k+2,\{k+3: n\}}(\lambda) \cdots \pmb{D}^-_{n}(\lambda),
     \end{aligned}
  \end{equation}
  where, for $1 \le j < k < m \le n$,
  \begin{equation}\label{def-D-pm}
    \begin{aligned}
      \pmb{D}^+_{j,\{j+1: k-1\}}(\lambda) &= \mathbb{I}_2 - \frac{\lambda_j-\lambda_j^*}{\lambda-\lambda_j^*} \frac{\pmb{\varphi}_j\pmb{\varphi}_j^\dagger}{\pmb{\varphi}_j^\dagger\pmb{\varphi}_j}, \quad \pmb{\varphi}_j = \pmb{D}^+_{j+1: k-1}(\lambda_j) \pmb{\beta}_{kj}^{+}, \\
      \pmb{D}^-_{m,\{m+1: n\}}(\lambda) &= \mathbb{I}_2 - \frac{\lambda_m-\lambda_m^*}{\lambda-\lambda_m^*} \frac{\pmb{\phi}_m\pmb{\phi}_m^\dagger}{\pmb{\phi}_m^\dagger\pmb{\phi}_m}, \quad \pmb{\phi}_m = \pmb{D}^-_{m+1: n}(\lambda_m) \pmb{\beta}_{km}^{-}, \\
      \pmb{\beta}_{kl}^{\pm} &= \left( \prod_{j \neq l, 1 \le j \le n} \frac{\lambda_l-\lambda_j}{\lambda_l-\lambda_j^*} \right) \frac{[\pmb{\Delta}_{\xi_k}^\pm(\lambda_l)]^{-1} \mathbf{c}_l}{\Delta^\pm_{\xi_k}(\lambda_l) (\lambda_l-\lambda_l^*)}.
    \end{aligned}
  \end{equation}
  Furthermore, define the total dressing factor $\pmb{D}_{1:n}(\lambda)$ as:
\begin{equation}
   \pmb{D}_{1:n}(\lambda) = \pmb{D}_{1,\{2:n\}}(\lambda) \pmb{D}_{2,\{3:n\}}(\lambda) \cdots \pmb{D}_{n}(\lambda),
\end{equation}
where, for $1 \le j \le n$,
\begin{equation}\label{def-D}
  \pmb{D}_{j,\{j+1:n\}}(\lambda) = \mathbb{I}_2 - \frac{\lambda_j-\lambda_j^*}{\lambda-\lambda_j^*} \frac{\pmb{\varphi}_j\pmb{\varphi}_j^\dagger}{\pmb{\varphi}_j^\dagger\pmb{\varphi}_j}, \quad \pmb{\varphi}_j = \pmb{D}_{j+1:n}(\lambda_j) \pmb{\beta}_{j},
\end{equation}
with the initial vector $\pmb{\beta}_{j}$ defined as:
\begin{equation}
  \pmb{\beta}_{j} = \left( \prod_{l \neq j, 1 \le l \le n} \frac{\lambda_j-\lambda_l}{\lambda_j-\lambda_l^*} \right) \frac{\mathbf{c}_j}{(\lambda_j-\lambda_j^*)}.
\end{equation}
\end{definition}

\begin{remark}
  The notation $\frac{\pmb{\varphi}_j\pmb{\varphi}_j^\dagger}{\pmb{\varphi}_j^\dagger\pmb{\varphi}_j}$ 
  represents a projection operator that is invariant under scaling. This implies that for any complex scalar $c \in \mathbb{C} \setminus \{0\}$, 
  the replacement $\pmb{\varphi}_j \to c \pmb{\varphi}_j$ leaves the projector unchanged:
  \begin{equation*}
    \frac{\pmb{\varphi}_j\pmb{\varphi}_j^\dagger}{\pmb{\varphi}_j^\dagger\pmb{\varphi}_j}
    = \frac{(c\pmb{\varphi}_j)(c\pmb{\varphi}_j)^\dagger}{(c\pmb{\varphi}_j)^\dagger(c\pmb{\varphi}_j)}.
  \end{equation*}
\end{remark}

We can then use the notation introduced in Definition \ref{def-D-pm-lemma} to describe the 
$n$-soliton collision in the CNLS system \eqref{CNLS}. The following proposition 
shows that as $t\to\pm\infty$, 
the solution $\mathbf{q}(x,t)$ corresponding to scattering data 
\begin{equation*}
  \mathcal{D}_n(\mathbf{q}_0)=
    \left\{\hat{\mathbf{R}}(\lambda),\left\{\lambda_i
    \right\}_{i=1}^n, \left\{\mathbf{c}_i\right\}_{i=1}^n\right\}
\end{equation*}
asymptotically decouples into a sum of $n$ single-soliton solutions, 
whose polarizations are given in the product form.

\begin{prop}\label{nls-n-soliton-iteration-form}
  Suppose without loss of generality that $\lambda_j=\xi_j+\mathrm{i}\eta_j$ and 
  $\xi_1<\xi_2<\dots<\xi_n$. Let $\mathbf{q}(x,t)$ be 
  a solution for the CNLS equation \eqref{CNLS} with 
  initial data $\mathbf{q}_0(x)\in H^{1,1}(\mathbb{R})$ and 
  $\mathcal{D}_n(\mathbf{q}_0)=
  \left\{\hat{\mathbf{R}}(\lambda),\left\{\lambda_i
  \right\}_{i=1}^n, \left\{\mathbf{c}_i\right\}_{i=1}^n\right\}$. Denote 
  $\mathbf{q}^\pm(x,t)=\lim_{t\to\pm\infty}\mathbf{q}(x,t)$. Then we have:
  \begin{equation}\label{asymptotic-n-soliton}
  \mathbf{q}^\pm(x,t)=\sum_{k=1}^{n}2\eta_k\operatorname{sech}\left(2\eta_k 
  (x-x_k^\pm+2t\xi_k)\right)
    \mathrm{e}^{-2\mathrm{i}(\xi_kx-(\eta_k^2-\xi_k^2)t)-\frac{\mathrm{i}\pi}{2}}\pmb{v}_k^\pm
    +\mathcal{O}(t^{-1/2}),
\end{equation}
where $x_k^\pm=\frac{1}{2\eta_k}\ln|\pmb{\gamma}^\pm_k|$, 
$\pmb{v}_k^\pm=\frac{(\pmb{\gamma}^\pm_k)^\dagger}{|\pmb{\gamma}^\pm_k|}$ and 
\begin{equation}
  \pmb{\gamma}^+_k=\left(\prod_{j=k+1}^n \frac{\lambda_k-\lambda_j^*}{\lambda_k-\lambda_j}\right) \pmb{D}^{+}_{1: k-1}(\lambda_k) \pmb{\beta}_{kk}^{+}, \quad 
  \pmb{\gamma}^-_k=\left(\prod_{j=1}^{k-1} \frac{\lambda_k-\lambda_j^*}{\lambda_k-\lambda_j}\right) \pmb{D}^{-}_{k+1: n}(\lambda_k) \pmb{\beta}_{kk}^{-}.
\end{equation}
\end{prop}
\begin{proof}
  We apply the following transformation to the solution $\hat{\M}(\lambda;x,t)$ of 
  RHP \ref{initial-rhp-with-pole}:
  \begin{equation}\label{relation-hatM-M-T}
    \begin{cases}
      \hat{\M}(\lambda ; x, t)=\mathbf{T}_{1: n}^{(+)}(\lambda ; x, t) \M(\lambda ; x, t) 
      \mathrm{diag}\left(\prod\limits_{k=1}^{n}\frac{\lambda-\lambda_{k}^{*}}{\lambda-\lambda_{k}},
       \mathbb{I}_2\right),
      \quad&{\rm Im} \lambda>0,\\
      \hat{\M}(\lambda ; x, t)=\mathbf{T}_{1: n}^{(-)}(\lambda ; x, t) \M(\lambda ; x, t) 
      \mathrm{diag}\left(1, \prod\limits_{k=1}^{n}\frac{\lambda-\lambda_{k}}
      {\lambda-\lambda_{k}^{*}}\mathbb{I}_2\right),
      &{\rm Im} \lambda<0,
    \end{cases}
  \end{equation}
  where $\mathbf{T}_{1: n}^{(+)}(\lambda ; x, t)$ is defined by equation \eqref{def-T-1-n} and 
  \begin{equation}
    \mathbf{T}_{1: n}^{(-)}(\lambda ; x, t)=
    \prod\limits_{k=1}^{n}\frac{\lambda-\lambda_{k}^{*}}{\lambda-\lambda_{k}}
    \mathbf{T}_{1: n}^{(+)}(\lambda ; x, t).
  \end{equation}
  Substituting equations \eqref{expansion-M1} and \eqref{relation-hatM-M-T} into 
  the reconstruction formula \eqref{recover-CNLS} yields
  \begin{equation}\label{recover-iteration}
      \mathbf{q}(x,t)=\lim_{\lambda\rightarrow\infty} 
      (2\lambda \mathbf{T}_{1: n}^{(+)}(\lambda ; x, t))_{12}+\oo(t^{-1/2}).
  \end{equation}
  In fact, $\mathbf{q}(x,t)$ approaches an individual single-soliton 
  solution within each region $\left|\xi-\xi_k\right|<|t|^{-1}$ for $1\le k\le n$, 
  while vanishing exponentially elsewhere in the $(x, t)$-plane. 
  Therefore, it suffices to estimate $\mathbf{T}^{(+)}_{1:n}(\lambda)$ strictly within 
  these local neighborhoods $\left|\xi-\xi_k\right|<|t|^{-1}$, 
  for $1\le k\le n$, as $t\to \pm\infty$. We provide the detailed analysis solely for the case 
  $\left|\xi-\xi_k\right|<t^{-1}$ as $t\to+\infty$, since the remaining cases follow analogously. 
  For brevity, we denote $\mathbf{T}_{1:n}^{(+)}(\lambda;x,t)=\mathbf{T}_{1:n}(\lambda)$.
  
  By virtue of Theorem \ref{property-dressing-factor}, 
  we can rewrite $\mathbf{T}_{1:n}(\lambda)$ in the following factorized form:
  \begin{equation}
    \mathbf{T}_{1:n}(\lambda)=\mathbf{T}_{k,\{1,\cdots,k-1,k+1,\cdots n\}}(\lambda)
    \mathbf{T}_{1,\{2,\cdots,k-1,k+1,\cdots n\}}(\lambda)\cdots
    \mathbf{T}_{n}(\lambda).
  \end{equation}
  Consequently, for $j\neq k$, the factors $\mathbf{T}_{j,\{j+1:k-1,k+1:n\}}(\lambda)$ 
  asymptotically approach diagonal matrices as $t\to+\infty$. 
  Utilizing equation \eqref{expansion-M-lambda}, alongside Lemma \ref{det-delta-lambda1} 
  and Lemma \ref{delta-lambda1}, we obtain
  \begin{equation}\label{evaluate-Tj-2-n}
       \begin{cases}
        \mathbf{T}_{j,\{j+1:n\}}(\lambda)=
            \begin{pmatrix}
              \frac{\lambda-\lambda_{j}^*}{\lambda-\lambda_{j}^*}
              &\mathbf{0}\\
              \mathbf{0}&\mathbb{I}_2
            \end{pmatrix}
        +\oo(t^{-1/2}),&j>k,\\
        \mathbf{T}_{j,\{j+1:k-1,k+1:n\}}(\lambda)=
      \begin{pmatrix}
      1&\mathbf{0}\\
      \mathbf{0}&\left[\mathbb{I}_2-\frac{\lambda_{j}-\lambda_{j}^*}
      {\lambda-\lambda_{j}^*}
      \frac{\pmb{\beta}_{j}\pmb{\beta}_{j}^\dagger}
      {\pmb{\beta}_{j}^\dagger\pmb{\beta}_{j}}\right]
      \end{pmatrix}+\oo(t^{-1/2}),&j<k,
    \end{cases}
\end{equation}
where 
\begin{equation}
  \pmb{\beta}_{j}=\pmb{D}^+_{j+1:k-1}(\lambda_{j})
  \pmb{\beta}_{kj}^{+},
\end{equation}
and $\pmb{\beta}_{kj}^{+}$ 
is defined by equation \eqref{def-D-pm}.
Hence, we can deduce
\begin{equation}
  \mathbf{T}_{2:k-1,k+1:n}(\lambda)=\begin{pmatrix}
    \prod\limits_{j>k}\frac{\lambda-
    \lambda_{j}}{\lambda-\lambda_{j}^*}&\mathbf{0}\\
    \mathbf{0}&\pmb{D}^+_{2: k-1}(\lambda)
  \end{pmatrix}+\oo(t^{-1/2}),
\end{equation}
which subsequently yields
   \begin{multline}\label{evaluate-Phi-k}
    \mathbf{\Phi}_{k,\{2:k-1,k+1:n\}}=\begin{pmatrix}
     \prod\limits_{j>k}\frac{\lambda_k-\lambda_{j}}
     {\lambda_k-\lambda_{j}^*}&\mathbf{0}\\
     \mathbf{0}&\pmb{D}^+_{2:k-1}(\lambda_k)
   \end{pmatrix}\\
   \begin{pmatrix}
       \Delta^+_{\xi_k}(\lambda_k)&\mathbf{0}\\
       \mathbf{0}&\left[\pmb{\Delta}_{\xi_k}^+(\lambda_k)
     \right]^{-1}
     \end{pmatrix}
     \ee^{-\ii\lambda_{k}( x+\ii t\lambda_{k} )\mathbf{\Lambda}_3}
     \begin{pmatrix}
       \prod\limits_{j\neq k}^{n}\frac{\lambda_k-\lambda_{j}^{*}}{\lambda_k-\lambda_{j}}\\
     \frac{\mathbf{c}_k}{(\lambda_k-\lambda_k^*)}
     \end{pmatrix}
     +\oo(t^{-1/2}).
   \end{multline}
   Combining equations \eqref{recover-iteration}, \eqref{evaluate-Tj-2-n}, 
   and \eqref{evaluate-Phi-k}, we arrive at
   \begin{equation}
    \mathbf{q}(x,t)=-4\eta_k \begin{bmatrix}
      \frac{\mathbf{\Phi}_{k}
      \mathbf{\Phi}_{k}^\dagger}
      {\mathbf{\Phi}_{k}^\dagger
      \mathbf{\Phi}_{k}}
    \end{bmatrix}_{12}+\oo(t^{-1/2}),
   \end{equation}
   where 
   \begin{equation}
    \mathbf{\Phi}_{k}=\ee^{-\ii\lambda_{k}( x+t\lambda_{k} )
    \mathbf{\Lambda}_3}\begin{pmatrix}
      1\\  \pmb{\gamma}^+_k
    \end{pmatrix}.
   \end{equation}
   A direct calculation then reveals that within the region $\left|\xi-\xi_k\right|<t^{-1}$, 
   as $t\to+\infty$, the solution reduces to
   \begin{equation}
    \mathbf{q}^+(x,t)=2\eta_k \mathrm{sech}\left(2\eta_k 
  \left(x-x_k^++2t\xi_k\right)\right)
    \ee^{-2\ii\left(\xi_kx-\left(\eta_k^2-\xi_k^2\right)t\right)
    -\frac{\ii\pi}{2}}\pmb{v}_k^+
    +\oo(t^{-1/2}).
   \end{equation}
   The identical procedure can be applied to extract the asymptotic forms $\mathbf{q}^\pm(x,t)$ 
   in the remaining regions $\left|\xi-\xi_k\right|<t^{-1}$ as $t\to\pm\infty$. 
   This completes the proof. 
\end{proof}

Notably, the polarization vectors $\tilde{\pmb{c}}_l^\pm$ in equation \eqref{asy-q-l} are expressed in a fractional form, whereas the polarization vectors $\pmb{v}_k^\pm$ in equation \eqref{asymptotic-n-soliton} are presented in a product form. In the reflectionless case where the reflection coefficient $\hat{\mathbf{R}}(\lambda)\equiv0$, it follows that $\M(\lambda;x,t)=\mathbb{I}_3$. Under this condition, the long-time asymptotic expansion of the pure multi-soliton solution for the CNLS equation \eqref{CNLS} can be readily obtained from Proposition \ref{nls-n-soliton-iteration-form} by replacing $\Delta_{\xi_j}(\lambda_k)$ and $\pmb{\Delta}_{\xi_j}(\lambda_k)$ with $1$ and $\mathbb{I}_2$, respectively.
In the following lemma, we demonstrate that the solution $\mathbf{q}(x,t)$ corresponding to the initial scattering data 
$\mathcal{D}_n(\mathbf{q}(x,0))=\left\{\hat{\mathbf{R}}(\lambda),\left\{\lambda_i\right\}_{i=1}^n, \left\{\mathbf{c}_i\right\}_{i=1}^n\right\}$ 
asymptotically approaches a pure $n$-soliton solution $\mathbf{q}_{sol}^{n\pm}(x,t)$ as $t\to\pm\infty$.

\begin{lemma}\label{asy-pure-n soliton}
  Suppose without loss of generality that $\lambda_j=\xi_j+\ii\eta_j$ and 
  $\xi_1<\xi_2<\cdots<\xi_n$. Let $\mathbf{q}(x,t)$ be 
  a solution to the CNLS equation \eqref{CNLS} with initial data $\mathbf{q}_0(x)\in H^{1,1}(\mathbb{R})$ and 
  scattering data $\mathcal{D}_n(\mathbf{q}_0)=
  \left\{\hat{\mathbf{R}}(\lambda),\left\{\lambda_i
  \right\}_{i=1}^n, \left\{\mathbf{c}_i\right\}_{i=1}^n\right\}$. 
  Then there exist two pure $n$-soliton solutions $\mathbf{q}_{sol}^{n\pm}(x,t)$ 
  such that 
  \begin{equation}
    \left\|\mathbf{q}(\cdot,t)-\mathbf{q}_{sol}^{n\pm}(\cdot,t)\right\|_{
    L^\infty(\mathbb{R})}\lesssim t^{-1/2},\ \ \text{as}\,\ t\to\pm\infty.
  \end{equation}
  Furthermore, these asymptotic pure $n$-soliton solutions $\mathbf{q}_{sol}^{n\pm}(x,t)$ 
  are uniquely determined by the modified scattering data:
  \begin{equation}
    \mathcal{D}_n(\mathbf{q}_{sol}^{n\pm}(x,0))=
  \left\{0,\left\{\lambda_i
  \right\}_{i=1}^n, \left\{\tilde{\mathbf{c}}_i^\pm
  \right\}_{i=1}^n\right\},
  \end{equation}
  where the norming constants are explicitly given by
  \begin{equation}\label{def-scattering-data-positive}
    \tilde{\mathbf{c}}_1^+=
    \frac{\left[\pmb{\Delta}_{\xi_1}^+(\lambda_1)
    \right]^{-1}\mathbf{c}_1}{\Delta^+_{\xi_1}(\lambda_1)},\,\
    \tilde{\mathbf{c}}_k^+=\pmb{D}_{1:k-1}^{-1}(\lambda_k)
    \pmb{D}_{1:k-1}^+(\lambda_k)
    \frac{\left[\pmb{\Delta}_{\xi_k}^+(\lambda_k)
    \right]^{-1}\mathbf{c}_k}{\Delta^+_{\xi_k}(\lambda_k)},\,\
    k=2,\cdots,n,
  \end{equation}
  and 
  \begin{equation}\label{def-scattering-data-neg}
    \tilde{\mathbf{c}}_n^-=
    \frac{\left[\pmb{\Delta}_{\xi_n}^-(\lambda_n)
    \right]^{-1}\mathbf{c}_n}{\Delta^-_{\xi_n}(\lambda_n)},\,\
    \tilde{\mathbf{c}}_k^-=\pmb{D}_{k+1:n}^{-1}(\lambda_k)
    \pmb{D}_{k+1:n}^-(\lambda_k)
    \frac{\left[\pmb{\Delta}_{\xi_k}^-(\lambda_k)
    \right]^{-1}\mathbf{c}_k}{\Delta^-_{\xi_k}(\lambda_k)},\,\
    k=1,\cdots,n-1.
  \end{equation}
  Here, $\pmb{D}_{1:n}^\pm(\lambda)$ is defined by equation \eqref{def-D-pm}, and $\pmb{D}_{1:n}(\lambda)$ is defined by equation \eqref{def-D}.
\end{lemma}

\begin{proof}
  According to equation \eqref{asymptotic-n-soliton}, the long-time asymptotic behavior of the multi-soliton solution for the CNLS equation \eqref{CNLS} is entirely determined by the spectral parameters $\lambda_k$ and the vectors $\pmb{\gamma}_k^\pm$. 
  Therefore, it suffices to construct appropriate scattering data $\left\{0,\left\{\lambda_i\right\}_{i=1}^n, \left\{\tilde{\mathbf{c}}_i^\pm\right\}_{i=1}^n\right\}$ such that $\mathbf{q}(x,t)$ converges uniformly to the corresponding pure $n$-soliton solutions as $t\to\pm\infty$. 
  We present the details for the limit $t\to+\infty$, as the argument for $t\to-\infty$ proceeds analogously. 
  
  Since the vectors $\pmb{\gamma}_1^+$ and $\tilde{\pmb{\gamma}}_1^+$ for the pure soliton case are independent of the dressing operators $\pmb{D}^+$ and $\pmb{D}$, we can directly set
  \begin{equation}
    \tilde{\mathbf{c}}_1^+=
    \frac{\left[\pmb{\Delta}_{\xi_1}^+(\lambda_1)
    \right]^{-1}\mathbf{c}_1}{\Delta^+_{\xi_1}(\lambda_1)},
  \end{equation}
  which ensures that $\tilde{\pmb{\gamma}}_1^+=\pmb{\gamma}_1^+$.   
  Proceeding iteratively, we can then determine $\pmb{D}_{1\cdots k-1}(\lambda_k)$ and the remaining constants $\tilde{\mathbf{c}}_k^+$ sequentially for $k=2,\cdots,n$, enforcing the matching condition:
  \begin{equation}
    \tilde{\pmb{\gamma}}_k^+=\pmb{\gamma}_k^+,\,\ \forall
    k\in\{1,2,\cdots,n\}.
  \end{equation}
  This matching guarantees that the pure $n$-soliton solution $\mathbf{q}_{sol}^{n+}(x,t)$, associated with the reflectionless scattering data $\left\{0,\left\{\lambda_i,\tilde{\mathbf{c}}_i^+\right\}_{i=1}^n\right\}$, exhibits the exact same asymptotic expansion as derived in equation \eqref{asymptotic-n-soliton}. 
  Consequently, we establish the uniform error bound
  \begin{equation}
    \left\|\mathbf{q}(\cdot,t)-\mathbf{q}_{sol}^{n+}(\cdot,t)\right\|_{
    L^\infty(\mathbb{R})}\lesssim t^{-1/2},\ \ \text{as}\,\ t\to+\infty,
  \end{equation}
  which completes the proof.
\end{proof}

Combining equation \eqref{asy-q-l}, Theorem \ref{q-max}, Theorem \ref{asy-q-n-neg}, Proposition \ref{nls-n-soliton-iteration-form}, and Lemma \ref{asy-pure-n soliton}, we finally arrive at the proof of Theorem \ref{main-result-asymptotic-stability-n-soliton}. 
Furthermore, the representation derived in Proposition \ref{nls-n-soliton-iteration-form} paves the way for a detailed study of the factorization property of $n$-soliton collisions within the CNLS model \eqref{CNLS}.

\subsection{Factorization and Yang-Baxter map}
Vector-soliton interactions can also be analyzed within the framework of 
Yang-Baxter maps. An operator $\mathcal{R}$ is defined as a 
Yang-Baxter map if it satisfies the Yang-Baxter relation:
\begin{equation}
  \mathcal{R}_{12}\mathcal{R}_{13}\mathcal{R}_{23}
  =\mathcal{R}_{23}\mathcal{R}_{13}\mathcal{R}_{12}.
\end{equation}

First, we examine the relationship between the quantities 
$\pmb{\gamma}_j^\pm$ ($j=1,2$) during a two-soliton collision. 
From equations \eqref{relation-gamma-two-collision-1} and 
\eqref{relation-gamma-two-collision-2}, we obtain
\begin{equation}\label{pair-collision-not-pure-n-soliton}
  \begin{cases}
    \left[\pmb{\Delta}_{\xi_1}^-(\lambda_1)\right]^{-1}\pmb{\Delta}_{\xi_1}^+(\lambda_1)\pmb{\gamma}^+_1=
    \frac{\Delta_{\xi_1}^-(\lambda_1)}{\Delta_{\xi_1}^+(\lambda_1)}
    \frac{\lambda_1-\lambda_2^*}
    {\lambda_1-\lambda_2}\left(\mathbb{I}_2+\frac{\lambda_2-\lambda_2^*}
    {\lambda_1-\lambda_2}\frac{\tilde{\pmb{\gamma}}^-_2\tilde{\pmb{\gamma}}^{-\dagger}_2}
    {\tilde{\pmb{\gamma}}^{-\dagger}_2\tilde{\pmb{\gamma}}^-_2}\right)\pmb{\gamma}^-_1,\\
    \pmb{\gamma}^+_2=\frac{\Delta_{\xi_2}^-(\lambda_2)}{\Delta_{\xi_2}^+(\lambda_2)}
    \frac{\lambda_2-\lambda_1}
    {\lambda_2-\lambda_1^*}\left(\mathbb{I}_2+\frac{\lambda_1^*-\lambda_1}
    {\lambda_2-\lambda_1^*}\frac{\tilde{\pmb{\gamma}}^+_1\tilde{\pmb{\gamma}}^{+\dagger}_1}
    {\tilde{\pmb{\gamma}}^{+\dagger}_1\tilde{\pmb{\gamma}}^+_1}\right)
    \left[\pmb{\Delta}_{\xi_2}^+(\lambda_2)\right]^{-1}\pmb{\Delta}_{\xi_2}^-(\lambda_2)\pmb{\gamma}^-_2,
  \end{cases}
\end{equation}
where 
\begin{equation}
  \tilde{\pmb{\gamma}}^-_2=\left[\pmb{\Delta}_{\xi_1}^-(\lambda_2)\right]^{-1}\pmb{\Delta}_{\xi_2}^-(\lambda_2)
    \pmb{\gamma}^-_2,\,\ \tilde{\pmb{\gamma}}^+_1=\left[\pmb{\Delta}_{\xi_2}^+(\lambda_1)\right]^{-1}
    \pmb{\Delta}_{\xi_1}^+(\lambda_1)\pmb{\gamma}^+_1.
\end{equation}
If the corresponding reflection coefficient vanishes (i.e., $\mathbf{R}(\lambda)=0$), 
then $\M_\pm(\lambda;x,t)=\mathbb{I}_3$ serves as a trivial regular solution 
to RHP \ref{initial-rhp-without-role}. Consequently, equations 
\eqref{relation-gamma-two-collision-1} and \eqref{relation-gamma-two-collision-2} 
reduce to
\begin{equation}
  \begin{cases}
    \pmb{\gamma}^+_1= 
        \frac{
        \mathbf{c}_1}{\left(\lambda_1-\lambda_1^*\right)},&
      \pmb{\gamma}^-_1=
      \left[\mathbb{I}_2-\frac{\lambda_2-\lambda_2^*}{\lambda_1-\lambda_2^*}
      \frac{\pmb{\beta}_2\pmb{\beta}_2^\dagger}{\pmb{\beta}_2^\dagger\pmb{\beta}_2}\right]
      \frac{\mathbf{c}_1}{\lambda_1-\lambda_1^*}
      \frac{\lambda_1-\lambda_2}{\left(\lambda_1-\lambda_2^*\right)},\\
    \pmb{\gamma}^+_2= \left[\mathbb{I}_2-\frac{\lambda_1-\lambda_1^*}{\lambda_2-\lambda_1^*}
      \frac{\pmb{\beta}_1\pmb{\beta}_1^\dagger}{\pmb{\beta}_1^\dagger\pmb{\beta}_1}\right]
      \frac{\mathbf{c}_2}
      {\lambda_2-\lambda_2^*}\frac{\lambda_2-\lambda_1}
      {\left(\lambda_2-\lambda_1^*\right)},& \pmb{\gamma}^-_2=\frac{\mathbf{c}_2}
        {\left(\lambda_2-\lambda_2^*\right)},
  \end{cases}
\end{equation}
where 
\begin{equation}
  \pmb{\beta}_1=\frac{\lambda_1-\lambda_2}{\lambda_1-\lambda_2^*}
    \frac{\mathbf{c}_1}{\lambda_1-\lambda_1^*},\,\
  \pmb{\beta}_2=\frac{\lambda_2-\lambda_1}{\lambda_2-\lambda_1^*}
    \frac{\mathbf{c}_2}{\lambda_2-\lambda_2^*}.
\end{equation}

Furthermore, we can deduce the following relationships between 
$\pmb{\gamma}_j^\pm$ ($j=1,2$) for the pure two-soliton solution:
\begin{equation}\label{relation-gamma-1-2}
  \begin{cases}
    \pmb{\gamma}^+_1=\frac{\lambda_1-\lambda_2^*}
    {\lambda_1-\lambda_2}\left(\mathbb{I}_2+\frac{\lambda_2-\lambda_2^*}
    {\lambda_1-\lambda_2}\frac{\pmb{\gamma}^-_2\pmb{\gamma}^{-\dagger}_2}
    {\pmb{\gamma}^{-\dagger}_2\pmb{\gamma}^-_2}\right)\pmb{\gamma}^-_1,\\
    \pmb{\gamma}^+_2=\frac{\lambda_2-\lambda_1}
    {\lambda_2-\lambda_1^*}\left(\mathbb{I}_2+\frac{\lambda_1^*-\lambda_1}
    {\lambda_2-\lambda_1^*}\frac{\pmb{\gamma}^+_1\pmb{\gamma}^{+\dagger}_1}
    {\pmb{\gamma}^{+\dagger}_1\pmb{\gamma}^+_1}\right)\pmb{\gamma}^-_2.
  \end{cases}
\end{equation}
From equation \eqref{relation-gamma-1-2}, it directly follows that \cite{Soliton_interactions}
\begin{equation}\label{pairwise-collision-two-soliton}
  \begin{cases}
     \pmb{v}_1^+=\frac{1}{\chi}\frac{\lambda_1^*-\lambda_2}
    {\lambda_1^*-\lambda_2^*}\pmb{v}_1^-\left(\mathbb{I}_2+\frac{\lambda_2^*-\lambda_2}
    {\lambda_1^*-\lambda_2^*}\pmb{v}_2^{-\dagger}\pmb{v}_2^-\right),\\
    \pmb{v}_2^+=\frac{1}{\chi}\frac{\lambda_2-\lambda_1^*}
    {\lambda_2-\lambda_1}\pmb{v}_2^-\left(\mathbb{I}_2+\frac{\lambda_1^*-\lambda_1}
    {\lambda_1-\lambda_2}\pmb{v}_1^{-\dagger}\pmb{v}_1^-\right),
  \end{cases}
\end{equation}
where the scaling factor $\chi$ satisfies
\begin{equation}
  \chi^2=\frac{\left|\pmb{\gamma}^+_1\right|^2}{\left|\pmb{\gamma}^-_1\right|^2}
  =\frac{\left|\pmb{\gamma}^-_2\right|^2}{\left|\pmb{\gamma}^+_2\right|^2}
  =\left|\frac{\lambda_1-\lambda_2^*}
    {\lambda_1-\lambda_2}\right|^2
    \left[1+\frac{(\lambda_1-\lambda_1^*)(\lambda_2^*-\lambda_2)}
  {\left|\lambda_1-\lambda_2\right|^2}\left|\pmb{v}_1^-\pmb{v}_2^{-\dagger}\right|^2\right].
\end{equation}

To derive the Yang-Baxter map in full generality for arbitrary polarization 
vectors within an $n$-soliton solution, we introduce the following notation:
\begin{equation}\label{def-gamma-i-rho}
  \pmb{\gamma}_{j,\{i_\rho\}}=
  \prod_{p\in\{1:n\}\setminus\{j,i_\rho\}}
  \frac{\lambda_j-\lambda_p^*}{\lambda_j-\lambda_p}\pmb{D}_{i_\rho}(\lambda_j)
   \pmb{\beta}_{j},\,\  \pmb{v}_{j,\{i_\rho\}}=\frac{\pmb{\gamma}_{j,\{i_\rho\}}^\dagger}
   {\left|\pmb{\gamma}_{j,\{i_\rho\}}\right|},
\end{equation}
where $i_\rho=i_1\cdots i_q,\,\ q\in\{1\cdots n\}$, and $j,i_1,\cdots, i_q$ are 
mutually distinct indices drawn from the set $\{1,\cdots,n\}$. In particular, 
the asymptotic polarization vectors reduce to 
$\pmb{v}_k^+=\pmb{v}_{k,\{1:k-1\}}$ and $\pmb{v}_k^-=\pmb{v}_{k,\{k+1:n\}}$, 
which are pictorially illustrated in Figure \ref{n-soliton-collision}.
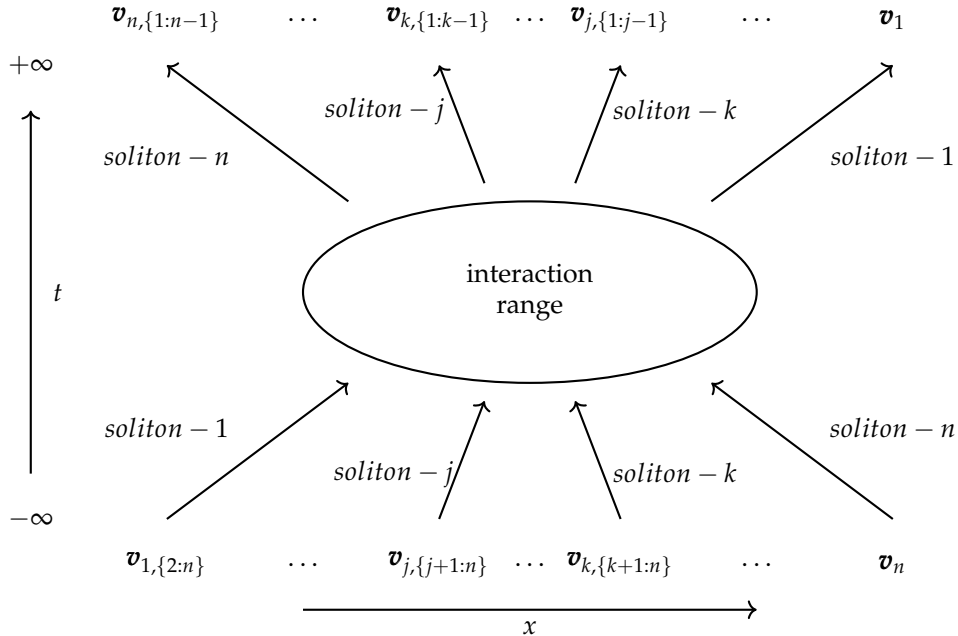
\begin{figure}[htbp]
    \centering
    \begin{tikzpicture}[scale=1.2, every node/.style={scale=1}]
        \draw[thick] (0,0) ellipse (2.5 and 1); 
        \node at (0,0) [align=center] {interaction \\ range};

        \node at (-4, 3) {$\pmb{v}_{n,\{1:n-1\}}$};
        \node at (-4, -3) {$\pmb{v}_{1,\{2:n\}}$};
        \node at (4, 3) {$\pmb{v}_{1}$};
        \node at (4, -3) {$\pmb{v}_{n}$};

        \node at (-2.5, 3) {$\ldots$};
        \node at (-2.5, -3) {$\ldots$};
        \node at (2.5, 3) {$\ldots$};
        \node at (2.5, -3) {$\ldots$};
        \node at (0, 3) {$\ldots$};
        \node at (0, -3) {$\ldots$};

        \node at (1, 3) {$\pmb{v}_{j,\{1:j-1\}}$};
        \node at (1, -3) {$\pmb{v}_{k,\{k+1:n\}}$};

        \node at (-1, 3) {$\pmb{v}_{k,\{1:k-1\}}$};
        \node at (-1, -3) {$\pmb{v}_{j,\{j+1:n\}}$};

        \draw[->, thick] (-4, -2.5) --  (-2, -1);
        \draw[->, thick] (2, 1) -- (4, 2.5);
        \draw[->, thick] (4, -2.5) -- (2, -1);
        \draw[->, thick] (-2, 1) --  (-4, 2.5);
        \draw[->, thick] (-1, -2.5) --  (-0.5, -1.2);
        \draw[->, thick] (0.5, 1.2) --  (1, 2.5);
        \draw[->, thick] (1, -2.5) --  (0.5, -1.2);
        \draw[->, thick] (-0.5, 1.2) --  (-1, 2.5);
        
        \node at (-4, -1.5) {$soliton-1$};
        \node at (-4, 1.5) {$soliton-n$};
        \node at (-1.5, -2) {$soliton-j$};
        \node at (1.6, -2) {$soliton-k$};
        \node at (4, -1.5) {$soliton-n$};
        \node at (1.6, 2) {$soliton-k$};
        \node at (-1.6, 2) {$soliton-j$};
        \node at (4, 1.5) {$soliton-1$};

        \draw[->, thick] (-5.5, -2) -- (-5.5, 2);
        \node at (-5.5, -2.5) {$-\infty$};
        \node at (-5.5, 2.5) {$+\infty$};
        \node at (-5.2,0) {$t$};
        \draw[->, thick] (-2.5, -3.5) -- (2.5, -3.5);
        \node at (0,-3.7) {$x$};
    \end{tikzpicture}
    \caption{The exchange of polarizations in the $n$-soliton collision}
    \label{n-soliton-collision}
\end{figure}

We now present the following lemma, which explicitly characterizes the relationship between the polarization vectors of soliton-$j$ and soliton-$k$ before and after soliton-$j$ overtakes soliton-$k$.

\begin{lemma}\label{pairwise-collision-within-n-soliton}
  Let the spectral parameters $\lambda_j$ and $\lambda_k$ satisfy $\xi_j<\xi_k$. Suppose $i_\rho=i_1\cdots i_q$ for $q\in\{1\cdots n\}$, and assume that $j, k, i_1, \cdots, i_q$ are mutually distinct indices in $\{1,\cdots,n\}$. Then the polarization vectors transform according to:
\begin{equation}\label{relation-Pj-pm}
     \pmb{v}_{j,\{i_\rho\}}=\frac{1}{\chi_{jk}}\frac{\lambda_j^*-\lambda_k}
    {\lambda_j^*-\lambda_k^*}\pmb{v}_{j,\{k,i_\rho\}}\left(\mathbb{I}_2+
    \frac{\lambda_k^*-\lambda_k}
    {\lambda_j^*-\lambda_k^*}\pmb{v}_{k,\{i_\rho\}}^{\dagger}\pmb{v}_{k,\{i_\rho\}}\right),
\end{equation}
\begin{equation}\label{relation-Pk-pm}
  \pmb{v}_{k,\{j,i_\rho\}}=\frac{1}{\chi_{kj}}\frac{\lambda_k-\lambda_j^*}
    {\lambda_k-\lambda_j}\pmb{v}_{k,\{i_\rho\}}\left(\mathbb{I}_2+
    \frac{\lambda_j^*-\lambda_j}
    {\lambda_j-\lambda_k}\pmb{v}_{j,\{k,i_\rho\}}^{\dagger}\pmb{v}_{j,\{k,i_\rho\}}\right),
\end{equation}
where the scaling factor is given by:
\begin{equation}
  \chi_{jk}^2=\frac{\left|\pmb{\gamma}_{j,\{i_\rho\}}\right|^2}
  {\left|\pmb{\gamma}_{j,\{k,i_\rho\}}\right|^2}=\chi_{kj}^2
  =\frac{\left|\pmb{\gamma}_{k,\{i_\rho\}}\right|^2}{\left|\pmb{\gamma}_{k,\{j,i_\rho\}}\right|^2}
  =\left|\frac{\lambda_j-\lambda_k^*}
    {\lambda_j-\lambda_k}\right|^2
    \left[1+\frac{(\lambda_j-\lambda_j^*)(\lambda_k^*-\lambda_k)}
  {\left|\lambda_j-\lambda_k\right|^2}\left|\pmb{v}_{j,\{k,i_\rho\}}\pmb{v}_{k,\{i_\rho\}}^{\dagger}\right|^2\right].
\end{equation}
\end{lemma}

\begin{proof}
  From the definition in equation \eqref{def-gamma-i-rho}, we have:
  \begin{equation}
    \begin{aligned}
      &\pmb{\gamma}_{j,\{i_\rho\}}=
      \prod_{p\in\{1:n\}\setminus\{j,i_\rho\}}
      \frac{\lambda_j-\lambda_p^*}{\lambda_j-\lambda_p}\pmb{D}_{i_\rho}(\lambda_j)
   \pmb{\beta}_{j},
   &\pmb{\gamma}_{j,\{k,i_\rho\}}=
   \prod_{p\in\{1:n\}\setminus\{j,k,i_\rho\}}
    \frac{\lambda_j-\lambda_p^*}{\lambda_j-\lambda_p}\pmb{D}_{i_\rho k}(\lambda_j)
   \pmb{\beta}_{j},\\
   &\pmb{\gamma}_{k,\{j,i_\rho\}}=
      \prod_{p\in\{1:n\}\setminus\{j,k,i_\rho\}}
       \frac{\lambda_j-\lambda_p^*}{\lambda_j-\lambda_p}\pmb{D}_{i_\rho j}(\lambda_k)
   \pmb{\beta}_{k},
   &\pmb{\gamma}_{k,\{i_\rho\}}=
   \prod_{p\in\{1:n\}\setminus\{k,i_\rho\}}
    \frac{\lambda_j-\lambda_p^*}{\lambda_j-\lambda_p}\pmb{D}_{i_\rho}(\lambda_k)
   \pmb{\beta}_{k}.\\
    \end{aligned}
  \end{equation}
  Applying equation \eqref{def-D}, we can decompose the dressing factors as follows:
  \begin{equation}
    \pmb{D}_{i_\rho k}(\lambda)=\pmb{D}_{k,\{i_\rho\}}(\lambda)\pmb{D}_{i_\rho}(\lambda),\,\
  \pmb{D}_{k,\{i_\rho\}}(\lambda)=\mathbb{I}_2-
    \frac{\lambda_k-\lambda_k^*}{\lambda-\lambda_k^*}\frac{\pmb{\varphi}_k\pmb{\varphi}_k^\dagger}
    {\pmb{\varphi}_k^\dagger\pmb{\varphi}_k},\,\
    \pmb{\varphi}_k=\pmb{D}_{i_\rho}(\lambda_k) \pmb{\beta}_{k},
\end{equation}
\begin{equation}
  \pmb{D}_{i_\rho j}(\lambda)=\pmb{D}_{j,\{i_\rho\}}(\lambda)\pmb{D}_{i_\rho}(\lambda),\,\
  \pmb{D}_{j,\{i_\rho\}}(\lambda)=\mathbb{I}_2-
    \frac{\lambda_j-\lambda_j^*}{\lambda-\lambda_j^*}\frac{\pmb{\varphi}_j\pmb{\varphi}_j^\dagger}
    {\pmb{\varphi}_j^\dagger\pmb{\varphi}_j},\,\
    \pmb{\varphi}_j=\pmb{D}_{i_\rho}(\lambda_j) \pmb{\beta}_{j}.
\end{equation}
Utilizing the symmetry property $\pmb{D}^{-1}(\lambda)=\pmb{D}^\dagger(\lambda^*)$, we obtain:
  \begin{equation}\label{relation-rho-j-pm}
    \pmb{\gamma}_{j,\{i_\rho\}}=\frac{\lambda_j-\lambda_k^*}
    {\lambda_j-\lambda_k}\left(\mathbb{I}_2+\frac{\lambda_k-\lambda_k^*}
    {\lambda_j-\lambda_k}\frac{\pmb{\gamma}_{k,\{i_\rho\}}\pmb{\gamma}^{\dagger}_{k,\{i_\rho\}}}
    {\pmb{\gamma}^{\dagger}_{k,\{i_\rho\}}\pmb{\gamma}_{k,\{i_\rho\}}}\right)\pmb{\gamma}_{j,\{k,i_\rho\}},
\end{equation}
and
\begin{equation}\label{relation-rho-k-pm}
  \pmb{\gamma}_{k,\{j,i_\rho\}}=\frac{\lambda_k-\lambda_j}
    {\lambda_k-\lambda_j^*}\left(\mathbb{I}_2+\frac{\lambda_j^*-\lambda_j}
    {\lambda_k-\lambda_j^*}\frac{\pmb{\gamma}_{j,\{i_\rho\}}\pmb{\gamma}^{\dagger}_{j,\{i_\rho\}}}
    {\pmb{\gamma}^{\dagger}_{j,\{i_\rho\}}\pmb{\gamma}_{j,\{i_\rho\}}}\right)\pmb{\gamma}_{k,\{i_\rho\}}.
\end{equation}
It directly follows from equation \eqref{relation-rho-j-pm} that:
\begin{equation}
  \begin{aligned}
    \chi_{jk}^2&=\frac{\left|\pmb{\gamma}_{j,\{i_\rho\}}\right|^2}
    {\left|\pmb{\gamma}_{j,\{k,i_\rho\}}\right|^2}\\
    &=\left|\frac{\lambda_j-\lambda_k^*}
      {\lambda_j-\lambda_k}\right|^2\frac{\pmb{\gamma}_{j,\{k,i_\rho\}}^\dagger}{\left|\pmb{\gamma}_{j,\{k,i_\rho\}}\right|}
      \left(\mathbb{I}_2+\frac{\lambda_k^*-\lambda_k}
      {\lambda_j^*-\lambda_k^*}\frac{\pmb{\gamma}_{k,\{i_\rho\}}\pmb{\gamma}^{\dagger}_{k,\{i_\rho\}}}
      {\pmb{\gamma}^{\dagger}_{k,\{i_\rho\}}\pmb{\gamma}_{k,\{i_\rho\}}}\right)
      \left(\mathbb{I}_2+\frac{\lambda_k-\lambda_k^*}
      {\lambda_j-\lambda_k}\frac{\pmb{\gamma}_{k,\{i_\rho\}}\pmb{\gamma}^{\dagger}_{k,\{i_\rho\}}}
      {\pmb{\gamma}^{\dagger}_{k,\{i_\rho\}}\pmb{\gamma}_{k,\{i_\rho\}}}\right)
      \frac{\pmb{\gamma}_{j,\{k,i_\rho\}}}{\left|\pmb{\gamma}_{j,\{k,i_\rho\}}\right|}\\
      &=\left|\frac{\lambda_j-\lambda_k^*}
      {\lambda_j-\lambda_k}\right|^2\pmb{v}_{j,\{k,i_\rho\}}
      \left(\mathbb{I}_2+\frac{\lambda_k^*-\lambda_k}
      {\lambda_j^*-\lambda_k^*}\pmb{v}_{k,\{i_\rho\}}^{\dagger}\pmb{v}_{k,\{i_\rho\}}\right)
      \left(\mathbb{I}_2+\frac{\lambda_k-\lambda_k^*}
      {\lambda_j-\lambda_k}\pmb{v}_{k,\{i_\rho\}}^{\dagger}\pmb{v}_{k,\{i_\rho\}}\right)
      \pmb{v}_{j,\{k,i_\rho\}}^\dagger\\
      &=\left|\frac{\lambda_j-\lambda_k^*}
    {\lambda_j-\lambda_k}\right|^2
    \left[1+\frac{(\lambda_j-\lambda_j^*)(\lambda_k^*-\lambda_k)}
  {\left|\lambda_j-\lambda_k\right|^2}\left|\pmb{v}_{j,\{k,i_\rho\}}\pmb{v}_{k,\{i_\rho\}}^{\dagger}\right|^2\right].
  \end{aligned}
\end{equation}
This leads directly to equation \eqref{relation-Pj-pm} via equation \eqref{relation-rho-j-pm}. A straightforward calculation, achieved by inserting equation \eqref{relation-Pj-pm} into equation \eqref{relation-rho-k-pm}, confirms that $\chi_{jk}^2=\chi_{kj}^2$ and subsequently yields equation \eqref{relation-Pk-pm}. This completes the proof.
\end{proof}
The relations established in Lemma \ref{pairwise-collision-within-n-soliton} 
offer a natural physical interpretation: they characterize an `intermediate-time' 
pairwise collision between soliton-$j$ and soliton-$k$. 
Given the velocity condition $v_j>v_k$ (where $\xi_j<\xi_k$), 
soliton-$j$ will overtake soliton-$k$ after having already interacted with 
a specific subset of other solitons indexed by $i_\rho$. 
During this interaction, their polarization states are dynamically updated: 
soliton-$j$ transitions from $\pmb{v}_{j,\{k,i_\rho\}}$ to $\pmb{v}_{j,\{i_\rho\}}$, 
and soliton-$k$ transitions from $\pmb{v}_{k,\{i_\rho\}}$ to $\pmb{v}_{k,\{j,i_\rho\}}$. 
These collision dynamics are graphically represented in Figure \ref{immediate-time-pair-collision}.
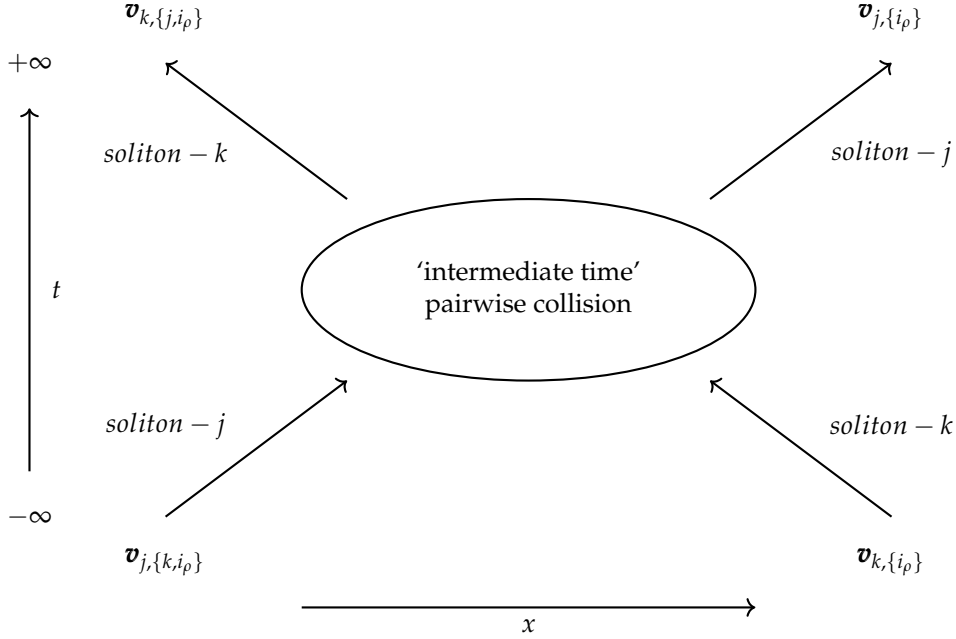
\begin{figure}[htbp]
    \centering
    \begin{tikzpicture}[scale=1.2, every node/.style={scale=1}]
        \draw[thick] (0,0) ellipse (2.5 and 1); 
        \node at (0,0) [align=center] {`intermediate time' \\ 
        pairwise collision};

        \node at (-4, 3) {$\pmb{v}_{k,\{j,i_\rho\}}$};
        \node at (-4, -3) {$\pmb{v}_{j,\{k,i_\rho\}}$};
        \node at (4, 3) {$\pmb{v}_{j,\{i_\rho\}}$};
        \node at (4, -3) {$\pmb{v}_{k,\{i_\rho\}}$};

       \node at (-4, -1.5) {$soliton-j$};
        \node at (-4, 1.5) {$soliton-k$};
        \node at (4, -1.5) {$soliton-k$};
        \node at (4, 1.5) {$soliton-j$};

        \draw[->, thick] (-4, -2.5) --  (-2, -1);
        \draw[->, thick] (2, 1) -- (4, 2.5);
        \draw[->, thick] (4, -2.5) -- (2, -1);
        \draw[->, thick] (-2, 1) --  (-4, 2.5);

        \draw[->, thick] (-5.5, -2) -- (-5.5, 2);
        \node at (-5.5, -2.5) {$-\infty$};
        \node at (-5.5, 2.5) {$+\infty$};
        \node at (-5.2,0) {$t$};
        \draw[->, thick] (-2.5, -3.5) -- (2.5, -3.5);
        \node at (0,-3.7) {$x$};
    \end{tikzpicture}
    \caption{The `intermediate time' pairwise collision 
     between soliton-$j$ and soliton-$k$.}
    \label{immediate-time-pair-collision}
\end{figure}

Building upon equation \eqref{pairwise-collision-two-soliton} and Lemma \ref{pairwise-collision-within-n-soliton}, we introduce a transition map $\mathcal{R}(\lambda_1,\lambda_2)$ to explicitly describe the transformation of the polarization vectors of two solitons following a pairwise collision:
\begin{equation}\label{pairwise-collision-map}
  \begin{aligned}
    &\mathcal{R}_{12}(\lambda_1,\lambda_2):\left(\pmb{v}_1^-,\pmb{v}_2^-\right)\to
    \left(\pmb{v}_1^+,\pmb{v}_2^+\right)\\
    &\begin{cases}
       \pmb{v}_1^+=\frac{1}{\chi}\frac{\lambda_1^*-\lambda_2}
      {\lambda_1^*-\lambda_2^*}\pmb{v}_1^-\left(\mathbb{I}_2+\frac{\lambda_2^*-\lambda_2}
      {\lambda_1^*-\lambda_2^*}\pmb{v}_2^{-\dagger}\pmb{v}_2^-\right),\\
      \pmb{v}_2^+=\frac{1}{\chi}\frac{\lambda_2-\lambda_1^*}
      {\lambda_2-\lambda_1}\pmb{v}_2^-\left(\mathbb{I}_2+\frac{\lambda_1^*-\lambda_1}
      {\lambda_1-\lambda_2}\pmb{v}_1^{-\dagger}
      \pmb{v}_1^-\right),
    \end{cases}
  \end{aligned}
\end{equation}
where $\chi=\left|\frac{\lambda_1-\lambda_2^*}
    {\lambda_1-\lambda_2}\right|^2
    \left[1+\frac{(\lambda_1-\lambda_1^*)(\lambda_2^*-\lambda_2)}
  {\left|\lambda_1-\lambda_2\right|^2}\left|\pmb{v}_1^-\pmb{v}_2^{-\dagger}\right|^2\right]$.
It has been established that the map $\mathcal{R}_{12}(\lambda_1,\lambda_2)$ constitutes a reversible, parameter-dependent Yang-Baxter map \cite{Yang_Baxter_maps,Soliton_interactions}. Specifically, $\mathcal{R}_{12}(\lambda_1,\lambda_2)$ satisfies the Yang-Baxter equation:
\begin{equation}
  \mathcal{R}_{12}(\lambda_1,\lambda_2)\mathcal{R}_{13}(\lambda_1,\lambda_3)
  \mathcal{R}_{23}(\lambda_2,\lambda_3)=\mathcal{R}_{23}(\lambda_2,\lambda_3)
  \mathcal{R}_{13}(\lambda_1,\lambda_3)\mathcal{R}_{12}(\lambda_1,\lambda_2),
\end{equation}
along with the reversibility condition:
\begin{equation}
  \mathcal{R}_{21}(\lambda_2,\lambda_1)\mathcal{R}_{12}(\lambda_1,\lambda_2)=\mathcal{I}d.
\end{equation}
Furthermore, combining the results from Proposition \ref{nls-n-soliton-iteration-form} and Lemma \ref{pairwise-collision-within-n-soliton}, we can deduce the following theorem which was initially presented in \cite{Nsoliton_collision}.
\begin{prop}
  (cf. \cite{Nsoliton_collision,Yang-Baxter-reflection}) The pure $n$-soliton collision in the CNLS model \eqref{CNLS} can be decomposed into a nonlinear superposition of $\binom{n}{2}$ two-soliton collisions, which may occur in an arbitrary sequence. 
\end{prop}

\begin{proof}
  Suppose the initial data $\mathbf{q}_0(x)$ 
  corresponds to the reflectionless scattering data 
  \begin{equation*}
    \mathcal{D}_n(\mathbf{q}_0)=
  \left\{\hat{\mathbf{R}}(\lambda)\equiv0,\left\{\lambda_i
  \right\}_{i=1}^n, \left\{\mathbf{c}_i\right\}_{i=1}^n
  \right\},
  \end{equation*}
  and assume that the discrete spectrum $\{\lambda_j=\xi_j+\ii\eta_j\}$ 
  satisfies $\xi_1<\xi_2<\cdots<\xi_n$. Then, 
  based on the asymptotic behavior of the $n$-soliton 
  solution as $t\to-\infty$ (see Figure \ref{n-soliton-collision}), 
  solitons $1$ through $n$ are initially arranged along the $x$-axis 
  in the following spatial order:
  \begin{equation*}
    \text{soliton-}1,\ \ \cdots,
    \ \ \text{soliton-}j,\ \ \cdots,
    \ \ \text{soliton-}k,\ \ \cdots,
    \ \ \text{soliton-}n.
  \end{equation*}
  Taking this initial configuration as the starting point, we assume 
  that the `intermediate-time' pairwise collisions 
  (see Figure \ref{immediate-time-pair-collision}) take place 
  $\binom{n}{2}=\frac{n(n-1)}{2}$ times in a specific sequence. 
  A fundamental observation is that the subset of indices enclosed in $\{\}$ for 
  each polarization vector $\pmb{v}_{j,\{j+1:n\}}$ 
  perfectly matches the complete set of indices for the adjacent vector 
  $\pmb{v}_{j+1,\{j+2:n\}}$ immediately to its right. 
  This structural consistency guarantees that Lemma \ref{pairwise-collision-within-n-soliton} 
  is directly applicable to every single `intermediate-time' pairwise collision. 
  
  Owing to the Yang-Baxter structure, regardless of the precise sequence of these pairwise collisions, 
  the solitons $1,\cdots,n$ are ultimately distributed 
  along the $x$-axis in the completely reversed order:
  \begin{equation*}
    \text{soliton-}n,\ \ \cdots,
    \ \ \text{soliton-}k,\ \ \cdots,
    \ \ \text{soliton-}j,\ \ \cdots,
    \ \ \text{soliton-}1.
  \end{equation*}
  Crucially, this final asymptotic state for each soliton 
  exactly coincides with the long-time asymptotic behavior of the $n$-soliton 
  solution as $t\to+\infty$ (see Figure \ref{n-soliton-collision}). 
  This completes the proof.
\end{proof}
Due to the presence of $\pmb{\Delta}_{\xi_k}^\pm(\lambda_k)$, the $n$-soliton collision (with the related reflection coefficient $\hat{\mathbf{R}}(\lambda)\not\equiv0$) seems to preclude a decomposition into $\binom{n}{2}$ pairwise collisions as described in equation \eqref{pair-collision-not-pure-n-soliton}.

\section*{Appendix A}
\setcounter{equation}{0} 
\renewcommand{\theequation}{A.\arabic{equation}}
\setcounter{rhp}{0} 
\renewcommand{\therhp}{A.\arabic{rhp}}
The formulation of the Darboux matrix $\mathbf{T}^{(\pm)}$ can be found in \cite{HYBLLMZXE2025}. 
For the sake of completeness, we provide its explicit construction here. For $1\le i\le N$,
\begin{equation}\label{def-Ti}
  \begin{aligned}
    &\mathbf{T}_i^{(+)}(\lambda ; x, t)=\mathbb{I}_3-\frac{\lambda_{i}-\lambda_{i}^*}
    {\lambda-\lambda_{i}^*} 
    \mathbf{P}_{i}(x, t), \quad \mathbf{P}_{i}(x, t)=\frac{\mathbf{\Phi}_{i} 
    \mathbf{\Phi}_{i}^{\dagger}}{\mathbf{\Phi}_{i}^{\dagger} \mathbf{\Phi}_{i}},\\
    &\mathbf{T}_i^{(-)}(\lambda ; x, t)=
    \frac{\lambda-\lambda_{i}^*}
    {\lambda-\lambda_{i}} \mathbf{T}_i^{(+)}(\lambda ; x, t),\quad
    \mathbf{\Phi}_{i}=\hat{\M}_+^{[i]}(\lambda_{i} ; x, t) \ee^{-\mathrm{i} 
  \lambda_{i}\left(x+\lambda_{i} t\right) \pmb{\Lambda}_{3}} 
  \mathbf{v}_{i},\\
  &\mathbf{v}_{i}=\begin{bmatrix}
    1,b_{i1},b_{i2}
  \end{bmatrix}^\dagger, \quad
  \begin{bmatrix}
    b_{i1},b_{i2}
  \end{bmatrix}^\dagger=\frac{\mathbf{c}_i}{\lambda_i-\lambda_i^*}
  \prod_{j=1}^{i-1}\frac{\lambda_i-\lambda_j}{\lambda_i-\lambda_j^*},
  \end{aligned}
\end{equation}
where the vector $\mathbf{c}_i$ is defined in equation \eqref{residue-condition}, 
$\hat{\M}^{[N]}=\M$ serves as the solution to RHP \ref{initial-rhp-without-role}, and 
$\hat{\M}^{[i]}$ (for $1\le i\le N-1$) represents the solution to the following RHP:

\begin{rhp}
  Find a matrix function $\hat{\M}^{[i]}(\lambda)=\hat{\M}^{[i]}(\lambda;x,t)$ with the following properties: 
  \begin{enumerate}
    \item Analyticity: $\hat{\M}^{[i]}(\lambda)$ is meromorphic in $\mathbb{C}\setminus\mathbb{R}$ and has simple 
    poles at $\lambda_k\in\mathcal{Z}_i$ and $\lambda_k^*\in\mathcal{Z}_i^*$,
    where $\mathcal{Z}_i=\mathcal{Z}\setminus\{\lambda_j\}_{j=1}^i$.
    \item Jump condition: $\hat{\M}^{[i]}(\lambda)$ has continuous boundary 
    values $\hat{\M}^{[i]}_\pm(\lambda)$ on $\mathbb{R}$
    \begin{equation}
      \hat{\M}^{[i]}_\pm(\lambda)=\lim_{\varepsilon \downarrow 0} \hat{\M}^{[i]}(\lambda\pm\ii\varepsilon), 
    \end{equation}
    satisfying 
    $\hat{\M}^{[i]}_+(\lambda)=\hat{\M}^{[i]}_-(\lambda)\hat{\V}^{[i]}(\lambda)$, where
    \begin{equation}
      \hat{\V}^{[i]}(\lambda)=
      \begin{pmatrix}
        1+\hat{\mathbf{R}}_i(\lambda)\hat{\mathbf{R}}_i^\dagger(\lambda)&-\hat{\mathbf{R}}_i(\lambda)\ee^{-2\ii t\theta(\lambda)}\\
        -\hat{\mathbf{R}}_i^\dagger(\lambda)\ee^{2\ii t\theta(\lambda)}&\mathbb{I}_2
      \end{pmatrix}
     \end{equation}
     and $\mathbf{R}_i(\lambda)=\prod_{k=1}^{i}\frac{\lambda-\lambda_k^*}{\lambda-\lambda_k}\hat{\mathbf{R}}(\lambda)$.
    \item Residue conditions: $\hat{\M}^{[i]}(\lambda)$ has simple poles at $\lambda_j\in\mathcal{Z}_i$ 
    and $\lambda_j^*\in\mathcal{Z}_i^*$  with 
    \begin{equation*}
      \begin{aligned}
        \operatorname*{Res}_{\lambda=\lambda_j}\hat{\M}^{[i]} &= 
        \lim_{\lambda\to\lambda_j}\hat{\M}^{[i]}
        \begin{pmatrix}
          0&\mathbf{0}\\
          -\mathbf{c}_{ij}\ee^{2\ii t\theta}&\mathbf{0}
        \end{pmatrix},\\
        \operatorname*{Res}_{\lambda=\lambda_j^*}\hat{\M}^{[i]} &= 
        \lim_{\lambda\to\lambda_j^*}\hat{\M}^{[i]}
        \begin{pmatrix}
          0&\mathbf{c}_{ij}^{\dagger}\ee^{-2\ii t\theta}\\
          \mathbf{0}&\mathbf{0}
        \end{pmatrix},
      \end{aligned}
    \end{equation*}
    where $\mathbf{c}_{ij}=\mathbf{c}_j\prod_{k=1}^{i}\frac{\lambda_j-\lambda_k}{\lambda_j-\lambda_k^*}$.
    \item Asymptotic condition: $\hat{\M}^{[i]}(\lambda)=\mathbb{I}_{3}+\mathcal{O}(\lambda^{-1})$ as $\lambda\rightarrow\infty$.
  \end{enumerate} 
\end{rhp}

Then we can construct the global Darboux matrix $\mathbf{T}^{(\pm)}$ 
from the product of $\mathbf{T}_i^{(\pm)}$,
\begin{equation}\label{construction-T}
  \mathbf{T}^{(\pm)}=\prod_{i=1}^{N}\mathbf{T}^{(\pm)}_i.
\end{equation}
Next, we provide detailed proofs for Lemma \ref{existence-S} and 
Lemma \ref{bdd-residue-M4}.

\begin{proof}[Proof of Lemma \ref{existence-S}]
  Our argument essentially follows the methodology presented in \cite{McLaughlin_2018}.
  We detail the estimates for matrix functions supported in the region $\Omega_1$; the procedure for the 
  remaining regions follows an analogous pattern. Let $f\in L^\infty(\Omega_1)$, $s=u+\ii v\in\Omega_1$, and 
  $\lambda=\alpha+\ii\beta$. 
  From the definitions in equations \eqref{def-S} and \eqref{def-V4}, it follows that 
  \begin{equation*}
    \left|S[f]\right|\le \|f\|_{L^\infty(\Omega_1)}\|\M^{(3)}\|_{L^\infty(\Omega_1)}
    \|\M^{(3)}\,^{-1}\|_{L^\infty(\Omega_1)}\underset{\Omega_1}{\iint}
    \frac{\left|\bar{\partial}R_1(s)\right|\ee^{-4tv(u-\xi)}}{\left|s-\lambda\right|}\dd A(s).
  \end{equation*}
  Consequently, utilizing the estimates from Lemma \ref{def-extension}, we can decompose the integral as:
  \begin{equation*}
    \begin{aligned}
    \|S\|_{L^\infty\to L^\infty}&\lesssim \underset{\Omega_1}{\iint}
    \frac{\left|\mathbf{R}^{\prime}(s)\right|\ee^{-4tv(u-\xi)}}{\left|s-\lambda\right|}\dd A(s)+
    \underset{\Omega_1}{\iint}
    \frac{\left|s-\xi\right|^{-1/2}\ee^{-4tv(u-\xi)}}{\left|s-\lambda\right|}\dd A(s)\\
    &\quad+\underset{\Omega_1}{\iint}
    \frac{\left|\bar{\partial}\chi(s)\right|\ee^{-4tv(u-\xi)}}{\left|s-\lambda\right|}\dd A(s)
    \equiv \mathrm{I}+\mathrm{II}+\mathrm{III}.
    \end{aligned}
  \end{equation*}
  To bound the term $\mathrm{I}$, we first apply the basic inequality:
  \begin{equation*}
    \left\|\frac{1}{s-\lambda}\right\|_{L_u^2(v+\xi,\infty)}^2
    =\int_{v+\xi}^\infty\frac{1}{(u-\alpha)^2+(v-\beta)^2}\dd u\leq
    \int_{\mathbb{R}}\frac{1}{u^2+(v-\beta)^2}\dd u=\frac{\pi}{|v-\beta|}.
  \end{equation*}
  This allows us to establish the bound for $\mathrm{I}$ via the Cauchy-Schwarz inequality:
  \begin{equation*}
    \begin{aligned}
      |\mathrm{I}|&\leq\int_0^\infty \ee^{-4tv^2}\int_{v+\xi}^\infty
    \frac{|\mathbf{R}^{\prime}(u)|}{|s-\lambda|}\dd u\dd v\leq\|\mathbf{R}^{\prime}(u)\|_{L^2(\mathbb{R})}
    \int_0^\infty \ee^{-4tv^2}\left\|\frac{1}{s-\lambda}\right\|_{L_u^2(v+\xi,\infty)}\dd v\\
    &\lesssim \int_0^\infty \frac{\ee^{-4tv^2}}{|v-\beta|^{1/2}}\dd v \lesssim 
    t^{-1/4} \int_\mathbb{R} \frac{\ee^{-4(w+\sqrt{t}\beta)^2}}{|w|^{1/2}}\dd w
    \lesssim t^{-1/4} \int_\mathbb{R} \frac{\ee^{-4w^2}}{|w|^{1/2}}\dd w \lesssim t^{-1/4}.
    \end{aligned}
  \end{equation*}
  For the term $\mathrm{II}$, we choose parameters $p > 2$ and $q$ satisfying $1/p+1/q=1$. Applying Hölder's inequality yields:
  \begin{equation*}
    \begin{aligned}
      |\mathrm{II}|&\leq\int_0^\infty \ee^{-4tv^2}\|(s-\xi)^{-1/2}\|_{L_u^p(v+\xi,\infty)}
      \|(s-\lambda)^{-1}\|_{L_u^q(v+\xi,\infty)}
      \dd v\\
      &\lesssim \int_0^\infty \ee^{-4tv^2} v^{1/p-1/2}\left|v-\beta\right|^{1/q-1}\dd v.
    \end{aligned}
  \end{equation*}
  To bound this resulting integral, we split the domain into two parts. For $v \in (0, \beta)$, we have:
  \begin{equation*}
    \begin{aligned}
      \int_0^\beta \ee^{-4tv^2}v^{1/p-1/2}(\beta-v)^{1/q-1}\dd v&=
    \int_0^1\beta^{1/2}\ee^{-4t\beta^2w^2}w^{1/p-1/2}(1-w)^{1/q-1}\dd w\\
    &\lesssim t^{-1/4} \int_0^1 w^{1/p-1}(1-w)^{1/q-1}\dd w\lesssim t^{-1/4},
    \end{aligned}
  \end{equation*}
  which relies on the elementary bound $\sup_{w>0}w^{1/2} \ee^{-w^2}\le C$. Furthermore, utilizing the condition $p>2$, the remaining integral satisfies:
  \begin{equation*}
    \begin{aligned}
      \int_{\beta}^\infty \ee^{-4tv^2}v^{1/p-1/2}(v-\beta)^{1/q-1}\dd v
    &\leq\int_{0}^{\infty}\ee^{-4t(w+\beta)^2}(w+\beta)^{1/p-1/2}w^{1/q-1}\dd w\\
    &\le \int_{0}^{\infty}\ee^{-4tw^2}w^{1/p-1/2}w^{1/q-1}\dd w
    \lesssim t^{-1/4}.
    \end{aligned}
  \end{equation*}
  The analysis for $\mathrm{III}$ proceeds analogously to that of $\mathrm{I}$. Given that the support of 
  $\bar{\partial}\chi(s)$ is strictly confined to small concentric circular neighborhoods around the discrete spectrum, 
  we deduce:
  \begin{equation*}
    |\mathrm{III}|\leq
    \int_0^\infty \ee^{-4tv^2}\left\|\bar{\partial}\chi(s)\right\|_{L_u^2(v+\xi,\infty)}
    \left\|\frac{1}{s-\lambda}\right\|_{L_u^2(v+\xi,\infty)}\dd v
     \lesssim t^{-1/4}.
  \end{equation*}
  Collecting these bounds, we establish that $\|S\|_{L^\infty\to L^\infty} \lesssim t^{-1/4}$. 
\end{proof}

\begin{proof}[Proof of Lemma \ref{bdd-residue-M4}]
  Our proof essentially follows the methodology outlined in \cite{McLaughlin_2018}.
  We detail the estimates for matrix functions supported in the region $\Omega_1$; the analysis for the 
  remaining regions follows an analogous pattern. Let $s=u+\ii v$ and $\lambda=\alpha+\ii\beta$. By combining equation \eqref{residue-M4} and Lemma \ref{def-extension}, we obtain
  \begin{equation*}
    \begin{aligned}
      \left|\M^{(4)}_1\right|&\le 
      \|\M^{(3)}\|_{L^\infty(\Omega_1)}
      \|\M^{(3)}\,^{-1}\|_{L^\infty(\Omega_1)}\underset{\Omega_1}{\iint}
      \left|\bar{\partial}R_1(s)\right|\ee^{-4tv(u-\xi)}\dd A(s)\\
      &\lesssim \underset{\Omega_1}{\iint}
      \left|\mathbf{R}^\prime(s)\right|\ee^{-4tv(u-\xi)}\dd A(s)+
      \underset{\Omega_1}{\iint}
      \left|s-\xi\right|^{-1/2}\ee^{-4tv(u-\xi)}\dd A(s)\\
      &\quad+\underset{\Omega_1}{\iint}
      \left|\bar{\partial}\chi(s)\right|\ee^{-4tv(u-\xi)}\dd A(s)
      \equiv\mathrm{I}+\mathrm{II}+\mathrm{III}.
    \end{aligned}
  \end{equation*}
  To bound term $\mathrm{I}$, we apply the Cauchy-Schwarz inequality:
  \begin{equation*}
    \begin{aligned}
      \mathrm{I}&\lesssim \int_0^\infty \|\mathbf{R}^{\prime}(u)\|_{L^2(\mathbb{R})}
    \left(\int_{v+\xi}^\infty \ee^{-8tv(u-\xi)}\dd u\right)^{1/2}\dd v\\
    &\lesssim t^{-1/2}\int_0^\infty \frac{\ee^{-4tv^2}}{\sqrt{v}}\dd v
    \lesssim t^{-3/4}\int_0^\infty \frac{\ee^{-4w^2}}{\sqrt{w}}\dd w\lesssim t^{-3/4}.
    \end{aligned}
  \end{equation*}
  For term $\mathrm{II}$, we employ H\"older's inequality with exponent $p$ satisfying $2<p<4$ to deduce
  \begin{equation*}
    \begin{aligned}
      \mathrm{II}&\le \int_0^\infty 
      \left(\int_{v+\xi}^\infty \left(\frac{1}{(u-\xi)^2+v^2}\right)^{p/4}
      \dd u\right)^{1/p}
      \left(\int_{v+\xi}^\infty \ee^{-8qtv(u-\xi)}\dd u\right)^{1/q}\dd v\\
      &\lesssim t^{-1/q}  \int_0^\infty v^{1/p-1/2}v^{-1/q}\ee^{-4tv^2}\dd v
      \lesssim t^{-1/q}  \int_0^\infty v^{2/p-3/2}\ee^{-4tv^2}\dd v\\
      &\lesssim t^{-3/4}\int_0^\infty w^{2/p-3/2}\ee^{-4w^2}\dd w\lesssim t^{-3/4},
    \end{aligned}
  \end{equation*}
  where we have utilized the algebraic constraint $-1<\frac{2}{p}-\frac32<-\frac12$.
  Finally, the bound for term $\mathrm{III}$ is established in the same manner as for term $\mathrm{I}$:
  \begin{equation*}
    \mathrm{III}\lesssim \int_0^\infty \|\bar{\partial}\chi(s)\|_{L^2_{L_u^2(v+\xi,\infty)}}
    \left(\int_{v+\xi}^\infty \ee^{-8tv(u-\xi)}\dd u\right)^{1/2}\dd v\lesssim t^{-3/4}.
  \end{equation*}
\end{proof}
\section*{Appendix B}
\setcounter{equation}{0} 
\renewcommand{\theequation}{B.\arabic{equation}}
\setcounter{rhp}{0} 
\renewcommand{\therhp}{B.\arabic{rhp}}

In this Appendix, we prove Theorem \ref{asy-q-n-neg}. The essential steps of the $\bar{\partial}$-steepest descent method for solving RHP 
\ref{initial-rhp-without-role} in the limit $t\to-\infty$ closely mirror those presented in 
Sections 3 and 4 for the case $t\to\infty$. The principal differences arise directly from the 
fact that the specific regions of exponential growth and decay for the phase factors $\ee^{\pm2\ii t\theta}$ are 
reversed when time is considered in the negative direction. Here, we briefly sketch these necessary modifications. 

The first significant change concerns the conjugation factor $\pmb{\Delta}(\lambda)$, which was initially defined in equation \eqref{def-Delta} in Section 3. 
Analogous to RHP \ref{RHP-delta}, we define the modified factor via the following problem:
\begin{rhp}\label{RHP-delta-neg}
Find a matrix function $\pmb{\delta}(\lambda)$ with the following properties: 
\begin{enumerate}
  \item Analyticity: $\pmb{\delta}(\lambda)$ is analytic in 
  $\mathbb{C}\setminus(\xi,+\infty)$.
  \item Jump condition: $\pmb{\delta}(\lambda)$ possesses continuous boundary 
  values $\pmb{\delta}_\pm(\lambda)$ on the ray $(\xi,+\infty)$,
  \begin{equation}
    \pmb{\delta}_\pm(\lambda)=\lim_{\varepsilon \downarrow 0} \pmb{\delta}(\lambda\pm\ii\varepsilon), 
  \end{equation}
  satisfying the jump relation 
  \begin{equation}
    \pmb{\delta}_+(\lambda)=(\mathbb{I}_2+\mathbf{R}^\dagger(\lambda)\mathbf{R}(\lambda))\pmb{\delta}_-(\lambda).
  \end{equation}
  \item Asymptotic condition: 
  $\pmb{\delta}(\lambda)=\mathbb{I}_{2}+\mathcal{O}(\lambda^{-1})$ as $\lambda\rightarrow\infty$.
\end{enumerate} 
\end{rhp}

Consequently, the determinant $\det\pmb{\delta}(\lambda)$ of the matrix function $\pmb{\delta}(\lambda)$ satisfies the 
following scalar RHP:
\begin{equation}\label{RHP-det-delta-neg}
  \begin{cases}
    \det\pmb{\delta}_+(\lambda)=\det\pmb{\delta}_-(\lambda)(1+|\mathbf{R}(\lambda)|^2),\quad  &\lambda>\xi,\\
    \det\pmb{\delta}(\lambda)\rightarrow1, &\lambda\rightarrow\infty.
  \end{cases}
\end{equation}
Similar to the properties established in Proposition \ref{proposition-delta}, $\det\pmb{\delta}(\lambda)$ 
can be explicitly expressed using the Plemelj formula:
\begin{equation}\label{def-kappa-neg}
  \begin{aligned}
  &\det\pmb{\delta}(\lambda)=\exp\left(\ii\int_{\xi}^{\infty}\frac{\kappa(s)}{s-\lambda} \dd s\right),\\
  &\kappa(s)=-\frac{1}{2\pi}\ln(1+|\mathbf{R}(s)|^2).
  \end{aligned}
\end{equation}
In particular, we arrive at the following crucial local behavior:
as $\lambda\to\xi$ along any ray $\xi+\ee^{\ii \phi}$ with $|\phi|\le c < \pi$,
\begin{equation}\label{boundedness-detdelta-neg}
  \left|\det\pmb{\delta}(\lambda)-T_0(\xi)(\xi-\lambda)^{-\ii\kappa}\right|\lesssim
  \left|\xi-\lambda\right|^{\frac{1}{2}},
\end{equation}
where $\kappa:=\kappa(\xi)=-\frac{1}{2\pi}\ln(1+|\mathbf{R}(\xi)|^2)$,
\begin{equation}
  T_0(\xi)=\ee^{\ii\beta(\xi,\xi)},\quad
  \beta(\lambda,\xi)=\kappa(\xi)\ln(1+\xi-\lambda)+\int_{\xi}^{\infty}
  \frac{\kappa(s)-\chi(s)\kappa(\xi)}{s-\lambda}\dd s,
\end{equation}
and $\chi(s)$ serves as the characteristic function for the interval $(\xi,\xi+1)$.
      
  \begin{figure}[h]
  \centering
  \begin{tikzpicture}
    \fill[gray!30] (0,0) -- (5,0) -- (5,3) -- (3,3) -- cycle;
    \fill[gray!30] (0,0) -- (-5,0) -- (-5,3) -- (-3,3) -- cycle;
    \fill[gray!30] (0,0) -- (5,0) -- (5,-3) -- (3,-3) -- cycle;
    \fill[gray!30] (0,0) -- (-5,0) -- (-5,-3) -- (-3,-3) -- cycle;
    \draw[thick] (-5, 0) -- (5, 0) ;
    \draw[->] (-1.5, 1.5) -- (1.5, -1.5) ;
    \draw[->] (-1.5, -1.5) -- (1.5, 1.5) ;
    \draw[->] (-3, 3) -- (-1.5, 1.5) ;
    \draw[->] (-3, -3) -- (-1.5, -1.5) ;
    \draw[thick] (1.5, 1.5) -- (3, 3) ;
    \draw[thick] (1.5, -1.5) -- (3, -3) ;
    \node at (1, 0.5)  {$\Omega_1$};
    \node at (0, 1)  {$\Omega_2$};
    \node at (-1, 0.5)  {$\Omega_3$};
    \node at (-1, -0.5)  {$\Omega_4$};
    \node at (0, -1)  {$\Omega_5$};
    \node at (0, -0.5)  {$\xi$};
    \node at (1, -0.5)  {$\Omega_6$};
    \node at (3, 2.5)  {$\Sigma_1$};
    \node at (3, -2.5)  {$\Sigma_4$};
    \node at (-3, 2.5)  {$\Sigma_2$};
    \node at (-3, -2.5)  {$\Sigma_3$};
    
    \node at (3, 1)  {$\pmb{\mathcal{R}}^{ (2)}=\begin{pmatrix}1&-R_1\ee^{-2\ii t\theta}\\\mathbf{0}&\mathbb{I}_2\end{pmatrix}$};
    \node at (3, -1)  {$\pmb{\mathcal{R}}^{ (2)}=\begin{pmatrix}1&\mathbf{0}\\R_6\ee^{2\ii t\theta}&\mathbb{I}_2\end{pmatrix}$};
    \node at (-3.2, -1)  {$\pmb{\mathcal{R}}^{ (2)}=\begin{pmatrix}1&R_4\ee^{-2\ii t\theta}\\\mathbf{0}&\mathbb{I}_2\end{pmatrix}$};
    \node at (-3.2, 1)  {$\pmb{\mathcal{R}}^{ (2)}=\begin{pmatrix}1&\mathbf{0}\\-R_3\ee^{2\ii t\theta}&\mathbb{I}_2\end{pmatrix}$};
    \node at (0,2) {$\pmb{\mathcal{R}}^{ (2)}=\mathbb{I}_3$};
    \node at (0,-2) {$\pmb{\mathcal{R}}^{ (2)}=\mathbb{I}_3$};
    
    \fill[white] (-2, 2) circle (3pt);
    \fill[white] (-2, -2) circle (3pt);

    \fill[white] (4, 2.5) circle (3pt);\fill[white] (4, -2.5) circle (3pt);

\end{tikzpicture}
\caption{\small{The contours $\Sigma_i$ and regions $\Omega_k$, for $i=1,\ldots,4$ and 
$k=1,\ldots,6$. 
The $\bar{\partial}$ derivative
$\bar{\partial}\M^{(2)}$ is supported in the gray region, excluding the neighborhoods of the discrete spectrums.}
}
\label{region-extension-neg}
\end{figure}
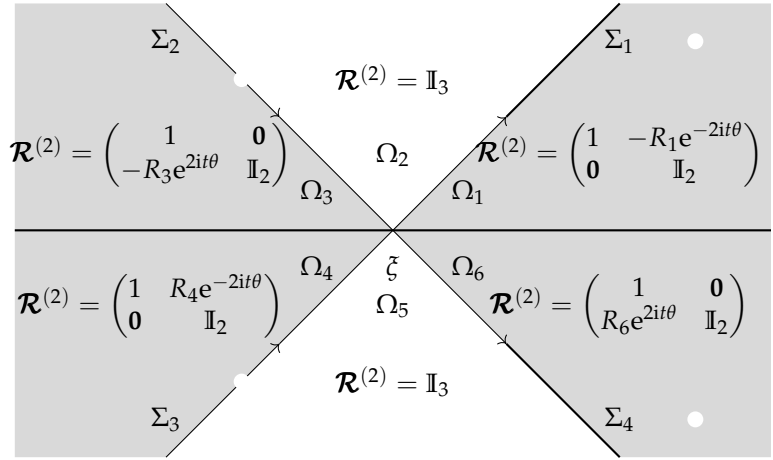

     The second modification pertains to the non-analytic extensions, which are introduced to deform the jump matrix onto the contours where it decays exponentially, as previously described in Section 4. One can then prove, analogously to Lemma \ref{def-extension}, that there exist continuous functions $\pmb{\mathcal{R}}_{j}:\Omega_j\to\mathbb{C}$ for $j = 1,3,4,6$, satisfying the boundary values:
\begin{equation}\label{boundary-R-neg}
\begin{aligned}
    &
    \pmb{\mathcal{R}}_{1}(\lambda)= 
    \begin{cases}
      -(\det\pmb{\delta}_{+})\frac{\mathbf{R}(\lambda)}{1+|\mathbf{R}(\lambda)|^{2}} \pmb{\delta}_+ , \\
      -\frac{\mathbf{R}(\xi)}{1+|\mathbf{R}(\xi)|^{2}}\pmb{\delta}T_{0}(\xi)(\xi-\lambda)^{ -\ii \kappa}
      \left(1-\chi(\lambda)\right),
    \end{cases}
    &
    \begin{aligned}
      & \lambda \in(\xi, +\infty),\\
      & \lambda \in \Sigma_{1},
    \end{aligned}\\
    &
    \pmb{\mathcal{R}}_{3}(\lambda)= 
    \begin{cases}
      -(\det\pmb{\delta})^{-1}\pmb{\delta}^{-1}\mathbf{R}^{\dagger}(\lambda),  \\
      -\pmb{\delta}^{-1}\mathbf{R}^\dagger(\xi) T^{-1}_{0}(\xi)(\xi-\lambda)^{ \ii \kappa}
      \left(1-\chi(\lambda)\right),
    \end{cases} 
    &
    \begin{aligned}
      &\lambda \in(-\infty, \xi),\\
      &\lambda \in \Sigma_{2},
    \end{aligned}\\
    &
    \pmb{\mathcal{R}}_{4}(\lambda)= 
    \begin{cases}
      -\det\pmb{\delta}_{-}\mathbf{R}(\lambda)\pmb{\delta}_{-},  \\
      -\mathbf{R}(\xi)\pmb{\delta} T_{0}(\xi)(\xi-\lambda)^{-\ii \kappa}
      \left(1-\chi(\lambda)\right),
    \end{cases} 
    &
    \begin{aligned}
      & \lambda \in(-\infty, \xi), \\
      & \lambda \in \Sigma_{3},
      \end{aligned}\\
    &
    \pmb{\mathcal{R}}_{6}(\lambda)= 
    \begin{cases}
      -(\det\pmb{\delta}_-)^{-1}\pmb{\delta}^{-1}_{-}\frac{\mathbf{R}^\dagger(\lambda)}{1+|\mathbf{R}(\lambda)|^{2}}, \\
      -\pmb{\delta}^{-1}\frac{\mathbf{R}^\dagger(\xi)}{1+|\mathbf{R}(\xi)|^{2}}T^{-1}_{0}(\xi)(\xi-\lambda)^{\ii \kappa}
      \left(1-\chi(\lambda)\right),
    \end{cases}
    &
    \begin{aligned}
      & \lambda \in(\xi, +\infty), \\
      & \lambda \in \Sigma_{4}.
    \end{aligned}
\end{aligned}
\end{equation}
Next, we construct a new transformation 
$\M^{(2)}(\lambda)=\M^{(1)}(\lambda)\pmb{\mathcal{R}}(\lambda)$ that satisfies the 
following RHP, where the piecewise matrix $\pmb{\mathcal{R}}(\lambda)$ is defined in each sector $\Omega_k$ 
as illustrated in Figure \ref{region-extension-neg}.

\begin{rhp}\label{dbar-rhp-neg}
  ($\bar{\partial}$-RHP) Find a matrix function $\M^{(2)}(\lambda)=\M^{(2)}(\lambda;x,t)$ with the following properties: 
  \begin{enumerate}
    \item Jump condition: $\M^{(2)}(\lambda)$ has continuous boundary 
    values $\M^{(2)}_\pm(\lambda)$ on $\Sigma$,
    \begin{equation}
      \M^{(2)}_\pm(\lambda)=\lim_{\varepsilon \downarrow 0} \M^{(2)}(\lambda\pm\ii\varepsilon), 
    \end{equation}
    satisfying 
    $\M^{(2)}_+(\lambda)=\M^{(2)}_-(\lambda)\V^{(2)}(\lambda)$, where
    \begin{equation}\label{jump-M2-neg}
      \V^{(2)}(\lambda)=
      \begin{cases}
        \begin{pmatrix}
         1 & -\frac{\mathbf{R}(\xi)}{1+|\mathbf{R}(\xi)|^{2}}
         \pmb{\delta}T_{0}(\xi)(\xi-\lambda)^{-\ii \kappa}\ee^{-2\ii t\theta}\left(1-\chi(\lambda)\right)\\
         \mathbf{0} & \mathbb{I}_2
       \end{pmatrix},
       &\lambda\in\Sigma_1,\\
       \begin{pmatrix}
         1 & \mathbf{0}\\
         -\pmb{\delta}^{-1}\mathbf{R}^\dagger(\xi) T^{-1}_{0}(\xi)(\xi-\lambda)^{ \ii \kappa}
         \ee^{2\ii t\theta}\left(1-\chi(\lambda)\right) & \mathbb{I}_2
       \end{pmatrix},
       &\lambda\in\Sigma_2,\\
      \begin{pmatrix}
         1 & -\mathbf{R}(\xi)\pmb{\delta}_{-} T_{0}(\xi)(\xi-\lambda)^{-\ii \kappa}
         \ee^{-2\ii t\theta}\left(1-\chi(\lambda)\right)\\
         \mathbf{0} & \mathbb{I}_2
       \end{pmatrix},
       &\lambda\in\Sigma_3,\\
       \begin{pmatrix}
         1 & \mathbf{0}\\
         -\pmb{\delta}^{-1}_{-}\frac{\mathbf{R}^\dagger(\xi)}
         {1+|\mathbf{R}(\xi)|^{2}}T^{-1}_{0}(\xi)(\xi-\lambda)^{\ii \kappa}
         \ee^{2\ii t\theta}\left(1-\chi(\lambda)\right) & \mathbb{I}_2
       \end{pmatrix},
       &\lambda\in\Sigma_4.
      \end{cases}
     \end{equation}
     \item $\bar{\partial}$ condition: 
     For $\lambda\in\mathbb{C}\setminus\Sigma$, $\M^{(2)}(\lambda)$ possesses nonzero 
     $\bar{\partial}$-derivatives in each corresponding region $\Omega_j$ (for $j=1,3,4,6$), given by:
     \begin{equation}
      \bar{\partial}\M^{(2)}(\lambda)=\M^{(2)}(\lambda)
      \bar{\partial}\mathbf{\pmb{\mathcal{R}}}^{(2)}(\lambda),
     \end{equation}
     where
  \begin{equation}\label{dbar-R-neg}
    \bar{\partial}\mathbf{\pmb{\mathcal{R}}}^{(2)}(\lambda)=
    \begin{cases}
      \begin{pmatrix}1&-\bar{\partial}
        \pmb{\mathcal{R}}_1
        \ee^{-2\ii t\theta}\\
        \mathbf{0}&\mathbb{I}_2\end{pmatrix}, &\lambda\in\Omega_1,\\
      \begin{pmatrix}1&\mathbf{0}\\-\bar{\partial}\pmb{\mathcal{R}}_3\ee^{2\ii t\theta}&\mathbb{I}_2\end{pmatrix}, &\lambda\in\Omega_3,\\
      \begin{pmatrix}1&\bar{\partial}
        \pmb{\mathcal{R}}_4
        \ee^{-2\ii t\theta}\\
        \mathbf{0}&\mathbb{I}_2\end{pmatrix}, &\lambda\in\Omega_4,\\
      \begin{pmatrix}1&\mathbf{0}\\\bar{\partial}\pmb{\mathcal{R}}_6\ee^{2\ii t\theta}&\mathbb{I}_2\end{pmatrix}, &\lambda\in\Omega_6,\\
      0, &\lambda\in\Omega_2\cup\Omega_5.
    \end{cases}
  \end{equation}
  \item Asymptotic condition: $\M^{(2)}(\lambda)=\mathbb{I}_{3}+\mathcal{O}(\lambda^{-1})$ as $\lambda\rightarrow\infty$.
  \end{enumerate} 
\end{rhp}
The third modification concerns the local model $\M^{(5)}(\lambda)$ employed to construct 
$\M^{(3)}(\lambda)$ in equation \eqref{construct-M3}, which satisfies the following RHP:
      
\begin{rhp}\label{solvable-model-neg}
Find a matrix function $\M^{(5)}(\lambda)=\M^{(5)}(\lambda;x,t)$ with the following properties: 
\begin{enumerate}
  \item Analyticity: $\M^{(5)}(\lambda)$ is analytic in 
  $\mathbb{C}\setminus\Sigma$.
  \item Jump condition: $\M^{(5)}(\lambda)$ has continuous boundary 
  values $\M^{(5)}_\pm(\lambda)$ on $\Sigma$,
  \begin{equation}
    \M^{(5)}_\pm(\lambda)=\lim_{\varepsilon \downarrow 0} \M^{(5)}(\lambda\pm\ii\varepsilon), 
  \end{equation}
  satisfying 
  $\M^{(5)}_+(\lambda)=\M^{(5)}_-(\lambda)\V^{(5)}(\lambda)$, where
  \begin{equation}
    \V^{(5)}(\lambda)=
    \begin{cases}
      \begin{pmatrix}
       1 & -\frac{\pmb{A}(\xi)}{1+|\mathbf{R}(\xi)|^{2}}
       T^2_{0}(\xi)(\xi-\lambda)^{ -2\ii \kappa}\ee^{-2\ii t\theta}\\
       \mathbf{0} & \mathbb{I}_2
     \end{pmatrix},
     &\lambda\in\Sigma_1,\\
     \begin{pmatrix}
       1 & \mathbf{0}\\
       -\pmb{A}^\dagger(\xi) T^{-2}_{0}(\xi)(\xi-\lambda)^{ 2\ii \kappa}\ee^{2\ii t\theta} & \mathbb{I}_2
     \end{pmatrix},
     &\lambda\in\Sigma_2,\\
     \begin{pmatrix}
       1 & -\pmb{A}(\xi)T^2_{0}(\xi)(\xi-\lambda)^{-2\ii \kappa}\ee^{-2\ii t\theta}\\
       \mathbf{0} & \mathbb{I}_2
     \end{pmatrix},
     &\lambda\in\Sigma_3,\\
     \begin{pmatrix}
       1 & \mathbf{0}\\
       -\frac{\pmb{A}^\dagger(\xi)}
       {1+|\mathbf{R}(\xi)|^{2}}T^{-2}_{0}(\xi)(\xi-\lambda)^{2\ii \kappa}\ee^{2\ii t\theta} & \mathbb{I}_2
     \end{pmatrix},
     &\lambda\in\Sigma_4,
    \end{cases}
   \end{equation}
   with the vector functions defined as $\pmb{A}(\xi)=\mathbf{R}(\xi)+C_\Gamma\pmb{g}(\xi)$ and 
   $\pmb{g}(s)=\det\pmb{\delta}_+^{-1}\mathbf{R}(\xi)
   (\mathbf{R}^\dagger(s)\mathbf{R}(s)-\mathbf{R}(s)
   \mathbf{R}^\dagger(s)\mathbb{I}_2)\pmb{\delta}_-(s)$.
   \item Asymptotic condition: 
  $\M^{(5)}(\lambda)=\mathbb{I}_{3}+\mathcal{O}(\lambda^{-1})$ as $\lambda\rightarrow\infty$.
\end{enumerate} 
\end{rhp}

To analyze this model, we rescale the variable 
$\lambda$ to $y$ via the transformation
\begin{equation}\label{variable-change-neg}
y=\phi(\lambda)=2\sqrt{-t}(\lambda-\xi).
\end{equation}
Subsequently, RHP \ref{solvable-model-neg} can be transformed into the following RHP:
  
\begin{rhp}
Find a matrix function $\M^{(6)}(y)=\M^{(6)}(y;x,t)$ with the following properties: 
\begin{enumerate}
  \item Analyticity: $\M^{(6)}(y)$ is analytic in 
  $\mathbb{C}\setminus\Sigma^\prime$, where the contour 
  $\Sigma^\prime=\sum_{j=1}^{4}\Sigma^\prime_j$ consists of the rays 
  $\Sigma^\prime_j=\ee^{\frac{(2j-1)\pi\ii}{4}}\mathbb{R}_+$.
  \item Jump condition: $\M^{(6)}(y)$ has continuous boundary 
  values $\M^{(6)}_\pm(y)$ on $\Sigma^\prime$,
  \begin{equation}
    \M^{(6)}_\pm(y)=\lim_{\varepsilon \downarrow 0} \M^{(6)}(y\pm\ii\varepsilon), 
  \end{equation}
  satisfying the jump relation
  $\M^{(6)}_+(y)=\M^{(6)}_-(y)\V^{(6)}(y)$, where
  \begin{equation}
    \V^{(6)}(y)=
    \begin{cases}
      \begin{pmatrix}
       1 & -\frac{\pmb{A}(\xi)}{1+|\mathbf{R}(\xi)|^{2}}
       \eta^{2}\left(-y\right)^{-2\ii\kappa} \ee^{\frac{1}{2}\ii y^2}\\
       \mathbf{0} & \mathbb{I}_2
     \end{pmatrix},
     &y\in\Sigma_1^\prime,\\
     \begin{pmatrix}
       1 & \mathbf{0}\\
       -\pmb{A}^\dagger(\xi) \eta^{-2}\left(-y\right)^{2\ii\kappa} \ee^{-\frac{1}{2}\ii y^2} & \mathbb{I}_2
     \end{pmatrix},
     &y\in\Sigma_2^\prime,\\
     \begin{pmatrix}
       1 & -\pmb{A}(\xi)\eta^{2}\left(-y\right)^{-2\ii\kappa} \ee^{\frac{1}{2}\ii y^2}\\
       \mathbf{0} & \mathbb{I}_2
     \end{pmatrix},
     &y\in\Sigma_3^\prime,\\
     \begin{pmatrix}
       1 & \mathbf{0}\\
       -\frac{\pmb{A}^\dagger(\xi)}
       {1+|\mathbf{R}(\xi)|^{2}}\eta^{-2}\left(-y\right)^{2\ii\kappa} \ee^{-\frac{1}{2}\ii y^2} & \mathbb{I}_2
     \end{pmatrix},
     &y\in\Sigma_4^\prime,
    \end{cases}
   \end{equation}
   and the scalar factor $\eta$ is defined as
   \begin{equation}\label{scalar-eta-neg}
    \eta=(2\sqrt{-t})^{\ii\kappa}\ee^{\ii t\xi^2}T_0.
   \end{equation}
   \item Asymptotic condition: 
  $\M^{(6)}(y)=\mathbb{I}_{3}+\mathcal{O}(y^{-1})$ as $y\rightarrow\infty$.
\end{enumerate} 
\end{rhp}

Analogous to equation \eqref{residue-M6}, the asymptotic expansion of $\M^{(6)}(y)$ 
can be explicitly obtained via the corresponding Weber equation:
\begin{equation}\label{residue-M6-neg}
\begin{aligned}
  &\M^{(6)}(y)=\mathbb{I}_3+\frac{\M^{(6)}_1}{y}+\oo(y^{-2}),
  \quad y\to\infty,\\
  &\M^{(6)}_1=\begin{pmatrix}
    0&-\ii\eta^2\pmb{\beta}_{12}\\
    \ii\eta^{-2}\pmb{\beta}_{21}&\mathbf{0}
  \end{pmatrix},\ \
  \pmb{\beta}_{12}=\frac{\ee^{\frac{\pi\kappa}{2}-\frac{\pi \ii}{4}}\kappa
    \Gamma(-\ii\kappa)}{\sqrt{2\pi}}\pmb{A}(\xi),\ \ \pmb{\beta}_{21}=-\pmb{\beta}_{12}^\dagger,
\end{aligned}
\end{equation}
where $\Gamma(\cdot)$ denotes the classical Gamma function. 
Reverting the variable change \eqref{variable-change-neg} and utilizing equation \eqref{residue-M6-neg}, 
we derive the following two expansions for $\M^{(5)}(\lambda)$:
\begin{equation}\label{expansion-M5-neg}
\begin{aligned}
  &\M^{(5)}(\lambda)=\mathbb{I}_3+\frac{\M^{(5)}_1}{\lambda}+\oo(\lambda^{-2}),
  \quad \M^{(5)}_1=\frac{1}{2\sqrt{-t}}\M^{(6)}_1,\quad \lambda\to\infty,\\
  &\M^{(5)}(\lambda)=\mathbb{I}_3+\frac{\M^{(6)}_1}{2\sqrt{-t}(\lambda-\xi)}+\oo(t^{-1}),
  \quad t\to-\infty.
\end{aligned}
\end{equation}
Consequently, for large negative time $t<0$, we can provide corresponding estimates for the term $\E_1$ and the evaluation of $\E(\lambda)$ at the discrete spectrum 
$\lambda=\lambda_i\in\mathcal{Z}$:
\begin{equation}\label{expression-E1-neg}
  \begin{aligned}
    &\E(\lambda_i)=\mathbb{I}_3+\left(-t\right)^{-1/2}\tilde{\mathbf{P}}_i+\oo(t^{-3/4}),\\
    &\E_1=\frac{1}{2\sqrt{-t}}\begin{pmatrix}
      0&-\ii\eta^2\pmb{\beta}_{12}\\
      \ii\eta^{-2}\pmb{\beta}_{21}&\mathbf{0}
    \end{pmatrix}+\oo(t^{-3/4}),
  \end{aligned}
\end{equation}
where the auxiliary matrix $\tilde{\mathbf{P}}_i$ is explicitly given by
\begin{equation}\label{def-tilde-P-neg}
  \tilde{\mathbf{P}}_i=\frac{1}{2(\lambda_i-\xi)}
  \begin{pmatrix}
    0&-\ii\eta^2\pmb{\beta}_{12}\\
    \ii\eta^{-2}\pmb{\beta}_{21}&\mathbf{0}
  \end{pmatrix}.
\end{equation}
Consequently, we can obtain the long-time asymptotics of $\M(\lambda_m)$ as follows: for $k,m=1,\cdots,n$,
\begin{equation}\label{M-lambda-neg}
  \M(\lambda_m)=\begin{pmatrix}
    \Delta^-_{\xi_k}(\lambda_m) & \mathbf{0}\\
    \mathbf{0} & \left(\pmb{\Delta}_{\xi_k}^-(\lambda_m)\right)^{-1}
  \end{pmatrix}+\oo(t^{-1/2}),\ \ \left|\xi-\xi_k\right|<|t|^{-1}\,\ \text{as}\,\
  t\to-\infty,
\end{equation}
where
\begin{equation}\label{det-delta-lambda1-2-neg}
    \Delta_{\xi_k}^-(\lambda_m)=\exp\left(\frac{1}{2\pi\ii}\int_{\xi_k}^{\infty}
    \frac{\ln\left(1+\left|\mathbf{R}(s)\right|^2\right)}{s-\lambda_m}\dd s\right),
\end{equation}
and the matrix function $\pmb{\Delta}^-_{\xi_k}(\lambda)$ satisfies the following RHP:

\begin{rhp}\label{RHP-Delta-neg}
  Find a matrix function $\pmb{\Delta}^-_{\xi_k}(\lambda)$ with the following properties: 
  \begin{enumerate}
    \item Analyticity: $\pmb{\Delta}^-_{\xi_k}(\lambda)$ is analytic in 
    $\mathbb{C}\setminus(\xi_k,+\infty)$.
    \item Jump condition: $\pmb{\Delta}^-_{\xi_k}(\lambda)$ possesses continuous boundary 
    values $\pmb{\Delta}^-_{\xi_k\pm}(\lambda)$ on the ray $(\xi_k,+\infty)$,
    \begin{equation}
      \pmb{\Delta}^-_{\xi_k\pm}(\lambda)=\lim_{\varepsilon \downarrow 0} \pmb{\Delta}^-_{\xi_k}(\lambda\pm\ii\varepsilon), 
    \end{equation}
    satisfying the jump relation
    \begin{equation}
      \pmb{\Delta}^-_{\xi_k+}(\lambda)=(\mathbb{I}_2+\mathbf{R}^\dagger(\lambda)\mathbf{R}(\lambda))\pmb{\Delta}^-_{\xi_k-}(\lambda).
    \end{equation}
    \item Asymptotic condition: $\pmb{\Delta}^-_{\xi_k}(\lambda)=\mathbb{I}_{2}+\mathcal{O}(\lambda^{-1})$ as $\lambda\rightarrow\infty$.
  \end{enumerate} 
\end{rhp}
Moreover, we arrive at the full long-time asymptotics for $\mathbf{q}(x,t)$:
\begin{equation}
    \mathbf{q}(x,t)
    =-\ii\eta^2\pmb{\beta}_{12}\left(-t\right)^{-1/2}+
    \frac{2}{\det \pmb{G}}\begin{bmatrix}
  \det \pmb{F}_1,\det \pmb{F}_2
  \end{bmatrix}
    +\oo(t^{-3/4}),\ \ \text{as}\ \ t\to-\infty,
\end{equation}
where the matrices $\pmb{G}$ and $\pmb{F}_j$ are provided by equations \eqref{def-G} and \eqref{def-F}, respectively.

Suppose $\left|\xi-\xi_k\right|<|t|^{-1}$ for $1\le k\le n$. Within this region, we can deduce that 
both $\ee^{2\varphi_k}$ and $\ee^{-2\varphi_k}$ remain bounded, and the dominant exponential factor is given by 
$\ee^{2(\varphi_1+\cdots+\varphi_{k-1})-2(\varphi_{k+1}+\cdots+\varphi_n)}$. 
It follows from equation \eqref{reduction-q-n} and the asymptotic behavior of $\pmb{y}_k$,
\begin{equation}\label{asymptotics-yk-neg}
  \pmb{y}_k=\begin{bmatrix}
    \ee^{a_k}b_k+\ee^{-a_k}\oo(t^{-1/2})\\
    \ee^{a_k}\oo(t^{-1/2})+\ee^{-a_k}\pmb{d}_k
  \end{bmatrix},\ \ \text{as}\ \ t\to-\infty,
\end{equation}
that we can derive analogous estimates for the determinants:
\begin{equation}
    \det \pmb{G}=\ee^{2(\varphi_1+\cdots+\varphi_{k-1})-2(\varphi_{k+1}+\cdots+\varphi_n)}\\
    \left(
      \det \pmb{B}_{k+1:n}\det \pmb{A}_{1:k}\ee^{2\varphi_k}+
      \det \pmb{B}_{k:n}\det \pmb{A}_{1:k-1} \ee^{-2\varphi_k}+\oo(t^{-1/2})
    \right),
\end{equation}
and for the cross terms:
\begin{equation}
  \begin{aligned}
    \sum_{j=1}^{n}\sum_{m=1}^{n}\pmb{N}_{m,j}
    \pmb{y}_m^1\pmb{y}_j^{2\dagger}&=
    \sum_{j=k}^{n}\sum_{m=1}^{k}\pmb{N}_{m,j}
    \pmb{y}_m^1\pmb{y}_j^{2\dagger}+
    \ee^{2(\varphi_1+\cdots+\varphi_{k-1})-2(\varphi_{k+1}+\cdots+\varphi_n)}\oo(t^{-1/2})\\
    &=\ee^{2(\varphi_1+\cdots+\varphi_{k-1})-2(\varphi_{k+1}+\cdots+\varphi_n)}
    \left(\ee^{2\ii\psi_k}\pmb{h}_k+\oo(t^{-1/2})\right),
  \end{aligned}
\end{equation}
where the vector $\pmb{h}_k$ and the block matrices are defined as:
\begin{equation}\label{def-h-k-neg}
  \pmb{h}_k=\begin{bmatrix}
      \left|
        \begin{array}{cc}
        0 & \pmb{e}_1^\top\\
        \pmb{d}_{[1]}^* & \pmb{B}_{k:n}
      \end{array}
      \right|,
      \left|
        \begin{array}{cc}
        0 & \pmb{e}_1^\top\\
        \pmb{d}_{[2]}^* & \pmb{B}_{k:n}
      \end{array}
      \right|
    \end{bmatrix}\left|
        \begin{array}{cc}
        \pmb{A}_{1:k} & \pmb{e}_k\\
        \pmb{b} & 0
      \end{array}
      \right|,\ \
      \pmb{A}_{ij}=\frac{b_i^*b_j}{\lambda_j-\lambda_i^*},
  \,\ \pmb{B}_{ij}=\frac{\pmb{d}_i^\dagger\pmb{d}_j}{\lambda_j-\lambda_i^*},
\end{equation}
\begin{equation}
      \pmb{b}=\begin{bmatrix}
        b_1,b_{2},\cdots,b_{k}
      \end{bmatrix},\,\
      \pmb{d}_{[j]}=\begin{bmatrix}
        \pmb{d}_k^{[j]},\pmb{d}_{k+1}^{[j]},\cdots,\pmb{d}_{n}^{[j]}
      \end{bmatrix}^\top.
\end{equation}
Hence, substituting these estimates back yields:
\begin{equation}\label{asy-q-neg}
    \mathbf{q}(x,t)=
    \frac{-2\ee^{2\ii\psi_l}\pmb{h}_k}{\det \pmb{B}_{k+1:n}\det \pmb{A}_{1:k}\ee^{2\varphi_k}+
      \det \pmb{B}_{k:n}\det \pmb{A}_{1:k-1} \ee^{-2\varphi_k}}+\oo(t^{-1/2}).
\end{equation}
Furthermore, it follows from equations \eqref{M-lambda-neg} and \eqref{asymptotics-yk-neg} that
\begin{equation}\label{estimate-b-d-j-neg}
  \begin{pmatrix}
    b_j\\
    \pmb{d}_j
  \end{pmatrix}=\begin{pmatrix}
    \Delta_{\xi_k}^-(\lambda_j)&\mathbf{0}\\
    \mathbf{0}&\left(\pmb{\Delta}^-_{\xi_k}\right)^{-1}(\lambda_j)
  \end{pmatrix}\pmb{v}_j+\oo(t^{-1/2}).
\end{equation}
Then, equation \eqref{asy-q-n-neg-1} follows immediately by combining equations \eqref{asy-q-neg} and \eqref{estimate-b-d-j-neg}. 

Finally, we can utilize Lemma \ref{lem-diag} to rigorously demonstrate the following algebraic identity:
\begin{equation}\label{equal-c-l-neg}
      4\eta_k^2\left|\det \pmb{A}_{1:k}\right|\left|\det\begin{bmatrix}
        \pmb{A}_{1:k} & \pmb{e}_k\\
        \pmb{e}_k^\top & 0
      \end{bmatrix}\right|
      \left|\det \pmb{B}_{k:n}\right|\left|\det\begin{bmatrix}
         0 & \pmb{e}_1^\top\\
        \pmb{e}_1 & \pmb{B}_{k:n}
      \end{bmatrix}\right|
      =\left|\det\begin{bmatrix}
         \pmb{A}_{1:k} & \pmb{e}_k\\
        \pmb{f} & 0
      \end{bmatrix}\right|^2
      \sum\limits_{j=1}^2
      \left|\det\begin{bmatrix}
         0 & \pmb{e}_1^\top\\
        \pmb{g}_{[j]}^* & \pmb{B}_{k:n}
      \end{bmatrix}\right|^2
\end{equation}
in the exact same manner as in Theorem \ref{q-max}. This verifies that the asymptotic state $\mathbf{q}_{sol}^{k-}(x,t)$ is 
indeed a pure single-soliton solution for the CNLS equation \eqref{CNLS}, thereby completing 
the proof of Theorem \ref{asy-q-n-neg}.

\section*{Appendix C}
\setcounter{equation}{0} 
\renewcommand{\theequation}{C.\arabic{equation}}

In this Appendix, we provide a proof that the sets $\mathcal{G}_n$, as defined 
by equation \eqref{def-g-n}, are open for $n\in\{0\}\cup\mathbb{N}$. Our argument follows the spirit of the classical work by Beals and Coifman \cite{beals1984scattering}. 
To this end, we introduce the following auxiliary problem: Find a bounded and absolutely continuous 
matrix function $\mathbf{m}\equiv\mathbf{m}(x,\lambda)$ satisfying the differential equation
\begin{equation}\label{m-ode}
    \mathbf{m}_{x}=-\ii\lambda \mathrm{ad}_{\mathbf{\Lambda}_3}
    \mathbf{m}+\ii \mathbf{Q}\mathbf{m},
\end{equation}
subject to the asymptotic condition
\begin{equation}\label{C-3}
  \mathbf{m}\to\mathbb{I}_3,\ \ x\to-\infty,
\end{equation}
where $\mathrm{ad}_{\mathbf{\Lambda}_3}\mathbf{m}=\begin{bmatrix}\mathbf{\Lambda}_3,\mathbf{m}\end{bmatrix}$ denotes the standard matrix commutator. 
Based on this formulation, we can derive the following lemma (refer to part (a) of 
Theorem A in \cite{beals1984scattering}):

\begin{lemma}\label{lem-small-pot}
For any potential $\mathbf{Q}\in L^1(\mathbb{R})$, 
there exists a bounded discrete set $\mathcal{Z}\subset\mathbb{C}\setminus\mathbb{R}$ 
such that for every $\lambda\in\mathbb{C}\setminus(\mathbb{R}\cup\mathcal{Z})$, 
the boundary value problem \eqref{m-ode}--\eqref{C-3} admits a unique solution 
$\mathbf{m}(\cdot,\lambda)$. Moreover, for each fixed $x\in\mathbb{R}$, 
$\mathbf{m}(x,\cdot)$ is a meromorphic function in $\mathbb{C}\setminus\mathbb{R}$, with 
its poles located precisely at the points of the set $\mathcal{Z}$. Additionally, within the domain $\mathbb{C}\setminus\mathbb{R}$, the solution satisfies
\begin{equation}
  \mathbf{m}(x,\lambda)\to\mathbb{I}_3,\ \ \text{as}\ \ \lambda\to\infty.
\end{equation}
\end{lemma}

Consequently, establishing the openness of the set $\mathcal{G}_n$ is equivalent to 
proving the invariance of the number of simple poles of the function $\mathbf{m}$ 
under a sufficiently small perturbation of 
the underlying potential function $\mathbf{Q}$. More precisely, we must show that there exists an 
$\varepsilon>0$ such that whenever 
\begin{equation}
  \|\tilde{\mathbf{Q}}-\mathbf{Q}\|_{L^1}< \varepsilon,
\end{equation}
the corresponding perturbed solution $\tilde{\mathbf{m}}$ possesses exactly the same number of simple poles as the original solution $\mathbf{m}$.

For $\lambda\in \mathbb{C}\setminus \mathbb{R}$, we let 
$\mathbf{\Pi}_{+}^{\lambda}, \mathbf{\Pi}_{0}^{\lambda}, \mathbf{\Pi}_{-}^{\lambda}$ denote the orthogonal 
projections of the general linear algebra $\mathrm{GL}(3,\mathbb{C})$ onto the positive, zero, and negative 
eigenspaces of the operator $\mathrm{Im}(\lambda) \mathrm{ad}_{\mathbf{\Lambda}_3}$. 
For a given block matrix $\mathbf{A}$, if $\lambda\in \mathbb{C}^{+}$, we have 
\begin{equation}\label{def-prod-Cupper}
    \mathbf{\Pi}_{+}^{\lambda}\mathbf{A}=\begin{pmatrix}
      0 & \mathbf{a}_{12} \\
      \mathbf{0}_{2\times 1} & \mathbf{0}_{2\times 2}
    \end{pmatrix},\quad \mathbf{\Pi}_{0}^{\lambda}\mathbf{A}=\begin{pmatrix}
      a_{11} & \mathbf{0}_{1\times 2} \\
      \mathbf{0}_{2\times 1} & \mathbf{a}_{22}
    \end{pmatrix},\quad \mathbf{\Pi}_{-}^{\lambda}\mathbf{A}=\begin{pmatrix}
      0 & \mathbf{0}_{1\times 2} \\
      \mathbf{a}_{21} & \mathbf{0}_{2\times 2}
    \end{pmatrix},
\end{equation}
and if $\lambda \in \mathbb{C}^{-}$, we have 
\begin{equation}\label{def-prod-Clower}
    \mathbf{\Pi}_{-}^{\lambda}\mathbf{A}=\begin{pmatrix}
      0 & \mathbf{a}_{12} \\
      \mathbf{0}_{2\times 1} & \mathbf{0}_{2\times 2}
    \end{pmatrix},\quad \mathbf{\Pi}_{0}^{\lambda}\mathbf{A}=\begin{pmatrix}
      a_{11} & \mathbf{0}_{1\times 2} \\
      \mathbf{0}_{2\times 1} & \mathbf{a}_{22}
    \end{pmatrix},\quad \mathbf{\Pi}_{+}^{\lambda}\mathbf{A}=\begin{pmatrix}
      0 & \mathbf{0}_{1\times 2} \\
      \mathbf{a}_{21} & \mathbf{0}_{2\times 2}
    \end{pmatrix}.
\end{equation}
For any potential $\mathbf{Q}\in L^{1}(\mathbb{R})$, 
we define the matrix-valued integral operator 
$K_{\lambda,\mathbf{Q}}:L^{\infty} \to L^{\infty}\cap C(\mathbb{R})$, 
which is explicitly given by 
\begin{equation*}
  K_{\lambda,\mathbf{Q}} f (x)=\ii\int_{-\infty}^{x}\ee^{-\ii(x-y)\lambda \mathrm{ad}_{\mathbf{\Lambda}_3}}
  (\mathbf{\Pi}_{0}^{\lambda}+\mathbf{\Pi}_{-}^{\lambda})\mathbf{Q}f(y)\dd y-\ii
  \int_{x}^{+\infty}\ee^{-\ii(x-y)\lambda \mathrm{ad}_{\mathbf{\Lambda}_3}}
  \mathbf{\Pi}_{+}^{\lambda}\mathbf{Q}f(y)\dd y. 
\end{equation*}
By virtue of equations \eqref{def-prod-Cupper} and 
\eqref{def-prod-Clower}, we obtain the following pointwise bounds:
\begin{equation}
  \begin{cases}
    \left|\ee^{-\ii(x-y)\lambda \mathrm{ad}_{\mathbf{\Lambda}_3}}
  \mathbf{\Pi}_{-}^{\lambda}\mathbf{Q}f\right|\le 
  \left|\mathbf{\Pi}_{-}^{\lambda}\mathbf{Q}f\right|,&y\le x,\\
  \left|\ee^{-\ii(x-y)\lambda \mathrm{ad}_{\mathbf{\Lambda}_3}}
  \mathbf{\Pi}_{0}^{\lambda}\mathbf{Q}f\right|= 
  \left|\mathbf{\Pi}_{0}^{\lambda}\mathbf{Q}f\right|,&y\le x,\\
  \left|\ee^{-\ii(x-y)\lambda \mathrm{ad}_{\mathbf{\Lambda}_3}}
  \mathbf{\Pi}_{+}^{\lambda}\mathbf{Q}f\right|\le 
  \left|\mathbf{\Pi}_{+}^{\lambda}\mathbf{Q}f\right|,&y\ge x,
  \end{cases}
\end{equation}
which directly yields the uniform operator norm estimate:
\begin{equation}\label{estimate-op-K}
  \|K_{\lambda,\mathbf{Q}}\|_{op}\leq \|\mathbf{Q}\|_{L^{1}}. 
\end{equation}
Consequently, an application of the Lebesgue dominated convergence theorem yields 
\begin{equation*}
  \lim_{x\to-\infty}K_{\lambda,\mathbf{Q}} f (x)=0, 
\end{equation*}
and a direct calculation verifies that 
the function $K_{\lambda,\mathbf{Q}}\mathbf{m}(x)$ satisfies the differential equation:
\begin{equation}
  (K_{\lambda,\mathbf{Q}}\mathbf{m})_x=
  -\ii\lambda \mathrm{ad}_{\mathbf{\Lambda}_3}
    K_{\lambda,\mathbf{Q}}\mathbf{m}+\ii \mathbf{Q}\mathbf{m}.
\end{equation}
We first consider small potentials, satisfying $\|\mathbf{Q}\|_{L^{1}}<1$. 
Under this condition, together with the operator norm estimate \eqref{estimate-op-K}, the inverse operator 
$(\mathbb{I}-K_{\lambda,\mathbf{Q}})^{-1}$ exists. Consequently, the explicit formula 
\begin{equation}\label{integral-form-m}
  \mathbf{m}(x,\lambda)=
  (\mathbb{I}-K_{\lambda,\mathbf{Q}})^{-1} \mathbb{I}_3
\end{equation}
satisfies both the differential equation \eqref{m-ode} and the asymptotic boundary condition 
\eqref{C-3}. Moreover, we immediately derive 
the following uniform bound for $\mathbf{m}(x,\lambda)$:
\begin{equation}\label{estimate-m}
  \left\|\mathbf{m}(\cdot,\lambda)\right\|_{L^\infty}
  \le \left(1-\|\mathbf{Q}\|_{L^{1}}\right)^{-1}.
\end{equation}
It directly follows from equation \eqref{m-ode} and Liouville's formula that 
\begin{equation}
  \det \mathbf{m}(x,\lambda)\equiv1,
\end{equation}
implying that the inverse matrix $\mathbf{m}^{-1}(x,\lambda)$ exists and is bounded. 
Utilizing the symmetry condition $\mathbf{Q}=\mathbf{Q}^\dagger$, we can further derive the relation 
\begin{equation}
  \mathbf{m}(x,\lambda)\mathbf{m}^\dagger(x,\lambda^*)=\mathbb{I}_3,
\end{equation}
which consequently yields the estimate
\begin{equation}\label{estimate-m-inverse}
  \left\|\mathbf{m}^{-1}(\cdot,\lambda)\right\|_{L^\infty}=
  \left\|\mathbf{m}^\dagger(\cdot,\lambda^*)\right\|_{L^\infty}
  \le \left(1-\|\mathbf{Q}\|_{L^{1}}\right)^{-1}.
\end{equation}

Now, suppose $\tilde{\mathbf{m}}$ is another solution to \eqref{m-ode} 
associated with a perturbed small potential $\tilde{\mathbf{Q}}$. We then obtain the equation for the relative matrix:
\begin{equation}
  \partial_{x}(\mathbf{m}^{-1}\tilde{\mathbf{m}})=
  -\ii\lambda \mathrm{ad}_{\mathbf{\Lambda}_3}(\mathbf{m}^{-1}\tilde{\mathbf{m}})
  +\ii \left[\mathbf{m}^{-1}(\tilde{\mathbf{Q}}-\mathbf{Q})
  \mathbf{m}\right](\mathbf{m}^{-1}\tilde{\mathbf{m}}).
\end{equation}
Analogous to equation \eqref{integral-form-m}, we can 
derive the integral representation for the 
solution $\mathbf{m}^{-1}\tilde{\mathbf{m}}$:
\begin{equation}
  \mathbf{m}^{-1}\tilde{\mathbf{m}}=
  \mathbb{I}_3+K_{\lambda,\ii \mathbf{m}^{-1}
  (\tilde{\mathbf{Q}}-\mathbf{Q})\mathbf{m}}
  (\mathbf{m}^{-1}\tilde{\mathbf{m}}). 
\end{equation}
By invoking the uniform bounds \eqref{estimate-m} and 
\eqref{estimate-m-inverse}, it follows that 
\begin{equation}
  \begin{aligned}
    |\mathbf{m}^{-1}\tilde{\mathbf{m}}-\mathbb{I}_3|\leq&
    \left|\int_{-\infty}^{x}\ee^{-\ii(x-y)\lambda \mathrm{ad}_{\mathbf{\Lambda}_3}}
    (\mathbf{\Pi}_{0}^{\lambda}+\mathbf{\Pi}_{-}^{\lambda})\mathbf{m}^{-1}(\tilde{\mathbf{Q}}-\mathbf{Q})\tilde{\mathbf{m}}(y)\dd y\right|
    \\& +
    \left|\int_{x}^{+\infty}\ee^{-\ii(x-y)\lambda \mathrm{ad}_{\mathbf{\Lambda}_3}}
    \mathbf{\Pi}_{+}^{\lambda}\mathbf{m}^{-1}(\tilde{\mathbf{Q}}-\mathbf{Q})\tilde{\mathbf{m}}(y)\dd y\right|\\
    \leq & \|\mathbf{m}^{-1}\|_{L^{\infty}_{x}}\|
    \tilde{\mathbf{m}}\|_{L^{\infty}_{x}}
    \|\tilde{\mathbf{Q}}-\mathbf{Q}\|_{L^{1}}. 
  \end{aligned}
\end{equation}
Here, the perturbed solution $\tilde{\mathbf{m}}(x,\lambda)$ satisfies an analogous bound 
$\left\|\tilde{\mathbf{m}}(\cdot,\lambda)\right\|_{L^\infty}\le \left(1-\|\tilde{\mathbf{Q}}\|_{L^{1}}\right)^{-1}$.
Hence, the point-wise difference between $\mathbf{m}(x,\lambda)$ and $\tilde{\mathbf{m}}(x,\lambda)$ 
can be directly controlled by the $L^1$-norm of the potential difference $\|\tilde{\mathbf{Q}}-\mathbf{Q}\|_{L^{1}}$:
\begin{equation}\label{estimate-m-by-Q}
  |\tilde{\mathbf{m}}(x,\lambda)-\mathbf{m}(x,\lambda)|=
  |\mathbf{m}(x,\lambda)||\mathbf{m}^{-1}\tilde{\mathbf{m}}
  (x,\lambda)-\mathbb{I}_3|
  \lesssim\|\tilde{\mathbf{Q}}-\mathbf{Q}\|_{L^{1}}.
\end{equation}
This estimate \eqref{estimate-m-by-Q} rigidly holds for 
sufficiently small potentials satisfying $\left\|\mathbf{Q}\right\|_{L^1}<1$. In the 
following lemma, we extend this uniform bound to an arbitrary 
potential $\mathbf{Q}\in L^1(\mathbb{R})$ via an induction argument.
\begin{lemma}\label{lemma-15}
  Let the potential $\mathbf{Q}$ be in $L^1(\mathbb{R})$ and let 
  $\mathbf{m}(x,\lambda)$ be the associated eigenfunction. 
  The function $\mathbf{m}(x,\lambda)$ admits a finite number of poles. 
  Suppose $K$ is a compact 
  subset of $\mathbb{C}\setminus \mathbb{R}$ such that 
  $\mathbf{m}(x,\lambda)$ has no singularities within $K$, and 
  $\mathbf{m}(x,\lambda)$ is continuous on $\mathbb{R}\times K$. 
  Then, for any fixed $\epsilon>0$, there exists a $\delta>0$ such that if 
  a perturbed potential $\tilde{\mathbf{Q}}\in L^{1}(\mathbb{R})$ satisfies 
  $\|\tilde{\mathbf{Q}}-\mathbf{Q}\|_{L^1}< \delta$, 
  then the corresponding 
  eigenfunction $\tilde{\mathbf{m}}(x,\lambda)$ also 
  extends continuously to $\mathbb{R} \times K$ and satisfies
  $|\tilde{\mathbf{m}}-\mathbf{m}|<\epsilon$. 
\end{lemma}

\begin{proof}
  We proceed by induction on the smallest integer $N\ge0$ 
  such that $\|\mathbf{Q}\|_{L^{1}}<2^{N}$. 
  There exists a point $x_{0}$ such that 
  \begin{equation}
    \int_{-\infty}^{x_{0}}|\mathbf{Q}(\xi)|\dd\xi=
  \int_{x_{0}}^{+\infty}|\mathbf{Q}(\xi)|\dd\xi=
  \frac{\|\mathbf{Q}\|_{L^{1}}}{2}\le 2^{N-1}.
  \end{equation} 
  Without loss of generality, we set $x_{0}=0$ via 
  the translation $x\to x+x_{0}$ and define the truncated potentials:
  \begin{equation}
    \mathbf{Q}_{1}(x)=\mathbf{Q}(x)\mathbf{1}_{x\leq 0}(x),\,\
    \mathbf{Q}_{2}(x)=\mathbf{Q}(x)\mathbf{1}_{x\geq 0}(x),
  \end{equation}
  where the indicator function $\mathbf{1}_{x\leq 0}(x)$ equals $1$ if $x\leq 0$ and $0$ 
  otherwise, with $\mathbf{1}_{x\ge 0}(x)$ defined analogously. 
  This partition ensures that $\|\mathbf{Q}_{1}\|_{L^{1}}=\|\mathbf{Q}_{2}\|_{L^{1}}<2^{N-1}$. 
  By the induction hypothesis, there exist well-defined eigenfunctions 
  $\mathbf{m}_{1}$ and $\mathbf{m}_{2}$  
  corresponding to $\mathbf{Q}_{1}$ and 
  $\mathbf{Q}_{2}$, respectively. 
  We now construct the function $\mathbf{m}(x,\lambda)$ as follows:
  \begin{equation}\label{construction-m-induction}
    \mathbf{m}(x,\lambda)=\left\{
      \begin{split}
        \mathbf{m}_{1}(x,\lambda)
        \ee^{-\ii x\lambda 
        \mathrm{ad}_{\mathbf{\Lambda}_{3}}}
        \mathbf{a}_{1}(\lambda), \quad x\leq 0, \\
        \mathbf{m}_{2}(x,\lambda)
        \ee^{-\ii x\lambda 
        \mathrm{ad}_{\mathbf{\Lambda}_{3}}}
        \mathbf{a}_{2}(\lambda), \quad x\geq 0,
      \end{split}
    \right.
  \end{equation}
  where $\mathbf{a}_{1}(\lambda)$ and $\mathbf{a}_{2}(\lambda)$ will be determined below. 
  We will demonstrate that $\mathbf{m}(x,\lambda)$ satisfies the boundary value problem
  \eqref{m-ode}--\eqref{C-3} with potential $\mathbf{Q}$ and fulfills 
  the statement of Lemma \ref{lem-small-pot}. Clearly, 
  $\mathbf{m}(x,\lambda)$ satisfies equation \eqref{m-ode} because 
  both $\mathbf{m}_{1}(x,\lambda)$ and $\mathbf{m}_{2}(x,\lambda)$ 
  satisfy it locally. To ensure $\mathbf{m}(x,\lambda)$ remains bounded, 
  $\mathbf{a}_{1}(\lambda)$ and $\mathbf{a}_{2}(\lambda)$ must satisfy the projection conditions:
  \begin{equation}\label{condition-m-bounded}
    \mathbf{\Pi}_{-}^{\lambda}\mathbf{a}_{1}(\lambda)=0,
    \ \ \mathbf{\Pi}_{+}^{\lambda}\mathbf{a}_{2}(\lambda)=0,
  \end{equation}
  where $\mathbf{\Pi}_{-}^{\lambda}$ and $\mathbf{\Pi}_{+}^{\lambda}$ are defined by equations 
  \eqref{def-prod-Clower} and \eqref{def-prod-Cupper}.
  To enforce the continuity of $\mathbf{m}(x,\lambda)$ at $x=0$, 
  we further require the matching condition:
  \begin{equation}\label{m-a1a2}
    \mathbf{m}_{2}(0,\lambda)^{-1}\mathbf{m}_{1}
    (0,\lambda)=\mathbf{a}_{2}
    (\lambda)\mathbf{a}_{1}^{-1}(\lambda). 
  \end{equation}
  It follows from the asymptotic behavior 
  of $\mathbf{m}_{1}(x,\lambda)$ and equation 
  \eqref{condition-m-bounded} that 
  \begin{equation}
    \mathbf{m}(x,\lambda)\to 
    \mathbf{\Pi}_{0}^{\lambda}\mathbf{a}_{1}(\lambda),
    \ \ \text{as}\,\ x\to-\infty.
  \end{equation}
  By invoking the asymptotic condition \eqref{C-3} for $\mathbf{m}(x,\lambda)$, this implies that 
  \begin{equation}
    \mathbf{\Pi}_{0}^{\lambda}\mathbf{a}_{1}(\lambda)
    =\mathbb{I}_3.
  \end{equation}
  From the definitions of the projections $\mathbf{\Pi}_{+}^{\lambda}$ and 
  $\mathbf{\Pi}_{-}^{\lambda}$, the 
  factorization \eqref{m-a1a2} corresponds to an upper-lower triangular decomposition if $\mathrm{Im}(\lambda)>0$, and to a 
  lower-upper triangular decomposition if $\mathrm{Im}(\lambda)<0$. 
  Consequently, the decomposition \eqref{m-a1a2} is uniquely determined. 
  Expressing $\mathbf{m}_{2}(0,\lambda)^{-1}\mathbf{m}_{1}(0,\lambda)$ in a block form as
  \begin{equation}
    \mathbf{m}_{2}(0,\lambda)^{-1}\mathbf{m}_{1}(0,\lambda)
    =\begin{pmatrix}
      f_{11}&\mathbf{f}_{12}\\
      \mathbf{f}_{21}&\mathbf{f}_{22}
    \end{pmatrix},
  \end{equation}
  the matrices $\mathbf{a}_{1}(\lambda)$ and $\mathbf{a}_{2}(\lambda)$ can be explicitly written as follows:
  \begin{equation}\label{expression-a-12}
    \mathbf{a}_{1}(\lambda)=\begin{pmatrix}
      1 & -f_{11}^{-1}\mathbf{f}_{12} \\
      \mathbf{0}_{2\times 1} & \mathbb{I}_{2}
    \end{pmatrix},\quad 
    \mathbf{a}_{2}(\lambda)=\begin{pmatrix}
      f_{11} & \mathbf{0}_{1\times 2}  \\ 
      \mathbf{f}_{21} & f_{11}^{-1}
      (f_{11}\mathbf{f}_{22}-\mathbf{f}_{21}\mathbf{f}_{12})
    \end{pmatrix}. 
  \end{equation}
  Moreover, it follows from the 
  asymptotic conditions of $\mathbf{m}_1(x,\lambda)$ and 
  $\mathbf{m}_2(x,\lambda)$ as $\lambda\to\infty$ that 
  \begin{equation}
    \mathbf{a}_{1}(\lambda)\to\mathbb{I}_3,\ \ \mathbf{a}_{2}(\lambda)
    \to\mathbb{I}_3,\ \ \text{as}\ \ \lambda\to\infty.
  \end{equation}
  Then $\mathbf{m}(x,\lambda)$ also approaches the identity matrix 
  as $\lambda\to\infty$. Thus, the zeros of $f_{11}(\lambda)$ 
  introduce at most a bounded, discrete set of new 
  singularities for $\mathbf{m}(x,\lambda)$. 
  Now for $z\in K$, there exists a 
  constant $C$ such that $\|f_{11}^{-1}\|_{L^{\infty}}<C$. 
  Then, for $x\le0$, we estimate:
  \begin{equation}
    \begin{aligned}
      |\mathbf{m}(x,\lambda)-\tilde{\mathbf{m}}(x,\lambda)|\leq& 
      |\mathbf{a}_{1}(\lambda)||\mathbf{m}_{1}(x,\lambda)-
      \tilde{\mathbf{m}}_{1}(x,\lambda)|+
      |\mathbf{m}(x,\lambda)||\mathbf{a}_{1}
      (\lambda)-\tilde{\mathbf{a}}_{1}(\lambda)|\\
      &=\mathrm{I}^{(1)}+\mathrm{I}^{(2)}.
    \end{aligned}
  \end{equation}
  By the inductive construction, we obtain
  \begin{equation}
    \mathrm{I}^{(1)}\lesssim\|\tilde{\mathbf{Q}}_{1}-
    \mathbf{Q}_{1}\|_{L^{1}}\le 
    \|\tilde{\mathbf{Q}}-\mathbf{Q}\|_{L^{1}}.
  \end{equation}
  Using equation \eqref{expression-a-12}, we can derive 
  \begin{equation}
     \mathrm{I}^{(2)}\lesssim\|\mathbf{m}\|_{L^{\infty}_{x}}\|f_{11}^{-1}\|_{L^{\infty}}
     (|\mathbf{f}_{12}-\tilde{\mathbf{f}}_{12}|+
     |\tilde{\mathbf{f}}_{12}||f_{11}-
     \tilde{f}_{11}|).
  \end{equation}
  Since the elements $f_{ij}$ are entirely determined by 
  $\mathbf{m}_{2}(0,\lambda)^{-1}\mathbf{m}_{1}(0,\lambda)$, 
  the map defining $f_{ij}(\mathbf{m}_{1},\mathbf{m}_{2})$ is 
  continuous. This implies that $\mathrm{I}^{(2)}$ can 
  also be bounded by $\|\tilde{\mathbf{Q}}-\mathbf{Q}\|_{L^{1}}$. 
  Hence, for a sufficiently small $\delta$, 
  the estimate $|\mathbf{m}(x,\lambda)-
  \tilde{\mathbf{m}}(x,\lambda)|<\epsilon$ holds for $x\le0$. 
  A similar argument can also be applied for the remaining case $x\geq 0$. 
  This completes the proof. 
\end{proof}
\begin{lemma}\label{same-pole-small-potential}
  Let the potentials $\mathbf{Q}$ and $\tilde{\mathbf{Q}}$ 
  belong to $L^1(\mathbb{R})$, with corresponding eigenfunctions
  $\mathbf{m}(x,\lambda)$ and $\tilde{\mathbf{m}}(x,\lambda)$, 
  respectively. There exists a sufficiently small $\varepsilon>0$ such that if 
  $\|\tilde{\mathbf{Q}}-\mathbf{Q}\|_{L^1}<\varepsilon$, then 
  the singularities of $\tilde{\mathbf{m}}(x,\lambda)$ in the $\lambda$-plane lie  
  close to those of 
  $\mathbf{m}(x,\lambda)$, and the total number of simple poles remains unchanged.
\end{lemma}
\begin{proof}
  For a small potential $\mathbf{Q}$ satisfying 
  $\|\mathbf{Q}\|_{L^1}<1$, equation \eqref{integral-form-m} and 
  the holomorphy of the operator $K_{\lambda,\mathbf{Q}}$ imply that 
  $\mathbf{m}(x,\lambda)$ is holomorphic in $\mathbb{C}\setminus\mathbb{R}$. 
  Consequently, Lemma \ref{same-pole-small-potential} holds for 
  small potentials. The general case then follows via an inductive 
  argument analogous to that in Lemma \ref{lemma-15}. According to the construction 
  of $\mathbf{m}(x,\lambda)$ and equation \eqref{expression-a-12}, it suffices 
  to study the zeros of $f_{11}(\lambda)$, which give rise to 
  the new singularities of $\mathbf{m}(x,\lambda)$. 
  
  Suppose $\lambda_{i}$ is a simple zero of $f_{11}(\lambda)$. Choose 
  a sufficiently small $A>0$ so that $f_{11}(\lambda)$ has exactly one zero  
  in the ball $B(\lambda_i,A)$. 
  By the argument principle, we obtain 
  \begin{equation}
    \frac{1}{2\pi\ii}\oint_{\partial 
    B(\lambda_i,A)}\frac{f_{11,\xi}(\xi)}
    {f_{11}(\xi)}\dd\xi =1. 
  \end{equation}
  By utilizing Lemma \ref{lemma-15}, we can derive the estimate:
  \begin{equation}
    \begin{aligned}
      \left|\frac{1}{2\pi \ii}\oint_{\partial B(\lambda_i,A)}
      \frac{\tilde{f}_{11,\xi}(\xi)}{\tilde{f}_{11}(\xi)}
      \dd\xi-\frac{1}{2\pi \ii}\oint_{\partial B(\lambda_i,A)}
      \frac{f_{11,\xi}(\xi)}{f_{11}(\xi)}\dd\xi
      \right|
      &\leq \frac{1}{2\pi}\oint_{\partial B(\lambda_i,A)}\left|
        \frac{\tilde{f}_{11,\xi}(\xi)}
        {\tilde{f}_{11}(\xi)}-\frac{f_{11,\xi}(\xi)}
        {f_{11}(\xi)}
      \right|\dd\xi
      \\ &\leq C(A)\sup_{\lambda\in \partial B(\lambda_i,A)}
      |\mathbf{m}(x,\lambda)-\tilde{\mathbf{m}}(x,\lambda)|\\
      &\leq C(A)\varepsilon.  
    \end{aligned}
  \end{equation}
  For a sufficiently small $\varepsilon$ such that $C(A)\varepsilon<1$, we necessarily have
  \begin{equation}\label{pole-ft}
    \frac{1}{2\pi\ii}\oint_{\partial B(\lambda_i,A)}
    \frac{\tilde{f}_{11,\xi}(\xi)}{\tilde{f}_{11}(\xi)}\dd\xi =1, 
  \end{equation}
  since the integral in \eqref{pole-ft} must evaluate to an integer.
  This implies that the new simple poles of 
  $\tilde{\mathbf{m}}(x,\lambda)$ lie close to those 
  of $\mathbf{m}(x,\lambda)$ and 
  that the number of simple poles remains unchanged.
\end{proof}

  \section*{Acknowledgements}
 Liming Ling is supported by the National Natural Science Foundation of China (Grant No. 12471236), Guangzhou Science and Technology Plan (Grant No. 2024A04J6245), and
 Guangdong Natural Science Foundation (Grant No. 2025A1515011868).

  \section*{Conflict of interest}
  The authors declare that they have no conflict of interest with other people or organizations that may inappropriately influence the author's work.
  
	\bibliographystyle{unsrt}
	
	\bibliography{reference-dbar}

@article{deift1994long,
author = {Deift, P.A. and Zhou, X.},
title = {{Long-time asymptotics for integrable systems. Higher order theory}},
volume = {165},
journal = {Communications in Mathematical Physics},
number = {1},
publisher = {Springer},
pages = {175--191},
year = {1994},
}

@article{deift2002long,
  title={{Long-time asymptotics for solutions of the NLS equation with initial data in a weighted Sobolev space}},
  author={Deift, P.A. and Zhou, X.},
  journal={Communications on Pure and Applied Mathematics},
  year={2003},
  volume={56},
  number={8},
  pages={1029--1077},
  publisher={Wiley Online Library}
}

@article{deift2011long,
  title={{Long-time asymptotics for solutions of the NLS equation with a delta potential and even initial data}},
  author={Deift, P.A. and Park, J.},
  journal={International Mathematics Research Notices},
  volume={2011},
  number={24},
  pages={5505--5624},
  year={2011},
  publisher={OUP}
}

@article{manakov1974theory,
  title={{On the theory of two-dimensional stationary self-focusing of electromagnetic waves}},
  author={Manakov, S.V.},
  journal={Soviet Physics-JETP},
  volume={38},
  number={2},
  pages={248--253},
  year={1974},
  publisher={AIP}
}

@article{deift1993steepest,
  title={{A steepest descent method for oscillatory Riemann--Hilbert problems. Asymptotics for the mKdV equation}},
  author={Deift, P.A. and Zhou, X.},
  journal={Annals of Mathematics},
  year={1993},
  volume = {137},
  number = {2},
  pages = {295--368},
  publisher = {Princeton University Press},
}

@article{Liu_2019,
year = {2019},
publisher = {IOP},
volume = {32},
number = {3},
pages = {1012--1043},
author = {Liu, J.Q.},
title = {{$L^2$-Sobolev space bijectivity of the inverse scattering of a $3\times3$ AKNS system}},
journal = {Nonlinearity}
}

@book{ablowitz2003complex,
  title={Complex variables: introduction and applications},
  author={Ablowitz, M.J. and Fokas, A.S.},
  year={2003},
  publisher = {McGraw-Hill}
}

@article{beals1984scattering,
  title={Scattering and inverse scattering for first order systems},
  author={Beals, R. and Coifman, R.R.},
  journal={Communications on Pure and Applied Mathematics},
  volume={37},
  number={1},
  pages={39--90},
  year={1984},
  publisher={Wiley Online Library}
}

@article{Guo2011,
	author = {Guo, B.L. and Ling, L.M.},
	title = {{Rogue wave, breathers and bright-dark-rogue solutions for the coupled Schrödinger equations}},
	year = {2011},
	journal = {Chinese Physics Letters},
	volume = {28},
	number = {11},
	publisher={Elsevier B.V.},
  pages={110202}
}

@article{Mu2015,
author = {Mu, G. and Qin, Z.Y. and Grimshaw, R.},
title = {{Dynamics of rogue waves on a multisoliton 
background in a vector nonlinear Schrödinger equation}},
journal = {SIAM Journal on Applied Mathematics},
volume = {75},
number = {1},
pages = {1-20},
year = {2015},
publisher={Springer}
}

@book{deift1994casestudy,
  title={{Long-time behavior of the focusing nonlinear Schr\"odinger 
  equation--a case study}},
  author={Deift, P.A. and Zhou, X.},
  year={1994},
  publisher={University of Tokyo}
}

@article{McLaughlin_2018,
  author = {Borghese, M. and Jenkins, R. and McLaughlin, K.D.T.-R.},
  title = {{Long time asymptotic behavior of the focusing nonlinear Schr\"odinger equation}},
  journal = {Annales de l'Institut Henri Poincaré. C, Analyse non linéaire},
  volume = {35},
  number = {4},
  pages = {887--920},
  year = {2018},
  publisher = {Elsevier},
  doi = {10.1016/j.anihpc.2017.08.006}
}

@article{Geng_Liu_2017,
  author = {Geng, X.G. and Liu, H.},
  title = {{The nonlinear steepest descent method to long-time asymptotics of the coupled nonlinear Schrödinger equation}},
  journal = {Journal of Nonlinear Science},
  volume = {27},
  number = {2},
  pages = {739--763},
  year = {2017},
  publisher = {Springer},
  doi = {10.1007/s00332-017-9426-x}
}

@article{Fanengui2024Ablowitz-Ladik,
title = {{Long-time asymptotics for the defocusing Ablowitz-Ladik system with initial data in lower regularity}},
journal = {Advances in Mathematics},
volume = {450},
pages = {109769},
year = {2024},
issn = {0001-8708},
doi = {https://doi.org/10.1016/j.aim.2024.109769},
url = {https://www.sciencedirect.com/science/article/pii/S0001870824002846},
author = {Chen, M.S. and He, J.S. and Fan, E.G.},
}

@article{Fan2024Hunter-Saxton,
title = {{Long-time asymptotics of the Hunter-Saxton equation on the line}},
journal = {Journal of Differential Equations},
volume = {390},
pages = {451-493},
year = {2024},
issn = {0022-0396},
doi = {https://doi.org/10.1016/j.jde.2024.02.012},
url = {https://www.sciencedirect.com/science/article/pii/S0022039624000834},
author = {Ju, L.M. and Xu, K. and Fan, E.G.},
}

@article{Fan2024Camassa-Holm,
title = {{The Cauchy problem of the Camassa-Holm equation in a weighted Sobolev space: Long-time and Painlevé asymptotics}},
journal = {Journal of Differential Equations},
volume = {380},
pages = {24-91},
year = {2024},
issn = {0022-0396},
doi = {https://doi.org/10.1016/j.jde.2023.10.019},
url = {https://www.sciencedirect.com/science/article/pii/S0022039623006629},
author = {Xu, K. and Yang, Y.L. and Fan, E.G.},
}

@article{Fan2023mKdV,
title = {{On the Cauchy problem of defocusing mKdV equation with finite density 
initial data: Long time asymptotics in soliton-less regions}},
journal = {Journal of Differential Equations},
volume = {372},
pages = {55-122},
year = {2023},
issn = {0022-0396},
doi = {https://doi.org/10.1016/j.jde.2023.06.038},
url = {https://www.sciencedirect.com/science/article/pii/S002203962300445X},
author = {Xu, T.Y. and Zhang, Z.C. and Fan, E.G.},
}

@article{Fan2023Novikov,
title = {{Soliton resolution and large time behavior of solutions to the Cauchy problem for 
the Novikov equation with a nonzero background}},
journal = {Advances in Mathematics},
volume = {426},
pages = {109088},
year = {2023},
issn = {0001-8708},
doi = {https://doi.org/10.1016/j.aim.2023.109088},
url = {https://www.sciencedirect.com/science/article/pii/S0001870823002311},
author = {Yang, Y.L. and  Fan, E.G.},
}

@article{Fan2023nonload-mKdV,
  title={{Long time asymptotic behavior for the nonlocal mKdV equation in solitonic space--time regions}},
  author={Zhou, X. and  Fan, E.G.},
  journal={Mathematical Physics, Analysis and Geometry},
  volume={26},
  number={1},
  pages={3},
  year={2023},
  publisher={Springer}
}

@article{Fan2022DNLS,
  title={{Long-time asymptotic behavior for the derivative Schr{\"o}dinger equation with finite density type initial data}},
  author={Yang, Y.L. and Fan, E.G.},
  journal={Chinese Annals of Mathematics, Series B},
  volume={43},
  number={6},
  pages={893--948},
  year={2022},
  publisher={Springer}
}

@article{Fan2022Sasa-Satsuma,
title = {{Long time and Painlevé-type asymptotics for the Sasa-Satsuma equation in solitonic space time regions}},
journal = {Journal of Differential Equations},
volume = {329},
pages = {89-130},
year = {2022},
issn = {0022-0396},
doi = {https://doi.org/10.1016/j.jde.2022.05.006},
url = {https://www.sciencedirect.com/science/article/pii/S0022039622003059},
author = {Xun, W.K. and Fan, E.G.},
}

@article{GengLiu2024DNLS,
url = {https://doi.org/10.1515/ans-2023-0145},
title = {{Long-time asymptotic behavior for the Hermitian symmetric space derivative nonlinear Schrödinger equation}},
author = {Chen, M.M. and Geng, X.G. and Liu, H.},
pages = {819--856},
volume = {24},
number = {4},
journal = {Advanced Nonlinear Studies},
doi = {doi:10.1515/ans-2023-0145},
year = {2024},
lastchecked = {2024-12-15}
}

@article{Geng2024short-pulse,
title = {Long-time asymptotics for the coupled modified complex short-pulse equation},
journal = {Communications on Pure and Applied Analysis},
volume = {23},
number = {4},
pages = {507-545},
year = {2024},
issn = {1534-0392},
doi = {10.3934/cpaa.2024023},
url = {https://www.aimsciences.org/article/id/660537f7d8a66b482f68303b},
author = {Liu, W.H. and Geng, X.G. and Liu, H.},
}

@article{Geng2024short-pulse2,
title = {Long-time asymptotics for the coupled complex short-pulse equation with decaying initial data},
journal = {Journal of Differential Equations},
volume = {386},
pages = {113-163},
year = {2024},
issn = {0022-0396},
doi = {https://doi.org/10.1016/j.jde.2023.12.019},
url = {https://www.sciencedirect.com/science/article/pii/S0022039623007957},
author = {Geng, X.G. and Liu, W.H. and Li, R.M.},
}

@article{Geng2022CNLS,
  title={{Spectral analysis and long-time asymptotics of a coupled nonlinear Schr{\"o}dinger System}},
  author={Wang, K.D. and Geng, X.G. and Chen, M.M. and Li, R.M.},
  journal={Bulletin of the Malaysian Mathematical Sciences Society},
  volume={45},
  number={5},
  pages={2071--2106},
  year={2022},
  publisher={Springer}
}

@article{lingzhang2024,
title = {{Large and infinite-order solitons of the coupled nonlinear Schrödinger equation}},
journal = {Physica D. Nonlinear Phenomena},
volume = {457},
pages = {133981},
year = {2024},
author = {Ling, L.M. and Zhang, X.E.},
publisher={Elsevier B.V.}
}

@article{chen2021soliton,
  title={{Soliton resolution for the focusing modified KdV equation}},
  author={Chen, G. and Liu, J.Q.},
  journal={Annales de l'Institut Henri Poincar{\'e} C, Analyse non lin{\'e}aire},
  volume={38},
  number={6},
  pages={2005--2071},
  year={2021},
  publisher={Elsevier}
}

@article{DING2020109580,
title = {{Lax pair, conservation laws, Darboux transformation, breathers and rogue waves for the coupled 
nonautonomous nonlinear Schr\"odinger system in an inhomogeneous plasma}},
journal = {{Chaos, Solitons $\&$ Fractals}},
volume = {133},
pages = {109580},
year = {2020},
issn = {0960-0779},
author = {Ding, C.C. and Gao, Y.T. and Deng, G.F. and Wang, D.}
}

@article{SONG2020105046,
title = {{Dynamics of higher-order localized waves for a coupled nonlinear Schr\"odinger equation}},
journal = {Communications in Nonlinear Science and Numerical Simulation},
volume = {82},
pages = {105046},
year = {2020},
issn = {1007-5704},
author = {Song, N. and Xue, H. and Xue, Y.K.}
}

@article{Zhou1998,
author = {Zhou, X.},
title = {{$L^2$-Sobolev space bijectivity of the scattering and inverse scattering transforms}},
journal = {Communications on Pure and Applied Mathematics},
volume = {51},
number = {7},
pages = {697-731},
year = {1998}
}

@book{beals1988direct,
  title={Direct and inverse scattering on the line},
  author={Beals, R. and Deift, P.A. and Tomei, C.},
  year={1988},
  publisher={American Mathematical Society}
}

@article{HYBLLMZXE2025,
title = {{Long-time asymptotics of the coupled nonlinear Schrödinger equation in a weighted Sobolev space}},
journal = {Physica D: Nonlinear Phenomena},
volume = {489},
pages = {135138},
year = {2026},
author = {Huang, Y.B. and Ling, L.M. and Zhang, X.E.},
}

@article{Exact_theory_Zakharov_1972,
author = {Zakharov, V. E. and Shabat, A. B. },
title = {Exact theory of two-dimensional self-focusing and one-dimensional self-modulation of waves in nonlinear media},
journal = {Soviet Physics JETP},
volume = {34},
number = {1},
pages = {62-69},
year = {1972},
}

@article {Yang_Baxter_maps,
    AUTHOR = {Veselov, A. P.},
     TITLE = {Yang-{B}axter maps and integrable dynamics},
   JOURNAL = {Physics Letters. A},
    VOLUME = {314},
      YEAR = {2003},
    NUMBER = {3},
     PAGES = {214--221},
}

@article {Soliton_interactions,
    AUTHOR = {Ablowitz, M. J. and Prinari, B. and Trubatch, A. D.},
     TITLE = {Soliton interactions in the vector {NLS} equation},
   JOURNAL = {Inverse Problems},
    VOLUME = {20},
      YEAR = {2004},
    NUMBER = {4},
     PAGES = {1217--1237},
}

@article {Nsoliton_collision,
    AUTHOR = {Tsuchida, T.},
     TITLE = {{$N$}-soliton collision in the {M}anakov model},
   JOURNAL = {Progress of Theoretical Physics},
    VOLUME = {111},
      YEAR = {2004},
    NUMBER = {2},
     PAGES = {151--182},
}

@article {Yang-Baxter-reflection,
    AUTHOR = {Caudrelier, V. and Zhang, Q. C.},
     TITLE = {Yang-{B}axter and reflection maps from vector solitons with a
              boundary},
   JOURNAL = {Nonlinearity},
    VOLUME = {27},
      YEAR = {2014},
    NUMBER = {6},
     PAGES = {1081--1103},
}

@article {stability-n-soliton--nls-nonzero,
    AUTHOR = {Cuccagna, S. and Jenkins, R.},
     TITLE = {{On the asymptotic stability of N-soliton solutions of the defocusing nonlinear Schr\"odinger equation}},
   JOURNAL = {Communications in Mathematical Physics},
    VOLUME = {343},
      YEAR = {2016},
    NUMBER = {3},
     PAGES = {921--969},
}

@article{stability-one-soliton-nls,
  title={{The asymptotic stability of solitons in the cubic NLS equation on the line}},
  author={Cuccagna, S. and Pelinovsky, D. E.},
  journal={Applicable Analysis},
  volume={93},
  number={4},
  pages={791--822},
  year={2014},
}

@article {stability-n-soliton-nls,
    AUTHOR = {Saalmann, A.},
     TITLE = {{Asymptotic stability of N-solitons in the cubic NLS equation}},
   JOURNAL = {Journal of Hyperbolic Differential Equations},
    VOLUME = {14},
      YEAR = {2017},
    NUMBER = {3},
     PAGES = {455--485},
}

@article{Energy-exchange-interactions,
  title={Energy-exchange interactions between colliding vector solitons},
  author={Anastassiou, C. and Segev, M. and Steiglitz, K. and Giordmaine, J. and Mitchell, M. and 
  Shih, M. and Lan, S. and Martin, J.},
  journal={Physical Review Letters},
  volume={83},
  number={12},
  pages={2332},
  year={1999},
  publisher={American Physical Society}
}

@article{yang1967some,
  title={Some exact results for the many-body problem in one dimension with repulsive delta-function interaction},
  author={Yang, C.N.},
  journal={Physical Review Letters},
  volume={19},
  number={23},
  pages={1312},
  year={1967},
  publisher={APS}
}

@article{drinfel1985hopf,
  title={{Hopf algebras and the quantum Yang-Baxter equation}},
  author={Drinfel'd, V.G.},
  journal={Soviet Mathematics. Doklady},
  volume={32},
  pages={254--258},
  year={1985}
}

@book{korepin1997quantum,
  title={Quantum inverse scattering method and correlation functions},
  author={Korepin, V. E. and Bogoliubov, N.M. and Izergin, A.G.},
  volume={3},
  year={1997},
  publisher={Cambridge university press}
}

@article{yanzhengya2025nNLS,
author = {Lin, Z.J. and Yan, Z.Y.},
title = {{Long-time asymptotics of solutions for any 
$n$-component focusing nonlinear Schrödinger 
equations in spacetime solitonic regions}},
journal = {Studies in Applied Mathematics},
volume = {155},
number = {2},
pages = {e70093},
year = {2025}
}

@article{LingSu2026orbitalstability,
  title={{Orbital stability of breathers 
  for coupled nonlinear Schr{\"o}dinger equations}},
  author={Ling, L.M. and Pelinovsky, D.E. and Su, H.J.},
  journal={Advances in Mathematics},
  volume={495},
  pages={110961},
  year={2026},
  publisher={Elsevier}
}

@article{pelinovsky_inertia_2005,
	title = {{Inertia law for spectral stability of solitary 
  waves in coupled nonlinear {Schr\"odinger} equations}},
	volume = {461},
	number = {2055},
	journal = {Proceedings of the Royal Society A: 
  Mathematical, Physical and Engineering Sciences},
	author = {Pelinovsky, D.E.},
	year = {2005},
	pages = {783--812},
	}

@article{pelinovsky_instabilities_2005,
	title = {{Instabilities of multihump vector solitons 
  in coupled nonlinear {Schr\"odinger} equations}},
	volume = {115},
	issn = {0022-2526, 1467-9590},
	number = {1},
	journal = {Studies in Applied Mathematics},
	author = {Pelinovsky, D.E. and Yang, J.K.},
	year = {2005},
	pages = {109--137},
	}

@article{mesentsev1992stability,
  title={{Stability of vector solitons in optical fibers}},
  author={Mesentsev, V.K. and Turitsyn, S.K.},
  journal={Optics letters},
  volume={17},
  number={21},
  pages={1497--1499},
  year={1992},
  publisher={Optica Publishing Group}
}

@article{li_mechanism_2000,
	title = {The mechanism of the polarizational mode 
  instability in birefringent fiber optics},
	volume = {31},
	issn = {0036-1410, 1095-7154},
	number = {6},
	journal = {SIAM Journal on Mathematical Analysis},
	author = {Li, Y.A. and Promislow, K.},
	year = {2000},
	pages = {1351--1373},
	}

@article{li_structural_1998,
	title = {{Structural stability of non-ground state 
  traveling waves of 
  coupled nonlinear {Schr\"odinger} equations}},
	volume = {124},
	issn = {0167-2789},
	number = {1},
	journal = {Physica D: Nonlinear Phenomena},
	author = {Li, Y.A. and Promislow, K.},
	year = {1998},
	pages = {137--165}
	}

@article{ling2025nonlinear,
  title={{Nonlinear stability of vector multi-solitons in coupled NLS and modified KdV equations}},
  author={Ling, L.M. and Su, H.J.},
  journal={arXiv preprint arXiv:2510.12129},
  year={2025}
}
	
\end{document}